%% file: master_final_2.tex
\documentclass[12pt,leqno]{amsart}
\usepackage{amssymb}
\usepackage{mathrsfs}
\usepackage{graphicx,color} 
\usepackage[varg]{newtxmath}
\usepackage[shortlabels]{enumitem}
\usepackage{comment}
\usepackage{constants}
\newcommand{\parenthezises}[1]{\arabic{#1}}
\newconstantfamily{c}{symbol=c,format=\parenthezises}

\newconstantfamily{e}{symbol=\varepsilon,format=\parenthezises}

\newconstantfamily{m}{symbol=m,format=\parenthezises}

\newconstantfamily{l}{symbol=\lambda,format=\parenthezises}

\newconstantfamily{k}{symbol=k,format=\parenthezises}

\theoremstyle{plain} 
\newtheorem{theorem}{Theorem}[section]

\newtheorem{proposition}[theorem]{Proposition}
\newtheorem{lemma}[theorem]{Lemma}
\newtheorem{thm}[theorem]{Theorem} 
\newtheorem{cor}[theorem]{Corollary} 
\newtheorem{prop}[theorem]{Proposition} 
\newtheorem{lem}[theorem]{Lemma} 

\theoremstyle{definition} 
\newtheorem{definition}[theorem]{Definition}
\newtheorem{remark}[theorem]{Remark}

\newcommand{\J}{{(j)}}
\newcommand{\JJ}{{(j+1)}}
\newcommand{\JJJ}{{(j+2)}}
\newcommand{\I}{{(i)}}
\newcommand{\K}{{(k)}}
\newcommand{\KK}{{(k+1)}}
\newcommand{\KKK}{{(k+2)}}
\newcommand{\pt}{\partial_t}
\newcommand{\px}{\partial_x}
\newcommand{\pa}{\partial_\alpha}
\newcommand{\ps}{\partial_s}
\newcommand{\N}{\mathbb{N}}

\newcommand{\R}{\mathbb{R}}

\usepackage{constants}
\newconstantfamily{eps}{symbol=\varepsilon}
\numberwithin{equation}{section}
\usepackage{bm}
\usepackage{hyperref}
\allowdisplaybreaks[4]

\title[Curvature flow of networks with misorientations]{Long time behavior for a curvature flow of networks related to grain boundary motion with the effect of lattice misorientations}

\author{Takashi Kagaya}
\address[Takashi Kagaya]%
{Graduate school of Engineering, Muroran Institute of Thechnology, 27-1 Mizumoto-cho, Muroran-shi, Hokkaido 050-8585, JAPAN}
\email{kagaya@mmm.muroran-it.ac.jp}

\author{Masashi Mizuno}
\address[Masashi Mizuno]%
{Department of Mathematics, College of Science and Technology, Nihon
University, 1-8-14 Kanda-Surugadai, Chiyoda-Ku, Tokyo 101-8308, JAPAN}
\email{mizuno.masashi@nihon-u.ac.jp}

\author{Keisuke Takasao}
\address[Keisuke Takasao]%
{Department of Mathematics/Hakubi Center, Kyoto University, Kitashirakawa-Oiwakecho Sakyo, Kyoto 606-8502, JAPAN}
\email{k.takasao@math.kyoto-u.ac.jp}



\allowdisplaybreaks[4]
\begin{document}

\begin{abstract}
The mathematical model of grain boundary motion, including lattice misorientations' effect, is considered. When time-dependent lattice misorientations are state variables of the surface tension of the grain boundary, to ensure the energy dissipation law, one can obtain a  curvature flow of networks with time-dependent mobilities. This paper studies the solvability and long-time asymptotic behavior of the curvature flow subjected to the Herring condition which ensures that the constituent grain boundary surface tensions are balanced at the triple junction.
\end{abstract}

\maketitle

 \section{Introduction}
 
\input{introduction_final2.tex}

 \section*{Acknowledgment}

The first author acknowledges partial support of JSPS KAKENHI Grant No.\ JP19K14572, JP20H01801 and JP21H00990. 
The second author acknowledges partial support of JSPS KAKENHI Grant No.\ JP18K13446.
The third author acknowledges partial support of JSPS KAKENHI Grant No.\ JP20K14343, and JSPS Leading Initiative for Excellent Young Researchers (LEADER) operated by Funds for the Development of Human Resources in Science and Technology. 
The first and third author acknowledge partial support of JSPS KAKENHI Grant No.\ JP18H03670. 

\section{Local existence theory in a smooth setting} \label{sec:local-exists}

In this section, we prove the local existence theory for the geometric flow in a smooth setting. 
We first introduce an angle function of $\Gamma_t^\J$ and re-formulate the problem to apply a classical theory for parabolic partial differential equations. 
We note that our re-formulation is different from it in \cite{BR, MR2075985} as we mentioned in the introduction. 

\subsection{Parametrization of $\Gamma_t^\J$ and re-formulation of the problem}\label{sec:re-formulation}

Let the curve $\Gamma^\J_t$ be parametrized by a smooth map $\xi^\J=\xi^\J(x,t): [0,1] \times [0,T) \to \mathbb{R}^2$ and define an angle function $\Theta^\J=\Theta^\J(x,t) : [0,1] \times [0,T)$ by
\begin{equation}\label{def-theta} 
(\tau^\J_t = )\frac{\partial_x \xi^\J}{|\partial_x \xi^\J|} = 
\begin{pmatrix}
\cos \Theta^\J \\
\sin \Theta^\J
\end{pmatrix} 
\end{equation} 
to re-formulate \eqref{eq-curve} and \eqref{bc-concurrency}--\eqref{bc-boundary}. 
Here, the angle $\Theta$ is chosen to be continuous with respect to $x$ and $t$, and to satisfy $\Theta^{(1)}(0,0) \in [0,\pi)$ and $\Theta^\JJ(1,0) - \Theta^\J(1,0) \in [0,2\pi)$. 
Therefore, the angle $\theta^\J(t)$ between $\Gamma^\J_t$ and $\Gamma^\JJ_t$ at the junction coincides with $\Theta^\JJ(1,t) - \Theta^\J(1,t)$ modulo $2\pi$. 
We let $\lambda^\J_t$ be the tangent velocity of $\Gamma^\J_t$ and thus 
\[ \partial_t \xi^\J = V^\J_t \nu^\J_t + \lambda^\J_t \tau^\J_t = \sigma(\Delta^\J \alpha(t)) \kappa_t^\J \nu^\J_t + \lambda^\J_t \tau^\J_t. \]
Recall that we have rescaled so that $\mu=1$ for \eqref{eq:1.2} when we introduce the geometric flow \eqref{eq-curve}--\eqref{bc-boundary}. 
Notice also that the arc-length $s$ of $\Gamma^\J_t$ satisfies 
\begin{equation}\label{para-s-x}
\partial_s = \frac{1}{|\partial_x \xi^\J|} \partial_x. 
\end{equation}
We discuss some geometric properties to re-formulate. 

\begin{lem}\label{lem:pros-kappa}
Any smooth geometric flow satisfying \eqref{eq-curve}--\eqref{bc-boundary} fulfills the following identities. 
\begin{align}
&\partial_t |\partial_x \xi^\J| = (-V^\J_t \kappa_t^\J + \partial_s \lambda^\J_t) |\partial_x \xi^\J| = (-\sigma(\Delta^\J \alpha)(\partial_s \Theta^\J)^2 + \partial_s \lambda_t^\J)|\partial_x \xi^\J|, \label{jacobi-deri-t}\\
&\partial_t \Theta^\J = \partial_s V^\J_t + \kappa_t^\J \lambda^\J_t= \sigma(\Delta^\J \alpha) \partial_s^2 \Theta^\J + (\partial_s \Theta^\J)\lambda_t^\J, \label{eq-theta1}
\end{align}
for any $(x,t) \in [0,1] \times [0,T)$ and $j \in \{1,2,3\}$. 
Furthermore, 
\begin{equation}
\kappa^\J_t = \partial_s \Theta^\J = \lambda^\J_t = 0 \quad \text{at} \; \; P^\J \label{bc-kappa}
\end{equation}
for any $j \in \{1,2,3\}$ and 
\begin{align}
\sum_{j=1}^3 \sigma(\Delta^\J\alpha) V^\J_t = \sum_{j=1}^3 (\sigma(\Delta^\J \alpha))^2 \partial_s \Theta^\J = 0 \quad \text{at} \; \; \vec{a}(t) \label{bc-sum-kappa}
\end{align}
for any $j \in \{1,2,3\}$. 
\end{lem}

\begin{proof}
The equalities \eqref{jacobi-deri-t} and \eqref{eq-theta1} can be obtained by a standard argument, \eqref{eq-curve} and $\kappa^\J_t = \partial_s \Theta^\J$. 
We thus refer to \cite[Chapter 1]{CZ} for the details of the proof. 
We now prove only \eqref{bc-kappa} and \eqref{bc-sum-kappa}. 

From $\xi^\J(0,t) = P^\J(t)$ for any $t \in [0,T)$ and $j \in \{1,2,3\}$, taking the time derivatives on both sides, we have 
\[ V_t^\J \nu_t^\J + \lambda^\J_t \tau_t^\J = \partial_t \xi^\J(0,t) = 0 \]
at the boundary point. 
Since $\nu_t^\J$ and $\tau_t^\J$ are linearly independent, we have \eqref{bc-kappa}. 

At the junction point $\vec{a}(t)$, from $\xi^{(1)}(1,t) = \xi^{(2)}(1,t) = \xi^{(3)}(1,t)$, those time derivatives are also same. 
We thus have by taking inner product $\partial_t \xi^\I(1,t)$ with \eqref{bc-angle} rotated by $90$ degrees counterclockwise 
\[ \begin{aligned}
0 =&\; \left\langle \sum_{j=1}^3 \sigma(\Delta^\J\alpha) \nu_t^\J, \partial_t \xi^\I(1,t) \right\rangle \\
=&\; \sum_{j=1}^3 \langle \sigma(\Delta^\J\alpha) \nu_t^\J, V_t^\J \nu_t^\J + \lambda^\J_t \tau_t^\J \rangle = \sum_{j=1}^3 \sigma(\Delta^\J\alpha) V^\J_t 
\end{aligned}\]
at the junction point $\vec{a}(t)$. 
\end{proof}

We can adopt \eqref{eq-theta1} as the differential equation to be solved. 
In order to apply a classical theory to solve it, we will further continue that represent $\lambda^\J_t$ by the angle functions $\{\Theta^\I\}_{i=1,2,3}$ and some geometric values. 
The following lemma is to control $\lambda^\J_t$ by $\{\Theta^\I\}_{i=1,2,3}$ at $\vec{a}(t)$. 

\begin{lem}\label{lem:l-v-junction}
Any smooth geometric flow satisfying \eqref{eq-curve}--\eqref{bc-boundary} fulfills 
\begin{equation}\label{eq-lambda-V} 
\begin{pmatrix}
\lambda^{(1)}_t \\
\lambda^{(2)}_t \\
\lambda^{(3)}_t 
\end{pmatrix}
= - \frac{1}{1 - c^{(1)}c^{(2)}c^{(3)}}
\begin{pmatrix}
c^{(1)}c^{(2)}s^{(3)} & s^{(1)} & c^{(1)}s^{(2)} \\
c^{(2)}s^{(3)} & s^{(1)}c^{(2)}c^{(3)} & s^{(2)} \\
s^{(3)} & s^{(1)}c^{(3)} & c^{(1)}s^{(2)}c^{(3)}
\end{pmatrix}
\begin{pmatrix}
\sigma(\Delta^{(1)} \alpha) \partial_s \Theta^{(1)} \\ 
\sigma(\Delta^{(2)} \alpha) \partial_s \Theta^{(2)} \\
\sigma(\Delta^{(3)} \alpha) \partial_s \Theta^{(3)}
\end{pmatrix}
\end{equation}
at $\vec{a}(t)$ (or at $x=1$), where $c^\J = \cos (\Theta^\JJ - \Theta^\J)$ and $s^\J = \sin(\Theta^\JJ - \Theta^\J)$ for $j=1,2,3$. 
\end{lem}

\begin{proof}
We have by \eqref{bc-concurrency} 
\[ \partial_t \vec{a}(t) = V^\J_t \nu^\J_t + \lambda^\J_t \tau^\J_t = V^\JJ_t \nu^\JJ_t + \lambda^\JJ_t \tau^\JJ_t, \]
which implies by taking the inner product with $\tau^\J_t$ 
\begin{equation}\label{eq-lambda-V2} 
\lambda_t^\J = -s^\J V^\JJ_t + c^\J \lambda^\JJ_t. 
\end{equation}
We thus obtain by a simple calculation 
\[
\begin{pmatrix}
\lambda^{(1)}_t \\
\lambda^{(2)}_t \\
\lambda^{(3)}_t 
\end{pmatrix}
= - \frac{1}{1 - c^{(1)}c^{(2)}c^{(3)}}
\begin{pmatrix}
c^{(1)}c^{(2)}s^{(3)} & s^{(1)} & c^{(1)}s^{(2)} \\
c^{(2)}s^{(3)} & s^{(1)}c^{(2)}c^{(3)} & s^{(2)} \\
s^{(3)} & s^{(1)}c^{(3)} & c^{(1)}s^{(2)}c^{(3)}
\end{pmatrix}
\begin{pmatrix}
V^{(1)}_t \\ 
V^{(2)}_t \\
V^{(3)}_t
\end{pmatrix}.
\]
Apply \eqref{eq-curve} and $\partial_s \Theta^\J = \kappa^\J_t$ to obtain \eqref{eq-lambda-V}. 
\end{proof}

Since the tangent velocity $\lambda^\J_t$ depends on the choice of the parametrization in the interior of curves $\Gamma^\J_t$, we thus restrict the parametrization to satisfy
\begin{equation}\label{rest-para} 
|\partial_x \xi^\J (x,t)| = L^\J(t) \quad \text{for} \; \; x \in [0,1]. 
\end{equation}
Note that we will construct a geometric flow satisfying \eqref{rest-para} latter (see also Remark \ref{rmk:construction-flow}). 
Then, $\lambda^\J_t$ can be determine uniquely and we obtain the following formula. 

\begin{lem}\label{lem:re-lambda}
For any smooth geometric flow, let $\xi^\J: [0,1] \times [0,T)$ be a parametrization of $\Gamma^\J_t$ satisfying \eqref{rest-para}. 
Then, 
\begin{equation}\label{re-lambda}
\lambda^\J_t (x) = \frac{\sigma(\Delta^\J \alpha(t))}{L^\J(t)} \int_0^x (\partial_x \Theta^\J(\tilde{x},t))^2 \; d\tilde{x} + x \frac{d}{dt} L^\J(t) 
\end{equation}
for $(x,t) \in [0,1] \times [0,T)$ and $j \in \{1,2,3\}$. 
\end{lem}

\begin{proof}
Taking the time derivative on both sides of the square of \eqref{rest-para} and dividing the equality by $2$, we have by Frenet-Serret formulas 
\begin{align*}
L^\J(t) \frac{d}{dt} L^\J(t) =&\; \langle \partial_x \xi^\J, \partial_x \partial_t \xi^\J \rangle \\
=&\; \langle \partial_x \xi^\J, \partial_x (V^\J_t \nu^\J_t + \lambda^\J_t \tau^\J_t) \rangle \\
=&\; (L^\J(t))^2 \langle \partial_s \xi^\J, \partial_s (V^\J_t \nu^\J_t + \lambda^\J_t \tau^\J_t) \rangle \\
=&\; (L^\J(t))^2 (-\kappa^\J_t V^\J_t + \partial_s \lambda^\J_t), 
\end{align*}
which implies 
\[ \partial_x \lambda^\J_t = \frac{ \sigma(\Delta^\J \alpha(t))}{L^\J(t)} (\partial_x \Theta^\J)^2 + \frac{d}{dt} L^\J(t) \]
due to \eqref{eq-curve} and $\partial_s \Theta^\J = \kappa^\J_t$. 
Integrating it with respect to $x$ and applying \eqref{bc-kappa}, we have \eqref{re-lambda}. 
\end{proof}

We obtain a re-formulated differential equation by substituting \eqref{re-lambda} into \eqref{eq-theta1} including $L^\J(t)$ in the coefficients. 
We thus have to introduce a differential equation of $L^\J(t)$ to re-formulate the geometric flow into a system of $\Theta^\J, \alpha^\J$ and $L^\J$. 
We obtain immediately the equation by substituting $x=1$ into \eqref{re-lambda}. 

\begin{lem}\label{lem:deri-L-1}
Any smooth geometric flow satisfying \eqref{rest-para} fulfills 
\begin{equation}\label{deri-L-1}
\frac{d}{dt} L^\J(t) = - \frac{\sigma(\Delta^\J \alpha)}{L^\J} \int_0^1(\partial_x \Theta^\J)^2 \; dx + \lambda^\J_t (1). 
\end{equation}
\end{lem}

Summarizing the results so far, we obtain the following system. 
Hereafter, let $\vec{\Theta} = (\Theta^{(1)}, \Theta^{(2)}, \Theta^{(3)})$, $\vec{L}=(L^{(1)}, L^{(2)}, L^{(3)})$ and $\vec{\alpha} = (\alpha^{(1)}, \alpha^{(2)}, \alpha^{(3)})$ for simplicity. 

\begin{prop}
Any smooth geometric flow satisfying \eqref{eq-curve}--\eqref{bc-boundary} and \eqref{rest-para} fulfills the following system of $\vec{\Theta}, \vec{L}$ and $\vec{\alpha}$. 
\begin{equation}\label{system-theta}
\begin{cases}
\partial_t \Theta^\J = \frac{\sigma(\Delta^\J \alpha)}{(L^\J)^2}\partial_x^2 \Theta^\J + f^\J(\vec{\Theta}, \vec{L}, \vec{\alpha}), & (x,t) \in (0,1) \times (0,T), \; \; j=1,2,3, \\
\partial_x \Theta^\J (0,t) = 0, & t \in (0,T), \; \; j=1,2,3, \\
\sum_{j=1}^3 \sigma(\Delta^\J \alpha(t)) \cos \Theta^\J(1,t) = 0, & t \in (0,T), \\
\sum_{j=1}^3 \sigma(\Delta^\J \alpha(t)) \sin \Theta^\J(1,t) = 0, & t \in (0,T), \\
\sum_{j=1}^3 \frac{(\sigma(\Delta^\J \alpha(t)))^2}{L^\J(t)} \partial_x \Theta^\J(1,t) = 0, & t \in (0,T), \\
\pt L^\J  = - \frac{\sigma(\Delta^\J \alpha)}{L^\J} \int_0^1 (\partial_x \Theta^\J)^2 \; dx + g^\J(\vec{\Theta}, \vec{\alpha}), & t \in (0,T), \; \; j=1,2,3, \\
\pt \alpha^\J = -\gamma \{ \pa \sigma (\Delta^\JJ \alpha) L^\JJ - \pa \sigma (\Delta^\J \alpha) L^\J \}, & t \in (0,T), \; \; j =1,2,3,
\end{cases}
\end{equation}
where
\begin{equation}\label{functions-system}
\begin{aligned}
&\begin{aligned}f^\J(\vec{\Theta}, \vec{L}, \vec{\alpha})(x,t) =&\; \frac{\sigma(\Delta^\J \alpha)\partial_x \Theta^\J(x,t)}{(L^\J)^2} \left\{\int_0^x (\partial_x \Theta^\J(\tilde{x}, t))^2 \; d\tilde{x} - x \int_0^1 (\partial_x \Theta^\J(x,t))^2 \; dx \right\} \\
&\;+ \frac{x\px\Theta^\J(x,t)}{L^\J} g^\J(\vec{\Theta}, \vec{\alpha}), \end{aligned}\\
&\begin{aligned} (\lambda^\J_t(1) =) g^\J(\vec{\Theta}, \vec{\alpha})(t) = - \frac{1}{1-c^{(1)}c^{(2)}c^{(3)}} \Big\{&c^\J c^\JJ s^\JJJ \frac{\sigma(\Delta^\J \alpha)}{L^\J} \partial_x \Theta^\J(1,t) \\
&+ s^\J \frac{\sigma(\Delta^\JJ \alpha)}{L^\JJ}\partial_x \Theta^\JJ(1,t) \\
&+ c^\J s^\JJ \frac{\sigma(\Delta^\JJJ \alpha)}{L^\JJJ}\partial_x \Theta^\JJJ(1,t) \Big\}, \end{aligned} \\
&c^\J(t) = \cos(\Theta^\JJ(1,t) - \Theta^\J(1,t)), \quad s^\J(t) = \sin(\Theta^\JJ(1,t) - \Theta^\J(1,t)). 
\end{aligned}
\end{equation}
\end{prop}

\begin{proof}
We have \eqref{system-theta} except the third and forth equalities by combining Lemma \ref{lem:pros-kappa}, Lemma \ref{lem:l-v-junction}, Lemma \ref{lem:re-lambda} and Lemma \ref{lem:deri-L-1}. 
The remained equalities follows from \eqref{bc-angle} and the definition of $\Theta^\J$. 
\end{proof}

\begin{remark}\label{rmk:construction-flow}
For any smooth functions $\Theta^\J: [0, 1] \times [0,T) \to \mathbb{R}$ and $L^\J: [0, T) \to (0, \infty)$, and also the fixed end-points $P^\J$, we can construct uniquely the a moving curve satisfying \eqref{def-theta} and \eqref{rest-para} as 
\[ \xi^\J(x,t) :=  \left(\int_0^x L^\J(t) \cos \Theta^\J(\tilde{x},t) \; d\tilde{x}, \int_0^x L^\J(t) \sin \Theta^\J(\tilde{x},t) \; d\tilde{x}\right) + P^\J, \]
where $\xi^\J$ is the parametrization of the moving curve. 
Therefore, our aim is to construct a geometric flow governed by \eqref{eq-curve}--\eqref{bc-boundary} from a solution to \eqref{system-theta} applying the above formula. 
\end{remark}

\subsection{Linearized system}\label{subsec:linear}
We introduce a linearized problem of $\vec{\Theta}$ to apply a standard existence theory for system as in \cite{LSU}. 
A solution to \eqref{system-theta} will be solved by constructing a contraction map from the linearized problem and the differential equations of $\vec{L}$ and $\vec{\alpha}$ in \eqref{system-theta}. 
We thus consider the following linearized problem around initial datum $\vec{\Theta}_0, \vec{L}_0$ and $\vec{\alpha}_0$ of $\vec{\Theta}, \vec{L}$ and $\vec{\alpha}$, respectively, for the differential equation of $\vec{\Theta}$ in \eqref{system-theta}. 
\begin{equation}\label{linear-theta}
\begin{cases}
\pt \Theta^\J - \frac{\sigma(\Delta^\J \alpha_0)}{(L^\J_0)^2} \px^2 \Theta^\J = F^\J(x,t), & (x,t) \in (0,1) \times (0,T), \; j \in \{1,2,3\}, \\
\px \Theta^\J(0,t) = 0, & t\in(0,T), \; \; j \in \{1,2,3\} \\
\sum_{j=1}^3 \left(\sigma(\Delta^\J \alpha_0) \sin \Theta^\J_0 (1)\right) \Theta^\J(1,t) = b_1(t), & t \in (0,T), \\
\sum_{j=1}^3 \left(\sigma(\Delta^\J \alpha_0) \cos \Theta^\J_0 (1)\right) \Theta^\J(1,t) = b_2(t), & t \in (0,T), \\
\sum_{j=1}^3 \frac{(\sigma(\Delta^\J \alpha_0))^2}{L_0^\J} \px \Theta^\J(1,t) = b_3(t), & t \in (0,T), 
\end{cases}
\end{equation}
where
\begin{equation}\label{want-linear}
\begin{aligned}
&F^\J (x,t) = \left(\frac{\sigma(\Delta^\J \alpha)}{(L^\J)^2} - \frac{\sigma(\Delta^\J \alpha_0)}{(L^\J_0)^2}\right) \px^2 \Theta^\J + f^\J(\vec{\Theta}, \vec{L}, \vec{\alpha}), \\
&b_1 (t) = \sum_{j=1}^3 \left\{\sigma(\Delta^\J \alpha) \cos \Theta^\J(1,t) + \left(\sigma(\Delta^\J \alpha_0) \sin \Theta^\J_0 (1)\right) \Theta^\J(1,t)\right\}, \\
&b_2 (t) = \sum_{j=1}^3 \left\{\left(\sigma(\Delta^\J \alpha_0) \cos \Theta^\J_0 (1)\right) \Theta^\J(1,t) - \sigma(\Delta^\J \alpha) \sin \Theta^\J(1,t)\right\}, \\
&b_3 (t) = \sum_{j=1}^3 \left(\frac{(\sigma(\Delta^\J \alpha_0))^2}{L_0^\J} - \frac{(\sigma(\Delta^\J \alpha))^2}{L^\J(t)}\right) \px \Theta^\J(1,t). 
\end{aligned}
\end{equation}

We here construct a solution to \eqref{linear-theta} for general functions $F^\J, b_i$ and initial datum $\vec{\Theta}_0, \vec{L}_0$ and $\vec{\alpha}_0$. 
In order to apply a standard existence theory for system as in \cite{LSU}, the parabolicity, complementing condition and compatibility condition should be satisfied. 
Since the parabolicity obviously holds, we will discuss the other conditions. 
We refer to \cite[p.601, Definition 4]{LSU} for the detail of the definition of the parabolicity for system. 
First, we show that the complementing condition is satisfied (see \cite[Section 9 in Chapter VII]{LSU} for the detail of the complementing condition). 

\begin{lem}
The boundary conditions in \eqref{linear-theta} satisfy the complementing condition if $\Theta^\JJ_0(1) - \Theta^\J_0(1) \in (0,\pi)$ for $j=1,2,3$. 
\end{lem}

\begin{proof}
We demonstrate according to Bronsard and Reitich \cite{BR} (see also Eidelman and Zhitarashu \cite{EZ}). 
We will discuss the complementing condition at $x=1$ since the condition at $x=0$ is obviously satisfied. 
To describe the condition at $x=1$, let $\mathcal{L}(x,t, \px,\pt)$ and $\mathcal{B}(1,t,\px,\pt)$ be the $3\times 3$ matrix of the differential equation and the boundary conditions at $x=1$ in \eqref{linear-theta}, respectively, i.e., 
\[ \mathcal{L}(x,t,\px,\pt) \vec{\Theta} = (F^{(1)}, F^{(2)}, F^{(3)})^T, \quad \mathcal{B}(1,t,\px,\pt)\vec{\Theta} = (b_1, b_2, b_3)^T, \]
where $\vec{\Theta} = (\Theta^{(1)}, \Theta^{(2)}, \Theta^{(3)})^T$. 
Then, for $p \in \mathbb{C} \setminus \{0\}$ with ${\rm Re}(p) \ge 0$, 
\begin{align*} 
&\mathcal{L}(1,t,\px,p) = 
\begin{pmatrix}
p-\frac{\sigma(\Delta^{(1)} \alpha_0)}{(L^{(1)}_0)^2} \px^2 & 0 & 0 \\
0 & p-\frac{\sigma(\Delta^{(2)} \alpha_0)}{(L^{(2)}_0)^2} \px^2 & 0 \\
0 & 0 & p-\frac{\sigma(\Delta^{(3)} \alpha_0)}{(L^{(3)}_0)^2} \px^2
\end{pmatrix}, \\
&\mathcal{B}(1,t,\px,p) = 
\begin{pmatrix}
\sigma(\Delta^{(1)} \alpha_0) \sin\Theta^{(1)}_0(1) & \sigma(\Delta^{(2)} \alpha_0) \sin\Theta^{(2)}_0(1) & \sigma(\Delta^{(3)} \alpha_0) \sin\Theta^{(3)}_0(1) \\
\sigma(\Delta^{(1)} \alpha_0) \cos\Theta^{(1)}_0(1) & \sigma(\Delta^{(2)} \alpha_0) \cos\Theta^{(2)}_0(1) & \sigma(\Delta^{(3)} \alpha_0) \cos\Theta^{(3)}_0(1) \\
 \frac{(\sigma(\Delta^{(1)} \alpha_0))^2}{L^{(1)}_0} \px & \frac{(\sigma(\Delta^{(2)} \alpha_0))^2}{L^{(2)}_0} \px & \frac{(\sigma(\Delta^{(3)} \alpha_0))^2}{L^{(3)}_0} \px
\end{pmatrix}.
\end{align*} 
Let $z$ be the distance parameter from the boundary $x=1$ directing the interior of $[0,1]$. 
Then we have $\px = - \partial_z$. 
A general solution of $\mathcal{L}(1,t,-\partial_z,p)W(z; p) = 0$ for $z \in [0, \infty)$ satisfying $|W(z; p)| \to 0$ as $z \to \infty$ is 
\[ W(z; p) = \left(a_1 e^{- L_0^{(1)} \sqrt{\frac{p}{\sigma(\Delta^{(1)} \alpha_0)}} z}, a_2 e^{- L_0^{(2)} \sqrt{\frac{p}{\sigma(\Delta^{(2)} \alpha_0)}} z}, a_3 e^{- L_0^{(3)} \sqrt{\frac{p}{\sigma(\Delta^{(3)} \alpha_0)}} z} \right)^T \]
for any $\vec{a} = (a_1, a_2, a_3) \in \mathbb{C}^3$, where we choose the square roots so that the real part of them is  positive. 
Let the $3 \times 3$ matrix $\mathcal{D}(p)$ be defined by 
\[ \mathcal{D}(p) \vec{a}^T = \left(\mathcal{B}(1,t,-\partial_z,p) W(z; p) \Big|_{z=0}\right). \]
Applying \cite[Lemma I.1]{EZ}, the complementing condition is satisfied if ${\rm det} \mathcal{D}(p) \neq 0$ for $p \in \mathbb{C} \setminus \{0\}$ with ${\rm Re}(p) \ge 0$. 
By a simple calculation, we can see that 
\begin{align*}
{\rm det} \mathcal{D}(p) =&\; \sqrt{p} \left(\prod_{j=1}^3 \sigma(\Delta^\J \alpha_0)\right) {\rm det} 
\begin{pmatrix}
\sin \Theta^{(1)}_0(1) & \sin \Theta^{(2)}_0(1) & \sin \Theta^{(3)}_0(1) \\
\cos \Theta^{(1)}_0(1) & \cos \Theta^{(2)}_0(1) & \cos \Theta^{(3)}_0(1) \\
\sqrt{\sigma(\Delta^{(1)} \alpha_0)} & \sqrt{\sigma(\Delta^{(2)} \alpha_0)} & \sqrt{\sigma(\Delta^{(3)} \alpha_0)}
\end{pmatrix}\\
=&\;  -\sqrt{p} \left(\prod_{j=1}^3 \sigma(\Delta^\J \alpha_0)\right) \left(\sum_{j=1}^3 \sqrt{\sigma(\Delta^\JJJ \alpha_0)} \sin \Big(\Theta^\JJ_0(1) - \Theta^\J_0(1)\Big)\right) \neq 0
\end{align*}
due to the positivity of $\sigma$ and the assumption $\Theta^\JJ_0(1) - \Theta^\J_0(1) \in (0,\pi)$ for $j=1,2,3$. 
\end{proof}

We next discuss the compatibility conditions for the system \eqref{linear-theta}. 
The conditions is related to the regularity of the solution, we thus define the following standard H\"older spaces and the norms. 

\begin{definition}
Let $I$ and $J $ be open intervals in $\R$
and $l \in (0,\infty) \setminus \N$. We denote
\[
C^{l} ( \overline{I} ) = \{ f \in C^{[l]} (\overline{I}) \mid \| f \|_{C^{l} ( \overline{I} )}<\infty \},
\]
where 
\[
\| f \|_{C^{l} ( \overline{I} )} 
=
\| f \|_{C^{[l]} (\overline{I})} + \sup _{x,y \in I , x\not=y}
\frac{| \partial _x ^{[l]} f(x) -\partial _x ^{[l]} f(y) |}{|x-y|^{l-[l]}}.
\]
For simplicity, we denote $C^{l} ([0,1]) $ and $C^{l} ([0,T]) $ by 
$C^{l} _x $ and $C^{l} _t$, respectively. In addition,
we define
\[
C^{l, \frac{l}{2}} ( \overline{I \times J} )
= \{ f \in C (\overline{I \times J}) \mid 
\partial^r _t \partial^s _x f \in C (\overline{I \times J}) \ \text{if} \ 2r +s <l, \
\| f \|_{C^{l, \frac{l}{2}} ( \overline{I \times J} ) }
<\infty \},
\]
where
\[
\begin{aligned}
\| f \|_{C^{l, \frac{l}{2}} ( \overline{I \times J} )} 
=&
\sum _{j=0} ^{[l]} \sum _{2r +s =j }\| \partial^r _t \partial^s _x f \|_{C (\overline{I\times J})} 
+ \sum_{2r+s=[l] } \sup _{x,y \in I , t \in J,  x\not=y}
\frac{| \partial _t ^r \partial _x ^s f(x,t) -\partial _t ^r \partial _x ^s f(y,t) |}
{|x-y|^{l-[l]} } \\
& + \sum_{0< (l-2r-s)/2<1} \sup _{x \in I , t,s \in J,  t\not=s}
\frac{| \partial _t ^r \partial _x ^s f(x,t) -\partial _t ^r \partial _x ^s f(x,s) |}
{|t-s|^{(l-2r-s)/2} }.
\end{aligned}
\]
As before, we denote $C^{l, \frac{l}{2}} ( [0,1] \times [0,T] )$ by 
$C^{l, \frac{l}{2}} _{x,t}$, for simplicity.
\end{definition}

If a solution $\vec{\Theta}$ to \eqref{linear-theta} with $b_1, b_2 \in C^\beta_t([0,T'])$ is of class $(C^{2+\beta, \frac{2+\beta}{2}}_{x,t}([0,1] \times [0,T']))^3$ for $\beta \in (0,1)$ and $T' \in (0,T)$, then the solution should satisfy 
\begin{equation}\label{com-2-system}
\begin{aligned} 
0 =&\; \pt \left(b_1(t) - \sum_{j=1}^3 \left(\sigma(\Delta^\J \alpha_0) \sin \Theta^\J_0 (1)\right) \Theta^\J(1,t)\right)\Big|_{t=0} \\
=&\; \pt b_1(0) - \sum_{j=1}^3 \left(\sigma(\Delta^\J \alpha_0) \sin \Theta^\J_0 (1)\right) \left(\frac{\sigma(\Delta^\J \alpha_0)}{(L^\J_0)^2} \px^2 \Theta^\J_0(1) + F^\J(1,0)\right)
\end{aligned}
\end{equation}
due to \eqref{linear-theta}. 
We have similarly 
\[ \pt b_2(0) - \sum_{j=1}^3 \left(\sigma(\Delta^\J \alpha_0) \cos \Theta^\J_0 (1)\right) \left(\frac{\sigma(\Delta^\J \alpha_0)}{(L^\J_0)^2} \px^2 \Theta^\J_0(1) + F^\J(1,0)\right) = 0. \]
These equalities can be regarded as boundary conditions of the initial data $\vec{\Theta}_0$. 
Notice that the boundary conditions in \eqref{linear-theta} also should be satisfied at $t=0$. 
We say that an initial data $\vec{\Theta}_0 \in (C_x^{2+\beta}([0,1]))^3$ satisfies the compatibility condition of order $k \in \mathbb{N}$ for \eqref{linear-theta} if the initial data $\vec{\Theta}_0$ fulfills all boundary conditions of $\vec{\Theta}_0$ which a solution $\vec{\Theta} \in (C^{k+\beta, \frac{k+\beta}{2}}_{x,t}([0,1] \times [0,T']))^3$ to \eqref{linear-theta} starting from $\vec{\Theta}_0$ should satisfy. 
Note that we derived the formulas for the compatibility condition of order 2 above with given functions $F^\J$ and $b_j$. 
The compatibility conditions for \eqref{system-theta} and the geometric flow \eqref{eq-curve}--\eqref{bc-boundary} are similarly defined. 
We then obtain the existence theory for \eqref{linear-theta} due to \cite[Theorem 10.1 in Chapter VII]{LSU}. 

\begin{prop}\label{prop:exist-linear}
Let $T>0$, $\beta \in (0,1)$ and $k \in \mathbb{N}$ with $k \ge 2$. 
Assume that $\sigma$ is a positive smooth function and $\alpha_0^\J \in \mathbb{R}$ and $L_0^\J > 0$ are given constants for $j \in \{1,2,3\}$. 
Let $F^\J \in C^{(k-2) + \beta, \frac{(k-2) + \beta}{2}}_{x,t}([0,1] \times [0,T])$ for $j \in \{1,2,3\}$, $b_1, b_2 \in C^{\frac{k+\beta}{2}}([0,T])$ and $b_3 \in C^{\frac{(k-1)+\beta}{2}}_t([0,T])$. 
Assume also $\vec{\Theta}_0 = (\Theta_0^{(1)}, \Theta_0^{(2)}, \Theta_0^{(3)}) \in (C^{k+\beta}_x([0,1]))^3$ satisfy the compatibility condition of order $k$ for \eqref{linear-theta} and $\Theta^\JJ_0(1) - \Theta^\J_0(1) \in (0,\pi)$ for $j=1,2,3$. 
Then, there exists $\Cl[c]{c-short}>0$ such that \eqref{linear-theta} has a unique solution $\vec{\Theta} = (\Theta^{(1)}, \Theta^{(2)}, \Theta^{(3)}) \in (C^{k+\beta, \frac{k+\beta}{2}}_{x,t} ([0,1] \times [0,T]))^3$ with $\vec{\Theta}(\cdot,0) = \vec{\Theta}_0$ and the following estimate holds. 
\begin{equation}\label{maximum-reg}
\begin{aligned}
\sum_{j=1}^3 \|\Theta^\J\|_{C^{k+\beta, \frac{k+\beta}{2}}_{x,t}} \le \Cr{c-short} \Bigg(\sum_{j=1}^3 &\Big(\|F^\J\|_{C^{(k-2)+\beta, \frac{(k-2)+\beta}{2}}_{x,t}} + \|\Theta_0^\J\|_{C^{k+\beta}_x}\Big) \\
&\;+ \|b_1\|_{C^{\frac{k+\beta}{2}}_t} + \|b_2\|_{C^{\frac{k+\beta}{2}}_t} + \|b_3\|_{C^{\frac{(k-1)+\beta}{2}}_t} \Bigg).
\end{aligned}
\end{equation}
\end{prop}

\subsection{Short-time existence}

In this section, we develop Proposition \ref{prop:exist-linear} to obtain a solution to \eqref{system-theta}. 
In order to construct a contraction map to obtain the solution, we have to introduce compatibility conditions, 
which are different from the compatibility conditions defined above (re-note that $F^\J$ and $b_j$ are given functions in the above conditions for \eqref{linear-theta}). 
To consider the fixed point theorem, let $(\underline{\vec{\Theta}}, \underline{\vec{L}}, \underline{\vec{\alpha}})$ 
belong a suitable function space defined below and
$F^\J$ and $b_j$ ($j=1,2,3$) are functions defined by \eqref{want-linear} with
$(\underline{\vec{\Theta}}, \underline{\vec{L}}, \underline{\vec{\alpha}})$
instead of $(\vec{\Theta}, \vec{L}, \vec{\alpha})$.
Then, to solve \eqref{linear-theta}, 
the initial data $(\vec{\Theta} _0, \vec{L}_0, \vec{\alpha}_0)$ needs to satisfy the compatibility condition for \eqref{linear-theta} with $F^\J$ and $b_j$ chosen as before.
Since $F^\J$ and $b_j$ depend on $(\underline{\vec{\Theta}}, \underline{\vec{L}}, \underline{\vec{\alpha}})$ in the choice, 
for given initial data $(\vec{\Theta} _0, \vec{L}_0, \vec{\alpha}_0)$,
we must choose suitable functions 
$(\underline{\vec{\Theta}}, \underline{\vec{L}}, \underline{\vec{\alpha}})$
so that the initial data satisfies the compatibility conditions for \eqref{linear-theta}.
Let $k \in \mathbb{N}$ with $k \ge 2$ and $\beta \in (0,1)$.  
If $\vec{\Theta} \in (C_{x,t}^{k+\beta, \frac{k+\beta}{2}}([0,1]\times[0,T']))^3$, $\vec{L} \in (C_t^{\frac{k+\beta}{2}}([0,T']))^3$ and $\vec{\alpha} \in (C_t^{\frac{k+\beta}{2}}([0,T']))^3$ are solution to \eqref{system-theta} starting from initial datum $\vec{\Theta}_0, \vec{L}_0$ and $\vec{\alpha}_0$ for some $T'$, then $\partial_t^i \Theta^\J(\cdot,0)$, $\partial_t^i L^\J(0)$ and $\partial_t^i \alpha^\J(0)$ can be represented by the initial datum and the derivatives of $\vec{\Theta}_0$ for $j=1,2,3$ and $i=1,2, \cdots, [\frac{k}{2}]$ due to the first and last two equalities in \eqref{system-theta}.
For example, if $k=2$, the representations are 
\begin{equation}\label{compatibility-initial} 
\begin{aligned}
&\partial_t \Theta^\J(x, 0) = \frac{\sigma( \Delta^\J \alpha_0)}{(L^\J_0)^2} \partial_x^2 \Theta^\J_0(x) + f^\J(\vec{\Theta}_0, \vec{L}_0, \vec{\alpha}_0)(x), \\
&\partial_t L^\J(0) = -\frac{\sigma(\Delta^\J \alpha_0)}{L^\J_0} \int_0^1(\px \Theta^\J_0)^2 \; dx + g^\J(\vec{\Theta}_0, \vec{\alpha}_0), \\
&\pt \alpha^\J(0) = -\gamma\{\pa \sigma(\Delta^\JJ \alpha_0) L^\JJ_0 - \pa \sigma(\Delta^\J\alpha_0) L^\J_0\} 
\end{aligned}
\end{equation}
which can be obtained by substituting $t=0$ into the first and last two equalities in \eqref{system-theta}. 
We say that general functions, which are not  necessary to satisfy any differential equations, $\underline{\vec{\Theta}} \in (C_{x,t}^{k+\beta, \frac{k+\beta}{2}}([0,1]\times[0,T']))^3$, $\underline{\vec{L}} \in (C_t^{\frac{k+\beta}{2}}([0,T']))^3$ and $\underline{\vec{\alpha}} \in (C_t^{\frac{k+\beta}{2}}([0,T']))^3$ satisfy a compatibility condition of order $k$ for \eqref{system-theta} at the initial time with initial datum $\vec{\Theta}_0 \in (C_x^{k+\beta}([0,1]))^3, \vec{L}_0 \in \{(0,\infty)\}^3$ and $\vec{\alpha}_0 \in \mathbb{R}^3$ if the functions satisfy all representations of $\partial_t^i \underline{\Theta}^\J(\cdot,0)$, $\partial_t^i \underline{L}^\J(0)$ and $\partial_t^i \underline{\alpha}^\J(0)$ by the initial datum for $j=1,2,3$ and $i=1,2, \cdots, [\frac{k}{2}]$, which follow from the first and last two equalities in \eqref{system-theta} 
with 
$(\underline{\vec{\Theta}}, \underline{\vec{L}}, \underline{\vec{\alpha}})
=(\vec{\Theta}, \vec{L}, \vec{\alpha})$
as above (for example, functions $(\underline{\vec{\Theta}}, \underline{\vec{L}}, \underline{\vec{\alpha}})$
should satisfy 
$(\underline{\vec{\Theta}}, \underline{\vec{L}}, \underline{\vec{\alpha}})|_{t=0}
=(\vec{\Theta}_0, \vec{L}_0, \vec{\alpha}_0)$
and \eqref{compatibility-initial} with 
$(\underline{\vec{\Theta}}, \underline{\vec{L}}, \underline{\vec{\alpha}})
=(\vec{\Theta}, \vec{L}, \vec{\alpha})$ when $k=2$). 
We then discuss the relation between the compatibility conditions for \eqref{system-theta} at the initial time and the compatibility conditions for \eqref{linear-theta}. 
Note that we assume a higher regularity on $\underline{\vec{L}}$ and $\underline{\vec{\alpha}}$ since the regularity assumption will be needed to prove that a map constructed later is contraction. 

\begin{lem}\label{lem:compatibility}
Let $k \in \mathbb{N}$ with $k \ge 2$ and $\beta, \beta' \in (0,1)$ with $0 < \beta < \beta' < 1$. 
Assume $\vec{\Theta}_0 \in (C_x^{k+\beta}([0,1]))^3, \vec{L}_0 \in \{(0,\infty)\}^3$ and $\vec{\alpha}_0 \in \mathbb{R}^3$ satisfy the compatibility condition of order $k$ for \eqref{system-theta}. 
Let 
\begin{equation}\label{def-initial-const} 
\Cl[m]{c-initial-max} := \max_{j=1,2,3} \|\Theta^\J_0 \|_{C_x^{k+\beta}([0,1])} + \max_{j=1,2,3} L^\J_0 + \max_{j=1,2,3}|\alpha^\J_0|, \quad \Cl[m]{c-initial-min} := \min_{j=1,2,3} L^\J_0. 
\end{equation}
Then, there exists $\Cl[c]{c-short2} > 0$ depending only on $\gamma, \sigma, k, \beta, \Cr{c-initial-max}$ and $\Cr{c-initial-min}$ such that 
\begin{equation}\label{def-sets}
\begin{aligned}
X_T^M:= \{(\underline{\vec{\Theta}}, \underline{\vec{L}}, \underline{\vec{\alpha}}) &\in (C_{x,t}^{k+\beta, \frac{k+\beta}{2}}([0,1]\times[0,T]))^3 \times (C_t^{\frac{k+\beta'}{2}}([0,T]))^3 \times (C_t^{\frac{k+\beta'}{2}}([0,T]))^3: \\
& \text{$\underline{\vec{\Theta}}, \underline{\vec{L}}$ and $\underline{\vec{\alpha}}$ satisfy conditions {\em (i)}, {\em (ii)} and {\em (iii)}}\}, 
\end{aligned}
\end{equation}
where 
\begin{itemize}
\item[{\em (i)}] $\underline{\vec{\Theta}}(\cdot,0) = \vec{\Theta}_0, \; \; \underline{\vec{L}}(0) = \vec{L}_0, \; \; \underline{\vec{\alpha}}(0) = \vec{\alpha}_0$ 
\item[{\em (ii)}] $\displaystyle \max_{j=1,2,3} \|\underline{\Theta}^\J \|_{C_{x,t}^{k+\beta}} + \max_{j=1,2,3} \|\underline{L}^\J\|_{C_t^{\frac{k+\beta'}{2}}} + \max_{j=1,2,3}\|\underline{\alpha}^\J\|_{C_t^{\frac{k+\beta'}{2}}} \le M$ and $\displaystyle \min_{j=1,2,3, \; \; t \in [0,T]} \underline{L}^\J \ge \Cr{c-initial-min}/2$ 
\item[{\em (iii)}] $(\underline{\vec{\Theta}}, \underline{\vec{L}}, \underline{\vec{\alpha}})$ satisfy the compatibility condition of order $k$ for \eqref{system-theta} at the initial time with $(\vec{\Theta}_0, \vec{L}_0, \vec{\alpha}_0)$, 
\end{itemize}
is non-empty for any $T>0$ and $M>\Cr{c-short2}$. 
Furthermore, for any $(\underline{\vec{\Theta}}, \underline{\vec{L}}, \underline{\vec{\alpha}}) \in X_T^M$, the initial function $\vec{\Theta}_0$ satisfies the compatibility condition of order $k$ for \eqref{linear-theta} with 
\begin{equation}\label{map-Fb}
\begin{aligned}
&F^\J = F^\J (x,t; \vec{\underline{\Theta}}, \vec{\underline{L}}, \vec{\underline{\alpha}}) = \left(\frac{\sigma(\Delta^\J \underline{\alpha})}{(\underline{L}^\J)^2} - \frac{\sigma(\Delta^\J \alpha_0)}{(L^\J_0)^2}\right) \px^2 \underline{\Theta}^\J + f^\J(\vec{\underline{\Theta}}, \vec{\underline{L}}, \vec{\underline{\alpha}}), \\
&b_1 = b_1(t; \vec{\underline{\Theta}}, \vec{\underline{L}}, \vec{\underline{\alpha}}) = \sum_{j=1}^3 \left\{\sigma(\Delta^\J \underline{\alpha}) \cos \underline{\Theta}^\J(1,t) + \left(\sigma(\Delta^\J \alpha_0) \sin \Theta^\J_0 (1)\right) \underline{\Theta}^\J(1,t)\right\}, \\
&b_2 = b_2(t; \vec{\underline{\Theta}}, \vec{\underline{L}}, \vec{\underline{\alpha}}) = \sum_{j=1}^3 \left\{\left(\sigma(\Delta^\J \alpha_0) \cos \Theta^\J_0 (1)\right) \underline{\Theta}^\J(1,t) - \sigma(\Delta^\J \underline{\alpha}) \sin \underline{\Theta}^\J(1,t)\right\}, \\
&b_3 = b_3(t; \vec{\underline{\Theta}}, \vec{\underline{L}}, \vec{\underline{\alpha}}) = \sum_{j=1}^3 \left(\frac{(\sigma(\Delta^\J \alpha_0))^2}{L_0^\J} - \frac{(\sigma(\Delta^\J \underline{\alpha}))^2}{\underline{L}^\J(t)}\right) \px \underline{\Theta}^\J(1,t), 
\end{aligned}
\end{equation}
where $f^\J$ is the function defined by \eqref{functions-system}. 
\end{lem}

\begin{proof}
We construct a typical element of $X_T^M$ when $k=2$ since a similar argument works even if $k>2$. 
Let expand $\Theta^\J_0$ on $\mathbb{R}$ so that $\Theta^\J_0 \in C_x^{2+\beta}(\mathbb{R})$ in the following argument and denote $\rho_\varepsilon(x) = \rho(x/\varepsilon)/\varepsilon$ by the mollifier with positive parameter $\varepsilon>0$. 
Fix sufficiently small $\delta$. 
Letting 
\begin{align*} 
\underline{\Theta}^\J(x,t) := &\; \Theta^\J_0(x) + \int_0^t \rho_{\tilde{t}^{\frac{1}{2}-\delta}} * \left(\frac{\sigma( \Delta^\J \alpha_0)}{(L^\J_0)^2} \partial_x^2 \Theta^\J_0 + f^\J(\vec{\Theta}_0, \vec{L}_0, \vec{\alpha}_0)\right)(x) \; d\tilde{t} \\
\underline{L}^\J (t) :=& \; L^\J_0 + \left(-\frac{\sigma(\Delta^\J \alpha_0)}{L^\J_0} \int_0^1(\px \Theta^\J_0)^2 \; dx + g^\J(\vec{\Theta}_0, \vec{\alpha}_0)\right) t, \\
\underline{\alpha}^\J(t) :=&\; \alpha^\J_0 - \gamma \left(\pa \sigma(\Delta^\JJ \alpha_0) L^\JJ_0 - \pa \sigma(\Delta^\J\alpha_0) L^\J_0\right)t
\end{align*}
for sufficiently small $t>0$, we then see that $(\underline{\vec{\Theta}}, \underline{\vec{L}}, \underline{\vec{\alpha}})$ satisfies the compatibility condition at the initial time \eqref{compatibility-initial}, and the condition (i) in Lemma \ref{lem:compatibility} is obviously satisfied. 
Note that, due to the choice of the parameter of the mollifier, we can see $\underline{\Theta}^\J \in C_{x,t}^{2+\beta, \frac{2+\beta}{2}}$ in short time, for example, 
\begin{align*} 
\|\partial_x^2 \underline{\Theta}^\J(\cdot,t)\|_{C_x^{\beta}} \le&\; \|\Theta^\J_0\|_{C_x^{2+\beta}} \\
&\;+ \int_0^t \frac{1}{\tilde{t}^{1-2\delta}} \int_{\mathbb{R}} |\partial_x^2\rho(x)| \; dxd\tilde{t} \left\|\frac{\sigma( \Delta^\J \alpha_0)}{(L^\J_0)^2} \partial_x^2 \Theta^\J_0 + f^\J(\vec{\Theta}_0, \vec{L}_0, \vec{\alpha}_0) \right\|_{C_x^{\beta}} \\
\le &\; \|\Theta^\J_0\|_{C_x^{2+\beta}} + C \left\|\frac{\sigma( \Delta^\J \alpha_0)}{(L^\J_0)^2} \partial_x^2 \Theta^\J_0 + f^\J(\vec{\Theta}_0, \vec{L}_0, \vec{\alpha}_0) \right\|_{C_x^{\beta}} t^{2\delta}
\end{align*}
for $t>0$ small and some constant $C$ depending only on $\delta$ and $\rho$. 
Therefore, choosing $\Cr{c-short2}$ large and adjusting $(\underline{\vec{\Theta}}, \underline{\vec{L}}, \underline{\vec{\alpha}})$ away from $t=0$ to satisfy the condition (ii), we see that $X_T^M$ is non-empty for any $M>\Cr{c-short2}$ and $T>0$. 

The remained statement can be seen easily from the definition of the two kind of compatibility conditions and the introduction of \eqref{linear-theta}. 
Indeed, by means of (i) and (iii), we can see that 
\begin{align*} 
&F^\J(1,0; \underline{\vec{\Theta}}, \underline{\vec{L}}, \underline{\vec{\alpha}}) =\frac{\px \Theta^\J_0(1)}{L^\J_0} g^\J(\vec{\Theta}_0, \vec{\alpha}_0) \\
&\pt b_1(0; \underline{\vec{\Theta}}, \underline{\vec{L}}, \underline{\vec{\alpha}})= -\gamma \sum_{j=1}^3 \pa \sigma(\Delta^\J \alpha_0) (\pa \sigma(\Delta^\JJ \alpha_0) L^\JJ_0 - \pa \sigma(\Delta^\J \alpha_0) L^\J_0) \cos \Theta_0^\J(1),
\end{align*}
where $g^\J$ is the function defined by \eqref{functions-system}. 
Therefore, \eqref{com-2-system}, which is one of the compatibility condition of order $2$ for \eqref{linear-theta}, is equivalent to 
\begin{align*}
\sum_{j=1}^3 \Bigg\{& \left(\sigma(\Delta^\J \alpha_0) \sigma \Theta^\J_0\right) \left(\frac{\sigma(\Delta^\J \alpha_0)}{(L^\J_0)^2} \px^2 \Theta^\J_0(1) + \frac{\px \Theta^\J_0(1)}{L^\J_0} g^\J(\vec{\Theta}_0, \vec{\alpha}_0)\right) \\
&\;+ \gamma \pa \sigma(\Delta^\J \alpha_0) \left(\pa \sigma(\Delta^\JJ \alpha_0) L^\JJ_0 - \pa \sigma(\Delta^\J \alpha_0) L^\J_0\right) \cos \Theta_0^\J(1)\Bigg\} = 0 
\end{align*}
if $F^\J$ and $b_j$ are chosen as in \eqref{map-Fb}, which coincides with one of the compatibility condition of order $2$ for \eqref{system-theta} derived from time derivative of the third equality in \eqref{system-theta}. 
Similarly we can see that $\vec{\Theta}_0$ satisfies the compatibility condition of order $k$ for \eqref{linear-theta} with \eqref{map-Fb}. 
\end{proof}

We next construct a contraction map on $X_T^M$ to obtain a solution to \eqref{system-theta}. 

\begin{definition}\label{def-con-map}
Let $k \in \mathbb{N}$ with $k \ge 2$ and $\beta, \beta' \in (0,1)$ with $0 < \beta < \beta' < 1$. 
Assume $\vec{\Theta}_0 \in (C_x^{k+\beta}([0,1]))^3, \vec{L}_0 \in \{(0,\infty)\}^3$ and $\vec{\alpha}_0 \in \mathbb{R}^3$ satisfy the compatibility condition of order $k$ for \eqref{system-theta} and $\Theta_0^\JJ(1) - \Theta_0^\J(1) \in (0,\pi)$ for $j=1,2,3$. 
Let $\Cr{c-short2}$ be the constant in Lemma \ref{lem:compatibility}. 
Define $X_T$ and $X_T^M$ by 
\[ \begin{aligned}
X_T:= \{(\underline{\vec{\Theta}}, \underline{\vec{L}}, \underline{\vec{\alpha}}) &\in (C_{x,t}^{k+\beta, \frac{k+\beta}{2}}([0,1]\times[0,T]))^3 \times (C_t^{\frac{k+\beta'}{2}}([0,T]))^3 \times (C_t^{\frac{k+\beta'}{2}}([0,T]))^3: \\
& \underline{\vec{\Theta}}(\cdot,0) = \vec{\Theta}_0, \; \; \underline{\vec{L}}(0) = \vec{L}_0, \; \; \underline{\vec{\alpha}}(0) = \vec{\alpha}_0 \} 
\end{aligned} \]
and \eqref{def-sets}, respectively, for any $T>0$ and $M > \Cr{c-short2}$. 
Define also a map $\Phi: X_T^M \to X_T$ by, for $(\vec{\underline{\Theta}}, \vec{\underline{L}}, \vec{\underline{\alpha}}) \in X_T^M$, 
\[ \Phi(x,t;\vec{\underline{\Theta}}, \vec{\underline{L}}, \vec{\underline{\alpha}}):= \left(\vec{\Theta}(x,t; \vec{\underline{\Theta}}, \vec{\underline{L}}, \vec{\underline{\alpha}}), \vec{L}(t; \vec{\underline{\Theta}}, \vec{\underline{L}}, \vec{\underline{\alpha}}), \vec{\alpha}(t; \vec{\underline{\Theta}}, \vec{\underline{L}}, \vec{\underline{\alpha}})\right), \]
where $\vec{\Theta}$ is the solution to \eqref{linear-theta} with \eqref{map-Fb} obtained by Proposition \ref{prop:exist-linear}, and $\vec{L}, \vec{\alpha}$ are defined by 
\begin{align*}
&L^\J(t; \vec{\underline{\Theta}}, \vec{\underline{L}}, \vec{\underline{\alpha}}) := L^\J_0 + \int_0^t \left( g^\J(\vec{\underline{\Theta}}, \vec{\underline{\alpha}})(\tilde{t})- \frac{\sigma(\Delta^\J \underline{\alpha}(\tilde{t}))}{\underline{L}^\J(\tilde{t})} \int_0^1 (\partial_x \underline{\Theta}^\J(x,\tilde{t}))^2 \; dx \right) \; d\tilde{t} \\
&\alpha^\J(t; \vec{\underline{\Theta}}, \vec{\underline{L}}, \vec{\underline{\alpha}}) = \alpha^\J_0 - \gamma \int_0^t \left( \pa \sigma (\Delta^\JJ \underline{\alpha}(\tilde{t})) \underline{L}^\JJ(\tilde{t}) - \pa \sigma (\Delta^\J \underline{\alpha}(\tilde{t})) \underline{L}^\J(\tilde{t}) \right) \; d\tilde{t}. 
\end{align*}
Here, $g^\J$ is the function defined by \eqref{functions-system}. 
\end{definition}

We will show that $\Phi$ is contraction map on $X_T^M$ for sufficiently small $T>0$ and sufficiently large $M>0$. 
We then obtain a unique fixed point of $\Phi$ in $X_T^M$ and it is easily seen that the fixed point is a solution to \eqref{system-theta} due to the definition of $\Phi$. 

\begin{lem}
There exists $T>0$ and $M>0$ such that the map $\Phi$ defined in Definition \ref{def-con-map} satisfies $\Phi(X_T^M) \subset X_T^M$ and is contraction. 
\end{lem}

\begin{proof}
For any $T>0$, $M> \Cr{c-short2}$ and $(\vec{\underline{\Theta}}, \vec{\underline{L}}, \vec{\underline{\alpha}}) \in X_T^M$, where $\Cr{c-short2}$ is the constant in Lemma \ref{lem:compatibility}, the function $\Phi(\vec{\underline{\Theta}}, \vec{\underline{L}}, \vec{\underline{\alpha}})$ obviously satisfies the condition (iii) in Lemma \ref{lem:compatibility} since the first and last two equalities in \eqref{system-theta} replaced $\Theta^\J, L^\J$ and $\alpha^\J$ on the right hand side by $\underline{\Theta}^\J, \underline{L}^\J$ and $\underline{\alpha}^\J$, respectively, and $\underline{\Theta}^\J(\cdot,0) = \Theta^\J_0, \underline{L}^\J(0) = L^\J_0$ and $\underline{\alpha}^\J(0) = \alpha^\J_0$. 
We thus prove that $\Phi(\vec{\underline{\Theta}}, \vec{\underline{L}}, \vec{\underline{\alpha}})$ satisfies the condition (ii) in Lemma \ref{lem:compatibility} for any $(\vec{\underline{\Theta}}, \vec{\underline{L}}, \vec{\underline{\alpha}}) \in X_T^M$ if $T>0$ is sufficiently small and $M> \Cr{c-short2}$ is sufficiently large to prove that $\Phi(X_T^M) \subset X_T^M$ and $\Phi$ is contraction. 

To show $\Phi (X_T^M) \subset X_T ^M$, we first prove that there exist $M>\Cr{c-short2}$ and $T>0$ such that
\[ \max_{j=1,2,3} \|\Theta^\J \|_{C_{x,t}^{k+\beta, \frac{k+\beta}{2}} ([0,1]\times [0,T])} \leq \frac{M}{3}. \]
Repeating the same argument as in \eqref{maximum-reg},
we need only prove
\begin{equation}
\Cr{c-short} \Bigg(\sum_{j=1}^3 \Big(\|F^\J\|_{C^{(k-2)+\beta, \frac{(k-2)+\beta}{2}}_{x,t}} + \|\Theta_0^\J\|_{C^{k+\beta}_x}\Big) 
+ \|b_1\|_{C^{\frac{k+\beta}{2}}_t} + \|b_2\|_{C^{\frac{k+\beta}{2}}_t} + \|b_3\|_{C^{\frac{(k-1)+\beta}{2}}_t} \Bigg)
\le \frac{M}{3}.
\label{eq:2.20}
\end{equation}
Now we prove \eqref{eq:2.20} for $k=2$. We compute that
\begin{equation*}
\begin{split}
&\left\| \left(\frac{\sigma(\Delta^\J \underline{\alpha})}{(\underline{L}^\J)^2} - \frac{\sigma(\Delta^\J \alpha_0)}{(L^\J_0)^2}\right) \px^2 \underline{\Theta}^\J 
\right\|_{C^{\beta,\frac{\beta}{2}} _{x,t} ([0,1]\times [0,T])} \\
\leq & \,
M \| \sigma(\Delta^\J \underline{\alpha}) \|_{C^{\frac{\beta}{2}} _{t} ([0,T])}
\left \| \frac{1}{(\underline{L} ^\J)^2} -\frac{1}{(L_0 ^\J)^2}\right \|_{C^{\frac{\beta}{2}} _{t} ([0,T])} 
+ \frac{M}{\Cr{c-initial-min} ^2} 
\| \sigma(\Delta^\J \underline{\alpha}) - \sigma(\Delta^\J \alpha _0) \|_{C^{\frac{\beta}{2}} _{t} ([0,T])}\\
\leq & \,
M C \Cr{c-initial-max}
\left \| \frac{1}{(\underline{L} ^\J)^2} -\frac{1}{(L_0 ^\J)^2}\right \|_{C^{\frac{\beta}{2}} _{t} ([0,T])} 
+ \frac{M}{\Cr{c-initial-min} ^2} 
\| \sigma(\Delta^\J \underline{\alpha}) - \sigma(\Delta^\J \alpha _0) \|_{C^{\frac{\beta}{2}} _{t} ([0,T])},
\end{split}
\end{equation*}
where $C>0$ is a constant depending only on $\sigma$ and $M$ (if $\sigma (x)\to \infty$ as $|x|\to \infty$, 
$C$ will depend on $M$).
By
\begin{equation*}
\begin{split}
\left | \frac{1}{(\underline{L} ^\J)^2} -\frac{1}{(L_0 ^\J)^2}\right |
\leq 
\frac{4}{\Cr{c-initial-min} ^4} | \underline{L} ^\J + L ^\J _0 | | \underline{L} ^\J - L ^\J _0 | 
\leq
\frac{4 (\Cr{c-initial-max}+M)}{\Cr{c-initial-min} ^4} | \underline{L} ^\J - L ^\J _0 |
\end{split}
\end{equation*}
and
\begin{equation*}
\begin{split}
\left | \frac{1}{(\underline{L} ^\J (s) )^2} - \frac{1}{(\underline{L} ^\J (t) )^2} \right |
\leq 
\frac{8}{\Cr{c-initial-min} ^4} | \underline{L} ^\J (s) + \underline{L} ^\J (t) | 
| \underline{L} ^\J (s) - \underline{L} ^\J (t) | 
\leq
\frac{16 M }{\Cr{c-initial-min} ^4} | \underline{L} ^\J (s) - \underline{L} ^\J (t) |,
\end{split}
\end{equation*}
we have
\begin{equation}\label{eq:2.23}
\begin{split}
& \left \| \frac{1}{(\underline{L} ^\J)^2} -\frac{1}{(L_0 ^\J)^2}\right \|_{C^{\frac{\beta}{2}} _{t} ([0,T])} \\
\leq & \, 
\frac{20 (\Cr{c-initial-max}+M)}{\Cr{c-initial-min} ^4} 
\left\{ \| \underline{L} ^\J - L ^\J _0 \|_{C^0 ([0,T])}
+ \sup _{t,s \in [0,T], t\not=s}\frac{ |\underline{L} ^\J (s) - \underline{L} ^\J (t)| }{|s-t|^{\frac{\beta}{2}}} \right\}\\
\leq & \,
\frac{20 (\Cr{c-initial-max}+M)}{\Cr{c-initial-min} ^4} 
\left\{ \| \underline{L} ^\J \|_{C^{\frac{\beta'}{2}} ([0,T])} T^{\frac{\beta'}{2}}
+ \sup _{t,s \in [0,T], t\not=s}\frac{ |\underline{L} ^\J (s) - \underline{L} ^\J (t)| }{|s-t|^{\frac{\beta'}{2}}} 
|s-t|^{\frac{\beta' -\beta}{2}}\right\} \\
\leq & \,
\frac{20 (\Cr{c-initial-max}+M)}{\Cr{c-initial-min} ^4} \| \underline{L} ^\J \|_{C^{\frac{\beta'}{2}} ([0,T])}
\left\{  T^{\frac{\beta'}{2}}
+  T^{\frac{\beta' -\beta}{2}}\right\}.
\end{split}
\end{equation}
Similarly, there exists $C=C(\Cr{c-initial-max} , \Cr{c-initial-min}, M,\sigma)>0$ such that
\[
\| \sigma(\Delta^\J \underline{\alpha}) - \sigma(\Delta^\J \alpha _0) \|_{C^{\frac{\beta}{2}} _{t} ([0,T])}
\leq C
\left\{  T^{\frac{\beta'}{2}}
+  T^{\frac{\beta' -\beta}{2}}\right\}.
\]
Therefore
\begin{equation*}
\begin{split}
\left\| \left(\frac{\sigma(\Delta^\J \underline{\alpha})}{(\underline{L}^\J)^2} - \frac{\sigma(\Delta^\J \alpha_0)}{(L^\J_0)^2}\right) \px^2 \underline{\Theta}^\J 
\right\|_{C^{\beta,\frac{\beta}{2}} _{x,t} ([0,1]\times [0,T])} 
\leq 
C
\left\{  T^{\frac{\beta'}{2}}
+  T^{\frac{\beta' -\beta}{2}}\right\},
\end{split}
\end{equation*}
where $C>0$ depends only on $\Cr{c-initial-max}$, $\Cr{c-initial-min}$, $M$, and $\sigma$.
By repeating the same argument for 
$f^\J$ and $b_i$, there exist
$C = C (\Cr{c-initial-max}, \Cr{c-initial-min}, M,\sigma)>0$, 
$C'=C' (\Cr{c-initial-max}, \Cr{c-initial-min}, \sigma)>0$,
and $\gamma = \gamma (\beta, \beta')>0$ such that
\[
\sum_{j=1}^3 \|F^\J\|_{C^{\beta, \frac{\beta}{2}}_{x,t}}  
+ \|b_1\|_{C^{\frac{\beta}{2}}_t} + \|b_2\|_{C^{\frac{\beta}{2}}_t} + \|b_3\|_{C^{\frac{\beta}{2}}_t} 
\leq C T^{\gamma} +C'.
\]
Note that $C'$ is independent of $M$. For the readers' convenience, 
we explain why $C'$ is necessary and is independent of $M$. 
For $t\in(0,T]$, in order to estimate the $C^{\beta}$ norm of $\partial _x \Theta ^\J (x,t) $ with respect to $x$,
we have
\begin{equation}\label{eq:2.22}
\begin{split}
& \frac{ | \partial _x \Theta ^\J (x,t) -\partial _x \Theta ^\J (y,t)| }{|x-y|^\beta} \\
\leq & \,  
\frac{ | \partial _x \Theta ^\J (x,0) -\partial _x \Theta ^\J (y,0)| }{|x-y|^\beta}
+
\frac{ | \partial _x \Theta ^\J (x,t) -\partial _x \Theta ^\J (x,0)| }{|x-y|^\beta}
+
\frac{ | \partial _x \Theta ^\J (y,t) -\partial _x \Theta ^\J (y,0)| }{|x-y|^\beta}\\
\leq & \, 
\Cr{c-initial-max} 
+
\frac{ | \partial _x \Theta ^\J (x,t) -\partial _x \Theta ^\J (x,0)| }{|x-y|^\beta}
+
\frac{ | \partial _x \Theta ^\J (y,t) -\partial _x \Theta ^\J (y,0)| }{|x-y|^\beta}
\end{split}
\end{equation}
for any $0\leq x<y\leq 1$. In the case of $y-x > \sqrt{t}$, we have
\[
\frac{ | \partial _x \Theta ^\J (x,t) -\partial _x \Theta ^\J (x,0)| }{|x-y|^\beta}
\leq M t^{\frac{1+\beta}{2} - \frac{\beta}{2}} = M t^{\frac{1}{2}},
\]
where we used
$
\sup _{0\leq t<s\leq T, x \in [0,1]}
\frac{ | \partial _x \Theta ^\J (x,s) -\partial _x \Theta ^\J (x,t)| }{|s-t|^{\frac{1+\beta}{2}}}
\leq M
$.
On the other hand, if $y-x \leq \sqrt{t}$,
\begin{equation*}
\begin{split}
\frac{ | \partial _x \Theta ^\J (x,t) -\partial _x \Theta ^\J (y,t)| }{|x-y|^\beta} 
\leq M |x-y| ^{1-\beta} \leq M t^{\frac{1-\beta}{2}},
\end{split}
\end{equation*}
where we used $\| \partial _{xx} \Theta ^\J \| _{C^0 _{x,t}} \leq M$.
Thus we obtain
\begin{equation*}
\begin{split}
\sup _{x,y \in [0,1], x\not =y}
\frac{ | \partial _x \Theta ^\J (x,t) -\partial _x \Theta ^\J (y,t)| }{|x-y|^\beta} 
\leq M (t^{\frac12} + t^{\frac{1-\beta}{2}}) + \Cr{c-initial-max} .
\end{split}
\end{equation*}
The second term of right hand side corresponds to $C'$ above.
For nonlinear terms, similar estimates hold.
Note that for sufficiently small $T>0$ 
there exists $C>0$ depending only on $M$ such that
$C<1 -c^{(1)}c^{(2)}c^{(3)}\leq 1$. 

Therefore, we obtain \eqref{eq:2.20} with $k=2$ for sufficiently small $T>0$ and sufficiently large $M>0$.
Repeating same argument above, we may obtain \eqref{eq:2.20} and
\[
\max_{j=1,2,3} \|\Theta^\J \|_{C_{x,t}^{k+\beta}} + \max_{j=1,2,3} \|L^\J\|_{C_t^{\frac{k+\beta'}{2}}} + \max_{j=1,2,3}\|\alpha^\J\|_{C_t^{\frac{k+\beta'}{2}}} \le M
\]
for the case of $k \geq 3$.
In addition, there exists $C>0$ depending only on $M$ such that
\[
\sup_{\tilde t \in [0,T)}
\left|
g^\J(\vec{\underline{\Theta}}, \vec{\underline{\alpha}})(\tilde{t})- \frac{\sigma(\Delta^\J \underline{\alpha}(\tilde{t}))}{\underline{L}^\J(\tilde{t})} \int_0^1 (\partial_x \underline{\Theta}^\J(x,\tilde{t}))^2 \; dx 
\right| \leq C.
\]
Therefore $\inf _{t \in [0,T]} L^\J (t) \geq \frac{\Cr{c-initial-min}}{2}$ for sufficiently small $T>0$.
Hence we obtain $\Phi (X_T^M) \subset X_T ^M$.

Finally we show that $\Phi$ is a contraction mapping.
Assume that $(\vec{\underline{\Theta _i}} , \vec{\underline{L_i}} , \vec{\underline{\alpha_i}}) \in X_T ^M$
and 
$
(\vec{\Theta _i} , \vec{L_i} , \vec{\alpha_i})
=
\Phi (\vec{\underline{\Theta _i}} , \vec{\underline{L_i}} , \vec{\underline{\alpha_i}})
$
for $i=1,2$. 
Set $
(\vec{\underline{\Theta }} , \vec{\underline{L}} , \vec{\underline{\alpha}})
:=
(\vec{\underline{\Theta _1}} , \vec{\underline{L_1}} , \vec{\underline{\alpha_1}}) 
-(\vec{\underline{\Theta _2}} , \vec{\underline{L_2}} , \vec{\underline{\alpha_2}})
$.
Then 
$
(\vec{\Theta } , \vec{L} , \vec{\alpha}):=
(\vec{\Theta _1} , \vec{L_1} , \vec{\alpha_1}) - (\vec{\Theta _2} , \vec{L_2} , \vec{\alpha_2})
$
is a solution of \eqref{linear-theta} with
$F^\J = F^\J (\vec{\Theta _1} , \vec{L_1} , \vec{\alpha_1}) - 
F^\J(\vec{\Theta _2} , \vec{L_2} , \vec{\alpha_2})$ and
$b_j = b_j (\vec{\Theta _1} , \vec{L_1} , \vec{\alpha_1}) - 
b_j (\vec{\Theta _2} , \vec{L_2} , \vec{\alpha_2})$
for $j=1,2,3$.
We need only prove that for sufficiently small $T>0$ we have
\begin{equation}\label{eq:2.21}
\begin{split}
\max_{j=1,2,3} \|\Theta^\J \|_{C_{x,t}^{k+\beta}} + \max_{j=1,2,3} \|L^\J\|_{C_t^{\frac{k+\beta'}{2}}} + \max_{j=1,2,3}\|\alpha^\J\|_{C_t^{\frac{k+\beta'}{2}}} 
\le 
\frac{M'}{2},
\end{split}
\end{equation}
where 
\[
M'=\max_{j=1,2,3} \| \underline{\Theta}^\J \|_{C_{x,t}^{k+\beta}} 
+ \max_{j=1,2,3} \|\underline{L}^\J\|_{C_t^{\frac{k+\beta'}{2}}}
+ \max_{j=1,2,3}\|\underline{\alpha}^\J\|_{C_t^{\frac{k+\beta'}{2}}}.
\]
By \eqref{maximum-reg}, we have
\begin{equation*}
\begin{aligned}
\sum_{j=1}^3 \|\Theta^\J\|_{C^{k+\beta, \frac{k+\beta}{2}}_{x,t}} \le \Cr{c-short} \Bigg(\sum_{j=1}^3 
\|F^\J\|_{C^{(k-2)+\beta, \frac{(k-2)+\beta}{2}}_{x,t}} 
+ \|b_1\|_{C^{\frac{k+\beta}{2}}_t} + \|b_2\|_{C^{\frac{k+\beta}{2}}_t} + \|b_3\|_{C^{\frac{(k-1)+\beta}{2}}_t} \Bigg),
\end{aligned}
\end{equation*}
where we used 
$(\vec{\Theta } , \vec{L} , \vec{\alpha}) |_{t=0} 
=\vec{0}$.
By 
$
(\vec{\underline{\Theta }} , \vec{\underline{L}} , \vec{\underline{\alpha}})|_{t=0}
=\vec{0},
$
there exists $\gamma >0$ such that
\begin{equation}\label{eq:2.24}
\sum_{j=1}^3 
\|F^\J\|_{C^{(k-2)+\beta, \frac{(k-2)+\beta}{2}}_{x,t}} 
+ \|b_1\|_{C^{\frac{k+\beta}{2}}_t} + \|b_2\|_{C^{\frac{k+\beta}{2}}_t} + \|b_3\|_{C^{\frac{(k-1)+\beta}{2}}_t}
\leq \Cl[c]{const:lem2.12} M' T^\gamma,
\end{equation}
where $\Cr{const:lem2.12}= \Cr{const:lem2.12} (\Cr{c-initial-max} , \Cr{c-initial-min}, M ,\sigma)>0$.
For example, in the case of $k=2$,
\begin{equation*}
\begin{split}
\|F^\J\|_{C^{\beta, \frac{\beta}{2}}_{x,t}}
= & \,
\| F^\J (\vec{\Theta _1} , \vec{L_1} , \vec{\alpha_1}) - 
F^\J(\vec{\Theta _2} , \vec{L_2} , \vec{\alpha_2})
\|_{C^{\beta, \frac{\beta}{2}}_{x,t}} \\
\leq & \,
\| F^\J (\vec{\Theta _1} , \vec{L_1} , \vec{\alpha_1}) - 
F^\J(\vec{\Theta _2} , \vec{L_1} , \vec{\alpha_1})
\|_{C^{\beta, \frac{\beta}{2}}_{x,t}}
+
\| F^\J (\vec{\Theta _2} , \vec{L_1} , \vec{\alpha_1}) - 
F^\J(\vec{\Theta _2} , \vec{L_2} , \vec{\alpha_1})
\|_{C^{\beta, \frac{\beta}{2}}_{x,t}}\\
& \, +
\| F^\J (\vec{\Theta _2} , \vec{L_2} , \vec{\alpha_1}) - 
F^\J(\vec{\Theta _2} , \vec{L_2} , \vec{\alpha_2})
\|_{C^{\beta, \frac{\beta}{2}}_{x,t}}.
\end{split}
\end{equation*}
Now we calculate 
\begin{equation*}
\begin{split}
\mathcal{F} := \left(\frac{\sigma(\Delta^\J \underline{\alpha_1} )}{(\underline{L_1}^\J )^2} 
- \frac{\sigma(\Delta^\J \alpha_0)}{(L^\J_0)^2}\right) \px^2 \underline{\Theta_1}^\J 
-
\left(\frac{\sigma(\Delta^\J \underline{\alpha_1} )}{(\underline{L_1}^\J )^2} 
- \frac{\sigma(\Delta^\J \alpha_0)}{(L^\J_0)^2}\right) \px^2 \underline{\Theta_2}^\J 
\end{split}
\end{equation*}
contained in $F^\J (\vec{\Theta _1} , \vec{L_1} , \vec{\alpha_1}) - 
F^\J(\vec{\Theta _2} , \vec{L_1} , \vec{\alpha_1})$. By an argument similar to that in 
\eqref{eq:2.23},
we compute
\begin{equation*}
\begin{split}
\|\mathcal{F} \|_{C^{\beta, \frac{\beta}{2}}_{x,t}}
= & \, 
\left \|
\left(\frac{\sigma(\Delta^\J \underline{\alpha_1} )}{(\underline{L_1}^\J )^2} 
- \frac{\sigma(\Delta^\J \alpha_0)}{(L^\J_0)^2}\right) 
(\px^2 \underline{\Theta_1}^\J -
\px^2 \underline{\Theta_2}^\J)
\right \|_{C^{\beta, \frac{\beta}{2}}_{x,t}} \\
\leq & \,
\left \| 
\frac{\sigma(\Delta^\J \underline{\alpha_1} )}{(\underline{L_1}^\J )^2} 
- \frac{\sigma(\Delta^\J \alpha_0)}{(L^\J_0)^2} 
\right \|_{C^{\frac{\beta}{2}}_{t}}
\left \|
\px^2 \underline{\Theta}^\J
\right \|_{C^{\beta, \frac{\beta}{2}}_{x,t}}
\leq C M' \left\{  T^{\frac{\beta'}{2}}
+  T^{\frac{\beta' -\beta}{2}}\right\}.
\end{split}
\end{equation*}
Similarly, we obtain \eqref{eq:2.24}. Then
we choose $T>0$ such that 
$
\Cr{const:lem2.12} T^\gamma
\leq \frac{M'}{10 \Cr{c-short}}
$.
Therefore we have
$
\max_{j=1,2,3} \|\Theta^\J\|_{C^{k+\beta, \frac{k+\beta}{2}}_{x,t}} \le \frac{M'}{10}.
$
Similarly we may obtain \eqref{eq:2.21}.
\end{proof}

The following existence of a solution to the system \eqref{system-theta} immediately follows since we constructed the contraction map associated to the system. 

\begin{prop}\label{prop:short-time-system}
Let $k \in \mathbb{N}$ with $k \ge 2$ and $\beta \in (0,1)$. 
Assume $\vec{\Theta}_0 \in (C_x^{k+\beta}([0,1]))^3, \vec{L}_0 \in \{(0,\infty)\}^3$ and $\vec{\alpha}_0 \in \mathbb{R}^3$ satisfy the compatibility condition of order $k$ for \eqref{system-theta} and $\Theta_0^\JJ(1) - \Theta_0^\J(1) \in (0,\pi)$ for $j=1,2,3$. 
Then, there exist $T>0$ such that the system \eqref{system-theta} have a unique solution $(\vec{\Theta}, \vec{L}, \vec{\alpha}) \in (C_{x,t}^{k+\beta, \frac{k+\beta}{2}}([0,1]\times[0,T]))^3 \times (C_t^{\frac{k+\beta}{2}}([0,T]))^3 \times (C_t^{\frac{k+\beta}{2}}([0,T]))^3$ starting from $(\vec{\Theta}_0, \vec{L}_0, \vec{\alpha}_0)$. 
\end{prop}

The system \eqref{system-theta} was derived from the geometric flow \eqref{eq-curve}--\eqref{bc-boundary} and, conversely, the geometric flow can be constructed from the solution to the system. 
We thus obtain the existence theory for the geometric flow as in Theorem \ref{thm:exists-flow}. 

\begin{proof}[Proof of Theorem \ref{thm:exists-flow}]
Let $L^\J_0$ be the length of $\Gamma^\J_0$ and parametrize $\Gamma^\J_0$ by $x= s/L^\J_0$, where $s$ is the arc-length of $\Gamma^\J_0$. 
Define the angle function $\Theta^\J_0$ of $\Gamma^\J_0$ as in \eqref{def-theta}. 
We then easily see that $\Theta^\J_0$ is of class $C_x^{(k-1)+\beta}([0,1])$, $\Theta^\JJ_0(1) - \Theta^\J_0(1)$ is positive and less than $\pi$ and $(\vec{\Theta}_0, \vec{L}_0, \vec{\alpha}_0)$ satisfies the compatibility condition of order $k-1$ for \eqref{system-theta}. 
Therefore, for any $\beta' \in (\beta, 1)$ and some $T>0$, the unique solution $(\vec{\Theta}, \vec{L}, \vec{\alpha})\in (C_{x,t}^{(k-1)+\beta, \frac{(k-1)+\beta}{2}}([0,1]\times[0,T]))^3 \times (C_t^{\frac{(k-1)+\beta'}{2}}([0,T]))^3 \times (C_t^{\frac{(k-1)+\beta'}{2}}([0,T]))^3$ to \eqref{system-theta} starting from $(\vec{\Theta}_0, \vec{L}_0, \vec{\alpha}_0)$ can be obtained due to Proposition \ref{prop:short-time-system}. 
Letting 
\[ \xi^\J(x,t) := \left(\int_0^x L^\J(t) \cos \Theta^\J(\tilde{x},t) \; d\tilde{x}, \int_0^x L^\J(t) \sin \Theta^\J(\tilde{x},t) \; d\tilde{x}\right) + P^\J \]
for $x \in [0,1]$ and $t \ge 0$, where $P^\J$ is the fixed point in \eqref{bc-boundary}, we can also easily see that the curves $\Gamma^\J_t := \xi^\J([0,1], t)$ start from $\Gamma^\J_0$. 
We now prove that $\{\Gamma^\J_t\}$ and $\vec{\alpha}$ satisfy the geometric flow equations \eqref{eq-curve}--\eqref{bc-boundary} until the time $T$. 
Since $f^\J$ defined by \eqref{functions-system} can be re-written as, by means of \eqref{system-theta},  
\[ f^\J = \left(\frac{\sigma(\Delta^\J \alpha)}{(L^\J)^2} \int_0^x (\px \Theta^\J(\tilde{x}, t))^2 \; d\tilde{x} + \frac{x}{L^\J} \pt L^\J \right) \px \Theta^\J, \]
we have by \eqref{system-theta} and integration by parts
\begin{align*}
&\; \pt \int_0^x L^\J(t) \cos \Theta^\J(\tilde{x}, t) \; d\tilde{x} \\
=&\; \int_0^x \pt L^\J \cos \Theta^\J - \left\{\frac{\sigma(\Delta^\J \alpha)}{L^\J} \px^2 \Theta + \left(\frac{\sigma(\Delta^\J \alpha)}{L^\J} \int_0^{\tilde{x}} (\px \Theta^\J)^2 \; d\hat{x} + \tilde{x} \pt L^\J \right) \px \Theta^\J\right\} \sin \Theta^\J \; d\tilde{x} \\
=&\; \int_0^x \pt L^\J \px(\tilde{x} \cos \Theta^\J) + \frac{\sigma(\Delta^\J \alpha)}{L^\J} \px \left(\int_0^{\tilde{x}} (\px \Theta^\J)^2 \; d\hat{x} \cos \Theta^\J\right) \; d\tilde{x} - \frac{\sigma(\Delta^\J \alpha)}{L^\J} \px \Theta^\J \sin \Theta^\J \\
=&\; - \frac{\sigma(\Delta^\J \alpha)}{L^\J} \px \Theta^\J \sin \Theta^\J + \left(\frac{\sigma(\Delta^\J \alpha)}{L^\J} \int_0^x (\px \Theta^\J)^2 \; d\hat{x} + x \pt L^\J\right) \cos \Theta^\J. 
\end{align*}
Since $\px \Theta^\J/ L^\J$ is the curvature $\kappa_t^\J$ of $\Gamma^\J_t$ and \eqref{def-theta} holds for $\Gamma^\J_t$, we can obtain similarly 
\[ \pt \xi^\J = \sigma(\Delta^\J \alpha) \kappa^\J_t \nu^\J_t + \lambda_t^\J \tau_t^\J, \]
where 
\begin{equation}\label{const-lambda}
\lambda^\J_t = \left(\frac{\sigma(\Delta^\J \alpha)}{L^\J} \int_0^x (\px \Theta^\J)^2 \; d\hat{x} + x \pt L^\J\right) 
\end{equation}
and $\tau_t^\J, \nu_t^\J$ are the unit tangent and normal vector of $\Gamma^\J_t$ respectively, which show that \eqref{eq-curve} is satisfied. 
The equalities \eqref{eq-alpha}, \eqref{bc-boundary} hold obviously. 
Note that the third and forth equalities in \eqref{system-theta} is equivalent to 
\begin{equation}\label{const-angle1} 
\sum_{j=1}^3 \sigma(\Delta^\J \alpha) \tau_t^\J = \sum_{j=1}^3 \sigma(\Delta^\J \alpha) (\cos\Theta, \sin \Theta) = 0 \quad \text{at} \; \; x=1, 
\end{equation}
and thus \eqref{bc-angle} is satisfied if \eqref{bc-concurrency} holds. 
We now thus prove \eqref{bc-concurrency}. 
By means of the form \eqref{const-lambda} of the tangent velocity and the differential equation of $L^\J$ in \eqref{system-theta}, we can see that $\lambda_t^\J$ satisfies \eqref{eq-lambda-V}.
Hereafter, we use $V^\J_t$ as the normal velocity of $\Gamma_t^\J$ and thus 
\[V^\J_t = \frac{\sigma(\Delta^\J \alpha)}{L^\J(t)} \px \Theta^\J = \sigma(\Delta^\J \alpha) \ps \Theta^\J = \sigma(\Delta^\J \alpha) \kappa_t^\J, \]
and also let $c^\J(t) = \cos(\Theta^\JJ(1,t)-\Theta^\J(1,t))$ and $s^\J(t) = \sin(\Theta^\JJ(1,t) - \Theta^\J(1,t))$ for simplicity. 
We then have by the fifth equality in \eqref{system-theta} and the addition theorem of trigonometric functions
\begin{align} 
&\sum_{j=1}^3 \sigma(\Delta^\J \alpha) V^\J_t = 0, \label{const-sum-V}\\
&\begin{aligned}c^\K(t) =&\; \cos\left((\Theta^\KK(1,t) - \Theta^\KKK(1,t)) + (\Theta^\KKK(1,t) - \Theta^\K(1,t))\right) \\
=&\;  c^\KK c^\KKK - s^\KK s^\KKK \quad \text{for} \; \; k \in \{1,2,3\}. \end{aligned} \label{const-add}
\end{align}
Furthermore, taking inner product of \eqref{const-angle1} and $\nu_t^\K$ at $x=1$, we have 
\begin{equation}\label{const-s1}
\sigma(\Delta^\KK \alpha) s^\K = \sigma(\Delta^\KKK \alpha) s^\KKK \quad \text{for} \; \; k \in \{1,2,3\}. 
\end{equation}
We now prove, for $j \in \{1,2,3\}$ and $t \ge 0$, 
\begin{equation}\label{const-wantprove} 
\pt \xi^\J - \pt \xi^\JJ = (V^\J_t \nu^\J_t + \lambda_t^\J \tau^\J_t) - (V^\JJ_t \nu^\JJ_t + \lambda^\JJ_t \tau^\JJ_t) = 0 \quad \text{at} \; \; x=1. 
\end{equation}
Taking inner product with $\tau^\J_t$, we have $\lambda^\J_t = -s^\J V^\JJ_t + c^\J \lambda_t^\JJ$ which obviously holds since \eqref{eq-lambda-V} is equivalent to the identities for $j \in \{1,2,3\}$. 
We next take inner product of the left hand side of \eqref{const-wantprove} and $\nu^\J_t$ at $x=1$ to obtain 
\begin{align*}
\langle \pt \xi^\J - \pt \xi^\JJ, \nu^\J_t \rangle =&\; V^\J_t - c^\J V^\JJ_t - s^\J \lambda^\JJ_t \\
=&\; \frac{1}{1-c^\J c^\JJ c^\JJJ} \Big( (1-c^\J c^\JJ c^\JJJ +s^\J c^\JJ s^\JJJ)V^\J \\
&\; \qquad + (-c^\J(1-c^\J c^\JJ c^\JJJ) + (s^\J)^2 c^\JJ c^\JJJ)V^\JJ + s^\J s^\JJ V^\JJJ \Big) 
\end{align*}
at $x=1$. 
Here, the following identity, which follows from \eqref{eq-lambda-V}, have been applied: 
\[ \lambda^\JJ_t = \frac{-1}{1-c^\J c^\JJ c^\JJJ}(c^\JJ s^\JJJ V_t^\J + s^\J c^\JJ c^\JJJ V^\JJ_t + s^\JJ V_t^\JJJ) \quad \text{at} \; \; x=1.\]
Since \eqref{const-add} implies 
\begin{align*}
&1-c^\J c^\JJ c^\JJJ +s^\J c^\JJ s^\JJJ = 1 - (c^\JJ)^2 = (s^\JJ)^2, \\
&-c^\J(1-c^\J c^\JJ c^\JJJ) + (s^\J)^2 c^\JJ c^\JJJ = c^\JJ c^\JJJ - c^\J = s^\JJ s^\JJJ, 
\end{align*}
by means of \eqref{const-sum-V} and \eqref{const-s1}, we have
\begin{align*} 
\langle \pt \xi^\J - \pt \xi^\JJ, \nu^\J_t \rangle =&\; \frac{1}{1-c^\J c^\JJ c^\JJJ} ( (s^\JJ)^2 V^\J_t + s^\JJ s^\JJJ V^\JJ_t + s^\J s^\JJ V^\JJJ_t) \\
=&\; \frac{s^\J s^\JJ}{\sigma(\Delta^\JJJ \alpha)(1- c^\J c^\JJ c^\JJJ)} \sum_{k=1}^3 \sigma(\Delta^\K \alpha) V^\K_t = 0 \quad \text{at} \; \; x=1. 
\end{align*}
Therefore, $\pt \xi^\J - \pt \xi^\JJ$ is zero at $x=1$ in each direction $\nu^\J_t$ and $\tau^\J_t$, which implies \eqref{const-wantprove} since $\nu^\J_t$ and $\tau^\J_t$ are linearly independent.
Since initial triod satisfies \eqref{bc-concurrency}, the identity \eqref{const-wantprove} show that the preservation of \eqref{bc-concurrency}. 

Next, assume a family of $\tilde{\Gamma}^\J_t$ and $\tilde{\alpha}^\J$, starting from $\Gamma^\J_0$ and $\alpha^\J_0$ respectively, is another solution to \eqref{eq-curve}--\eqref{bc-boundary} to prove the uniqueness of the geometric flow. 
We then see that the family of the angle function $\tilde{\Theta}^\J$ of $\tilde{\Gamma}^\J_t$, the length $\tilde{L}^\J$ of $\tilde{\Gamma}^\J_t$ and the orientation parameter $\tilde{\alpha}^\J$ satisfies \eqref{system-theta} with the variables $\tilde{x}=\tilde{s}/\tilde{L}^\J$ on $\tilde{\Gamma}^\J_t$, where $\tilde{s}$ is the arc-length of $\tilde{\Gamma}^\J_t$. 
Due to the uniqueness result for \eqref{system-theta}, we obtain $\tilde{\Theta}^\J \equiv \Theta^\J$, $\tilde{L}^\J \equiv L^\J$ and $\tilde{\alpha}^\J \equiv \alpha^\J$, which implies coincidence of two geometric flows. 

If $\Gamma^\J_0$ is smooth and the family of $\{\Gamma_0^\J\}_{j \in \{1,2,3\}}$ and $\{\alpha_0^\J\}_{j \in \{1,2,3\}}$ satisfies the compatibility condition of any order for the geometric flow \eqref{eq-curve}--\eqref{bc-boundary}, for any $k \in \mathbb{N}$ with $k \ge 2$, we can apply Proposition \ref{prop:short-time-system} with $T=T_k > 0$, which depends on $k$, and thus the angle function $\Theta^\J$ of $\Gamma^\J_t$ satisfies 
\begin{equation}\label{reg-k-theta} 
\Theta^\J \in C^{k+\beta, \frac{k+\beta}{2}}_{x,t}([0,1] \times [0, T_k]) \quad \text{for} \; \; j \in \{1,2,3\}, \; \; k \in \mathbb{N} \; \; \text{with} \; \; k \ge 2. 
\end{equation}
Note that, by means of the differential equations of $L^\J$ and $\alpha^\J$ in \eqref{system-theta}, we also see that 
\begin{equation}\label{reg-k-aL}
L^\J \in C^{\frac{k+1/2 + \beta}{2}}_t ([0,T_k]), \quad \alpha^\J \in C^{\frac{k+3/2+\beta}{2}}_t([0,T_k]), \quad \text{for} \; \; j \in \{1,2,3\}, \; \; k \in \mathbb{N} \; \; \text{with} \; \; k \ge 2. 
\end{equation}
Therefore, it is sufficient to prove that there exists $T'>0$ such that $T_k \ge T'$ for any $k$. 
Letting $v^\J := \pt \Theta^\J$, it follows from the regularity results \eqref{reg-k-theta}, \eqref{reg-k-aL} and the differential equations of $\Theta^\J$ in \eqref{system-theta} that $v$ satisfies
\begin{equation}\label{system-v} \begin{cases}
\pt v^\J = \frac{\sigma(\Delta^\J \alpha)}{(L^\J)^2} \px^2 v + \tilde{F}^\J, & (x,t) \in (0,1) \times (0,T_k], \; \; j=1,2,3, \\
\px v^\J(0,t) = 0, & t \in (0,T_k], \; \; j =1,2,3, \\
\sum_{j=1}^3 \left(\sigma(\Delta^\J \alpha(t)) \sin \Theta^\J(1,t)\right)v^\J(1,t) = \tilde{b}_1(t), & t \in (0,T_k], \\
\sum_{j=1}^3 \left(\sigma(\Delta^\J \alpha(t)) \cos \Theta^\J(1,t)\right)v^\J(1,t) = \tilde{b}_2(t), & t \in (0,T_k], \\
\sum_{j=1}^3 \frac{(\sigma(\Delta^\J \alpha(t)))^2}{L^\J(t)} \px v^\J(1,t) = \tilde{b}_3(t), & t \in (0,T_k], 
\end{cases} \end{equation}
where $\tilde{F}^\J, \tilde{b}_j$ are some functions of class 
\begin{align*}
&\tilde{F}^\J \in C^{k-3+\beta, \frac{k-3+\beta}{2}}_{x,t}([0,1] \times [0,T_k]) \quad \text{for} \; \; j \in \{1,2,3\}, \\
&\tilde{b}_j \in C^{\frac{k+\beta}{2}}_t ([0,T_k]) \quad \text{for} \; \; j \in \{1,2\}, \quad \tilde{b}_3 \in C^{\frac{k-1+\beta}{2}}_t([0,T_k]), 
\end{align*}
for any $k\ge 4$. 
While some coefficients in \eqref{system-v} depends on $t$, we can prove that, by a similar argument as in Section \ref{subsec:linear}, the system \eqref{system-v} satisfies the parabolicity and the complementing condition whenever $\Theta^\JJ(1,t) - \Theta^\J(1,t) \in (0,\pi)$, and the condition holds until some time $T_0>0$ since $\Theta^\J$ is continuous and the condition is satisfied at initial time. 
Furthermore, due to the assumptions on the initial datum for the geometric flow, $v(\cdot,0)$ satisfies the compatibility condition of any order for \eqref{system-v}. 
It is thus possible to apply \cite[Theorem 10.1 in Chapter VII]{LSU}, as in Proposition \ref{prop:exist-linear}, to conclude that $v^\J \in C^{k-1+\beta, \frac{k-1+\beta}{2}}_{x,t}([0,1] \times [0,T_k])$ for $j \in \{1,2,3\}$ and  $k \ge 4$. 
Since $\px^2 \Theta^\J$ can be represented by a summation of some terms contains $v^\J$ and $\px \Theta^\J$ as in the first differential equation in \eqref{system-theta}, $\px^2 \Theta^\J$ has same regularity with $v$, which yields $\Theta^\J \in C^{k+1+\beta, \frac{k+1+\beta}{2}}_{x,t}([0,1] \times [0,T_k])$ for $j \in \{1,2,3\}$ and $k \ge 4$ (see also \cite[Theorem 2.2]{MR1184027} for the regularity result of $\Theta^\J$ from it of $v^\J$ and $\px^2 \Theta^\J$).
Therefore, we have $T_{k+1} \ge \min\{T_k, T_0\}$ and further $T_k \ge \min\{T_4, T_0\}$ for any $k \ge 4$.
\end{proof}

\section{Equilibrium}\label{sec:equilibrium}

In this section, we study the equilibrium state for \eqref{eq-curve}--\eqref{bc-boundary} to continue to prove the local exponential stability of the steady state. 
In particular, we show that the steady state is unique up to constant difference for $(\alpha^{(1)}, \alpha^{(2)}, \alpha^{(3)})$ under the assumptions (A1)--(A3). 

\begin{proposition}\label{prop:unique-stationary}
Let $\gamma, \sigma, P^\J (j=1,2,3)$ satisfy the assumptions (A1)--(A3). 
Let also $(\Gamma_\infty^\J, \alpha_\infty^\J)$ be a stationary solution to \eqref{eq-curve}--\eqref{bc-boundary}. 
Then, $\cup_{j=1}^3 \Gamma_\infty^\J$ is the unique Steiner triod connectiong the three boundary points $P^\J$ and $\alpha_\infty^{(1)} = \alpha_\infty^{(2)} = \alpha_\infty^{(3)}$. 
\end{proposition}

The latter property for $\alpha_\infty^\J$ can be obtain by a quite similar argument in \cite[Proposition 5.1]{MR4283537}. 
We however give the proof of the property for self-contained arguments.

\begin{proof}
Let $L^\J_\infty$ be the length of $\Gamma_\infty^\J$. 
We first prove that $\alpha_\infty^{(1)} = \alpha_\infty^{(2)} = \alpha_\infty^{(3)}$. 
Since $(\Gamma_\infty^\J, \alpha_\infty^\J)$ is a stationary solution, we have 
 \begin{equation}\label{stationary-alpha}
  \pa \sigma (\Delta^{(1)} \alpha_\infty) L_\infty^{(1)}
   =
   \pa \sigma (\Delta^{(2)}\alpha_\infty) L_\infty^{(2)}
   =
   \pa \sigma (\Delta^{(3)}\alpha_\infty) L_\infty^{(3)}, 
 \end{equation}
which is derived from \eqref{eq-alpha}. 
We also note that the assumption (A3) implies 
\begin{equation}\label{convexity-sigma} 
\partial_\alpha \sigma(\alpha) \alpha \ge 0 \quad \text{for} \quad \alpha \in \mathbb{R}. 
\end{equation}
Multiplying $\Delta^{(2)}\alpha_\infty$ and $\Delta^{(3)}\alpha_\infty$ to \eqref{stationary-alpha} and using convexity of $\sigma$ as in \eqref{convexity-sigma}, we find
 \begin{equation}
  \pa \sigma (\Delta^{(1)}\alpha_\infty)\Delta^{(2)} \alpha_\infty L_\infty^{(1)} \geq0,
   \quad
   \text{and}
   \quad 
   \pa \sigma (\Delta^{(1)}\alpha_\infty)\Delta^{(3)} \alpha_\infty L_\infty^{(1)} \geq0, 
 \end{equation}
 which imply $\pa \sigma(\Delta^{(1)}\alpha_\infty) \Delta^{(1)}\alpha_\infty \le 0$ by taking the summation and dividing by $L_\infty^{(1)}$. 
 Applying the convexity of $\sigma$ as in \eqref{convexity-sigma} again to obtain $\pa \sigma(\Delta^{(1)}\alpha_\infty) \Delta^{(1)}\alpha_\infty \ge 0$, we thus obtain $\pa \sigma(\Delta^{(1)}\alpha_\infty) \Delta^{(1)} \alpha_\infty = 0$, which implies $\Delta^{(1)} \alpha_\infty = 0$ due to the assumption (A3). 
 As the similar argument, we obtain $\Delta^{(2)}\alpha_\infty=\Delta^{(3)}\alpha_\infty=0$ and thus $\alpha_\infty^{(1)}=\alpha_\infty^{(2)}=\alpha_\infty^{(3)}$ holds. 
 
 We next prove that the network $\cup_{j=1}^3 \Gamma_\infty^\J$ is the unique Steiner triod connecting the three boundary points $P^\J$. 
 Since $(\Gamma_\infty^\J, \alpha_\infty^\J)$ is a stationary solution, the curvature of $\Gamma_\infty^\J$ is $0$ for any $j \in \{1,2,3\}$ in view of \eqref{eq-curve}. 
 Therefore, $\Gamma_\infty^\J$ is a line segment for $j \in \{1,2,3\}$ and the line segments generate a network connecting $P^{(1)}, P^{(2)}, P^{(3)}$ and a junction point $p_\infty$. 
 Furthermore, due to $\Delta^{(1)}\alpha_\infty = \Delta^{(2)}\alpha_\infty=\Delta^{(3)}\alpha_\infty=0$, the each pair of line segments in $\{\Gamma_\infty^\J\}_{j \in \{1,2,3\}}$ generates $120$ degree contact angle at the junction point $p_\infty$ in view of \eqref{bc-angle}. 
 We thus conclude that the network is the unique Steiner triod connecting the three boundary points $P^\J$ due to the assumption (A2). 
\end{proof}

\begin{remark}\label{rk:equilibrium}
If the assumption (A3) is dropped in Proposition \ref{prop:unique-stationary}, we possibly obtain more than two networks of equilibriums. 
One of the most simple case losing the uniqueness of the network is that $\partial_\alpha \sigma$ has other zero points $a, b$ with $2a + b=0$ besides $\alpha =0$. 
In this case, we can choose the orientation parameters to satisfy $\Delta^{(1)} \alpha_\infty = \Delta^{(2)} \alpha_\infty = a$ and $\Delta^{(3)} \alpha_\infty = b$. 
Therefore, if we adjust well the value of $\sigma$ at $a, b$ and the fixed points $P^\J$, there is a network consists of line segments $\Gamma^\J_\infty$ between a junction point $\vec{a}_\infty$ and $P^\J$ satisfying 
\[ \sum_{j=1}^3 \sigma(\Delta^\J \alpha_\infty) \tau_\infty^\J = 0. \]
Since this triod is obviously different from the Steiner triod if $\sigma(a) \neq \sigma(b)$, we can see the non-uniqueness of the geometric form of the equilibriums. 
We note additionally that an equilibrium with $\partial_\alpha \sigma(\Delta^\J \alpha_\infty) \neq 0$ for $j \in \{1,2,3\}$ also can be constructed if we choose $\sigma$ and $P^\J$ well. 
\end{remark}

\section{Geometric properties and exponential decay of misorientations}\label{sec:orientation}

In this section, we assume that a smooth geometric flow governed by \eqref{eq-curve}--\eqref{bc-boundary} exists until a time $T>0$. 
Let $\xi^\J$ be the parametrization of $\Gamma^\J_t$ defined as in Section \ref{sec:re-formulation} without the restriction \eqref{rest-para}. 
According to the discussions in Section \ref{sec:local-exists}, the uniformly boundedness of $L^\J$ from below and above is required to ensure the uniformly parabolicity of the problem, and we thus first study the boundedness.
The boundedness of $L^\J$ from above can be obtained the energy dissipation. 

\begin{lem}\label{lem:dec-E}
Assume (A1). 
Let 
\[ E(t):= \sum_{j=1}^3 \int_{\Gamma^\J_t} \sigma(\Delta^\J\alpha(t)) \; ds. \]
Then, 
\begin{equation}\label{ene-decrease} 
\dfrac{d}{dt} E(t) = - \sum_{j=1}^3 \left(\int_{\Gamma^\J_t} (V^\J_t)^2 \; ds +  \dfrac{(\partial_t \alpha^\J(t))^2}{\gamma}\right). 
\end{equation}
\end{lem}

\begin{proof}
Applying \eqref{jacobi-deri-t}, we have by $V_t^\J = \sigma(\Delta^\J \alpha(t)) \partial_s \Theta^\J$
\begin{equation}\label{ene-decrease1} 
\begin{aligned}
\dfrac{d}{dt} E(t) =&\; \sum_{j=1}^3 \dfrac{d}{dt} \int_0^1 \sigma(\Delta^\J \alpha(t)) |\px \xi^\J(x,t)| \; dx \\
=&\; \sum_{j=1}^3 \int_{\Gamma^\J_t} \partial_\alpha \sigma(\Delta^\J\alpha(t)) \partial_t (\Delta^\J\alpha(t)) - (V_t^\J)^2 + \sigma(\Delta^\J \alpha(t))\partial_s \lambda_t^\J  \; ds 
\end{aligned}
\end{equation}
It follows from the definition of $\Delta^\J\alpha = \alpha^{(j-1)} - \alpha^\J$ and \eqref{eq-alpha} that 
\begin{equation}\label{ene-decrease2}
\begin{aligned}
&\; \sum_{j=1}^3 \int_0^1 \partial_\alpha \sigma(\Delta^\J\alpha(t)) \partial_t (\Delta^\J\alpha(t)) \; ds \\
=&\; \sum_{j=1}^3 \left(\partial_\alpha \sigma(\Delta^{(j+1)}\alpha(t))L^{(j+1)}(t) - \partial_\alpha \sigma(\Delta^\J\alpha (t)) L^\J(t) \right) \partial_t \alpha^\J(t) \\
=&\; - \sum_{j=1}^3 \dfrac{(\partial_t \alpha^\J(t))^2}{\gamma}. 
\end{aligned}
\end{equation}
We also have by \eqref{bc-angle}, \eqref{bc-kappa} and $\lambda_t^\J = \langle \pt \vec{a}, \tau^\J_t \rangle$ at the junction point $\vec{a}$ 
\begin{equation}\label{ene-decrease3}
\sum_{j=1}^3 \int_{\Gamma^\J_t} \sigma(\Delta^\J \alpha(t))\partial_s \lambda_t^\J \; ds = \sum_{j=1}^3  \sigma(\Delta^\J\alpha(t)) \langle \pt \vec{a}, \tau^\J_t \rangle\lfloor_{\text{at} \; \vec{a}}  =0. 
\end{equation}
Combining \eqref{ene-decrease1}--\eqref{ene-decrease3}, we obtain \eqref{ene-decrease}. 
\end{proof}

We next give a sufficient condition to obtain the boundedness of $L^\J$ from below.

\begin{lem}\label{lem:min-length}
Assume (A1)--(A3).  
Then, there exist constants $L_{\rm min}>0$ and $\Cl[m]{min-length}>0$ such that 
\[ L^\J(t) \ge L_{\rm min} \quad \text{for} \quad j \in \{1,2,3\}, \; \; t \in[0,T) \]
if $E(0) \le \sigma(0) \Cr{min-length}$. 
\end{lem}

\begin{proof}
First, we introduce a functional $\tilde{E}(\vec{a})$ for $\vec{a} \in \mathbb{R}^2$ as 
\[ \tilde{E}(\vec{a}) := |\vec{a} - P^{(1)}| + |\vec{a} - P^{(2)}| + |\vec{a} - P^{(3)}|. \]
Let $\vec{a}_0$ be the Fermat point of the triangle generated by the boundary points $\{P^\J\}_{j \in \{1,2,3\}}$, then $\vec{a}_0$ is the unique minimizer of $\tilde{E}$. 
Since $\tilde{E}(\vec{a}) \to \infty$ as $|\vec{a}| \to \infty$, $\tilde{E}$ is continuous and $\vec{a}_0$ is unique and an interior point in the triangle, we can choose a constant $\Cr{min-length} > \tilde{E}(\vec{a}_0)$ sufficiently close to $\tilde{E}(\vec{a}_0)$ so that 
\[ S_{\Cr{min-length}} := \{\vec{a} \in \mathbb{R}^2: \tilde{E}(\vec{a}) \le \Cr{min-length} \} \]
is contained in the triangle generated by $\{P^\J\}_{j \in \{1,2,3\}}$, and let $L_{\rm min} > 0$ be the distance between $S_{\Cr{min-length}}$ and $\{P^\J\}_{j \in \{1,2,3\}}$. 

Since $\sigma(\alpha) \ge \sigma(0)$ for any $\alpha \in \mathbb{R}$ and the distance between the junction point $\vec{a}(t)$ and the fixed point $P^\J$ is not larger than the length $L^\J(t)$, we can see 
\[ E(t) \ge \sigma(0) \tilde{E}(\vec{a}(t)). \] 
Thus, by the monotonicity of $E(t)$ as in \eqref{ene-decrease}, we have 
\[ \tilde{E}(\vec{a}(t)) \le E(0)/\sigma(0) \le \Cr{min-length} \]
if $E(0) \le \sigma(0) m_1$. 
By the choice of $m_1$, we can see that $\vec{a}(t) \in S_{m_1}$ and 
\[ L^\J(t) \ge |\vec{a}(t) - P^\J| \ge {\rm dist}(S_{\Cr{min-length}}, \{P^\J\}_{j \in \{1,2,3\}}) = L_{\rm min} \]
for any $t \in [0, T)$. 
\end{proof}

We next study the relation between the angles $\Theta^\JJ - \Theta^\J$ at the junction point $\vec{a}$ and the surface tensions $\sigma(\Delta^\J \alpha(t))$, where $\Theta^\J$ is the angle function defined as in \eqref{def-theta}, to make it easier to discuss the preservation of $\Theta^\JJ - \Theta^\J \in (0,\pi)$ at the junction point. 
The preservation ensures the complementarity of the boundary conditions when we will construct geometric flow to extend the maximum existence time.

\begin{lem}\label{lem:inner-tau}
For any geometric flow governed by \eqref{eq-curve}--\eqref{bc-boundary}, the equality 
\begin{equation}
\langle \tau^\J_t, \tau^{(j+1)}_t \rangle = \dfrac{(\sigma(\Delta^{(j-1)} \alpha(t)))^2 - (\sigma(\Delta^\J\alpha(t)))^2 - (\sigma(\Delta^{(j+1)}\alpha(t)))^2}{2 \sigma(\Delta^\J\alpha(t)) \sigma(\Delta^{(j+1)}\alpha(t))} \label{inner-tau} 
\end{equation}
holds at the junction point $\vec{a}(t)$ for $j \in \{1,2,3\}$ and $t >0$. 
\end{lem}

\begin{proof}
For simplicity, we write $\sigma^\J$ as $\sigma(\Delta^\J\alpha(t))$ for $j \in \{1, 2, 3\}$. 
Multiplying the equality in (A4) by $\tau^\I_t$ and $i = 1, 2, 3$, we have 
\[ 
\begin{pmatrix}
\sigma^{(2)} & 0 & \sigma^{(3)} \\
\sigma^{(1)} & \sigma^{(3)} & 0 \\
0 & \sigma^{(2)} & \sigma^{(1)}
\end{pmatrix}
\begin{pmatrix}
\langle \tau_t^{(1)}, \tau_t^{(2)} \rangle \\
\langle \tau_t^{(2)}, \tau_t^{(3)} \rangle \\
\langle \tau_t^{(3)}, \tau_t^{(1)} \rangle
\end{pmatrix}
= -
\begin{pmatrix}
\sigma^{(1)} \\
\sigma^{(2)} \\
\sigma^{(3)}
\end{pmatrix} 
\]
at the junction point $\vec{a}(t)$. 
Thus, \eqref{inner-tau} can be obtained by calculating the inverse of the matrix. 
\end{proof}

We here note that formula of $\langle \tau^\J_t, \nu_t^{(j+1)} \rangle$ to apply in the $L^2$ or higher order estimate of curvatures latter, although it will not be applied to discuss the preservation of $\Theta^\JJ - \Theta^\J \in (0,\pi)$ at the junction point.

\begin{lem}\label{lem:inner-nu}
For any geometric flow governed by \eqref{eq-curve}--\eqref{bc-boundary}, if $\Theta^\JJ - \Theta^\J \in (0, \pi)$ at the junction point $\vec{a}(t)$ for $j \in \{1,2,3\}$ and $t>0$, the equality
\begin{equation}
\langle \tau^\J_t, \nu_t^{(j+1)} \rangle = - \dfrac{\sqrt{\left\{\sum_{i < k} 2(\sigma(\Delta^\I \alpha(t)))^2(\sigma(\Delta^\K\alpha (t)))^2\right\} - \left\{\sum_{i=1}^3 (\sigma(\Delta^\I \alpha (t)))^4\right\}}}{2 \sigma(\Delta^\J\alpha (t)) \sigma(\Delta^{(j+1)}\alpha (t))} \label{inner-nu} 
\end{equation}
holds at the junction point $\vec{a}(t)$ for $j \in \{1,2,3\}$ and $t >0$. 
\end{lem}

\begin{proof}
Since the curves $\Gamma^\J_t$ are numbered counter-clockwise around the junction point, we have by $\Theta^\JJ - \Theta^\J \in (0,\pi)$ at the junction point $\vec{a}(t)$ 
\begin{align*} 
\langle \tau^\J_t, \nu_t^{(j+1)} \rangle =&\; \cos (\Theta^\JJ - \Theta^\J + \pi/2) = - \sin(\Theta^\JJ - \Theta^\J)  \\
=&\; - \sqrt{1-\cos^2 (\Theta^\JJ - \Theta^\J)} = - \sqrt{1 - (\langle \tau^\J_t, \tau_t^{(j+1)} \rangle)^2} 
\end{align*}
due to the choice of the direction of $\tau^\J_t$ and $\nu^\J_t$. 
Substituting \eqref{inner-tau} into the above equality, we obtain \eqref{inner-nu}. 
\end{proof}

\begin{remark}\label{lem:linear-independent}
It immediately follows from Lemma \ref{lem:inner-tau} that 
\[ |\cos (\Theta^\JJ - \Theta^\J)| = |\langle \tau^\J_t, \tau^{(j+1)}_t \rangle| < 1 \] 
holds if and only if 
\begin{equation}\label{pre-complementarity} 
(\sigma(\Delta^\J \alpha(t)) - \sigma(\Delta^\JJ \alpha (t)))^2 < (\sigma(\Delta^{(j-1)} \alpha (t)))^2 < (\sigma(\Delta^\J\alpha(t)) + \sigma(\Delta^\JJ \alpha(t)))^2 
\end{equation}
holds for any $j \in \{1,2,3\}$ and $t \in [0,T]$. 
We thus see that a geometric flow governed by \eqref{eq-curve}--\eqref{bc-boundary} and starting from $\{\Gamma^\J_0\}_{j\in \{1,2,3\}}, \vec{\alpha}_0$ with $\Theta_0^\JJ - \Theta_0^\J \in (0,\pi)$ at the junction point preserves $\Theta^\JJ - \Theta^\J \in (0,\pi)$ if and only if $\vec{\alpha}$ satisfies \eqref{pre-complementarity}. 
\end{remark}

Since the right hand side of \eqref{inner-tau} consists of only $\sigma(\Delta^\J \alpha(t))$, we can estimate the angle $\Theta^\JJ - \Theta^\J$ at the junction point only by $\sigma(\Delta^\J \alpha(t))$. 
We thus continue to estimate the orientation parameters $\alpha^\J$ which can be applied to not only prove the preservation of $\Theta^\JJ - \Theta^\J \in (0,\pi)$ at the junction point but also obtain $L^2$ or higher order estimate of the curvatures. 
We first prove the dissipation of the orientations and misorientations.

\begin{lem}\label{lem:ori-decrease}
Assume (A1). 
Then, 
\begin{align}
& \dfrac{d}{dt} \sum_{j=1}^3 (\alpha^\J(t))^2 \le 0, \label{mono-a}\\
& \dfrac{d}{dt} \sum_{j=1}^3 (\Delta^\J\alpha(t))^2 \le 0 \label{mono-delta-a}
\end{align}
for any $t \in [0,T)$. 
\end{lem}

\begin{proof}
Multiplying \eqref{eq-alpha} by $\alpha^\J$ and taking the sum for $j=1,2,3$, we obtain 
\[ \begin{aligned}
\dfrac{1}{2} \dfrac{d}{dt} \sum_{j=1}^3 (\alpha^\J(t))^2 =&\; - \gamma \sum_{j=1}^3 \left(\partial_\alpha \sigma(\Delta^\JJ\alpha(t)) L^\JJ(t) - \partial_\alpha \sigma(\Delta^\J\alpha(t)) L^\J(t)\right)\alpha^\J(t) \\
=&\; -\gamma \sum_{j=1}^3 \left(\partial_\alpha \sigma(\Delta^\J\alpha(t))L^\J(t) \right)(\alpha^{(j-1)}(t) - \alpha^\J(t)) \\
=&\; -\gamma \sum_{j=1}^3 \Delta^\J\alpha(t) \partial_\alpha \sigma(\Delta^\J\alpha(t)) L^\J(t) \le 0. 
\end{aligned} \]
Thus, \eqref{mono-a} holds. 

We have by a simple calculation and \eqref{eq-alpha}
\[ \dfrac{d}{dt} \Delta^\J\alpha(t) = \gamma\left\{-2 \partial_\alpha \sigma(\Delta^\J\alpha(t))L^\J(t) + \partial_\alpha \sigma(\Delta^{(j-1)}\alpha(t))L^{(j-1)}(t) + \partial_\alpha \sigma(\Delta^\JJ\alpha(t)) L^\JJ(t)\right\} \]
Multiplying the equality by $\Delta^\J\alpha$ and taking the sum for $j=1,2,3$ we obtain 
\begin{equation}\label{mono-delta-a1} 
\dfrac{1}{2} \dfrac{d}{dt} \sum_{j=1}^3 (\Delta^\J\alpha(t))^2 = -3 \gamma \sum_{j=1}^3 \Delta^\J\alpha \partial_\alpha(\Delta^\J\alpha(t)) L^\J(t) \le 0. 
\end{equation}
Thus, \eqref{mono-delta-a} holds. 
\end{proof}

The exponential decay of the misorientations and derivatives of the orientations or surface tensions can be obtained assuming additional conditions (A2) and (A3) as follows.

\begin{cor} \label{cor:exdecrease}
Assume (A1)--(A3).
Let $\Cr{min-length}$ be the constant defined in Lemma \ref{lem:min-length} and $\Cl[m]{m-exdeltaa}$ be an arbitrary positive constant. 
Then, there exist constants $\Cl[c]{l-ex-deltaa}, \Cl[c]{c-ex-deri-deltaa}, \Cl[c]{c-ex-deltaa} > 0$ such that if $E(0) \le \sigma(0) \Cr{min-length}$ and 
\begin{equation}\label{exdecrease-as}
\sum_{j=1}^3 (\Delta^\J\alpha(0))^2 \le  \Cr{m-exdeltaa}
\end{equation} 
then, 
\begin{align}
&\sum_{j=1}^3 (\Delta^\J\alpha(t))^2 \le e^{-\Cr{l-ex-deltaa} t} \sum_{j=1}^3 (\Delta^\J\alpha(0))^2 \quad \text{for} \quad t \in [0,T), \label{delta-a-ex-decrease} \\
&|\partial_\alpha \sigma(\Delta^\J\alpha (t))| \le \Cr{c-ex-deri-deltaa} e^{-\frac{\Cr{l-ex-deltaa} t}{2}}\sqrt{\sum_{j=1}^3 (\Delta^\J\alpha(0))^2} \quad \text{for} \quad t \in [0,T), \; \; j \in \{1,2,3\}, \label{delta-aa-ex-decrease} \\
&|\partial_t \alpha^\J (t)| \le \Cr{c-ex-deltaa} e^{-\frac{\Cr{l-ex-deltaa} t}{2}}\sqrt{\sum_{j=1}^3 (\Delta^\J\alpha(0))^2} \quad \text{for} \quad t \in [0,T), \; \; j \in \{1,2,3\}. \label{delta-at-ex-decrease}
\end{align} 
\end{cor}

\begin{proof}
First, we prove \eqref{delta-a-ex-decrease}. 
From \eqref{mono-delta-a} and \eqref{exdecrease-as}, we have $|\Delta^\I\alpha(t)| \le \Cr{m-exdeltaa}^{1/2}$ for $t>0$ and $i \in \{1,2,3\}$.
Due to the assumption (A3), we can see by the boundedness of $|\Delta^\I\alpha(t)|$ and the Taylor expansion 
\[ \partial_\alpha \sigma(\Delta^\J\alpha(t))\Delta^\J \alpha(t) \ge \left(\min_{|\alpha|\le \Cr{m-exdeltaa}^{1/2}} \partial_\alpha^2\sigma(\alpha)\right) (\Delta^\J\alpha(t))^2. \]
Note that $\min_{|\alpha|\le \Cr{m-exdeltaa}^{1/2}} \partial_\alpha^2\sigma(\alpha)$ is positive due to the assumption (A3). 
Applying this inequality and Lemma \ref{lem:min-length} to \eqref{mono-delta-a1}, we obtain 
\[ \dfrac{d}{dt} \sum_{j=1}^3 (\Delta^\J\alpha(t))^2 \le - \Cr{l-ex-deltaa} \sum_{j=1}^3 (\Delta^\J\alpha(t))^2 \]
for some $\Cr{l-ex-deltaa} >0$ and hence $\sum_{j=1}^3 (\Delta^\J\alpha(t))^2$ decreases exponentially. 

Next, we prove \eqref{delta-aa-ex-decrease} and \eqref{delta-at-ex-decrease}. 
Due to \eqref{ene-decrease} and $\sigma(\alpha) \ge \sigma(0)$, we have 
\begin{equation}\label{ex-deltaa-1} 
E(0) \ge E(t) \ge \sigma(0) \sum_{j=1}^3 L^\J(t).  
\end{equation}
It also follows from the boundedness of $|\Delta^\J\alpha(t)|$, the Tayler expansion, the assumption (A3) and the inequality \eqref{delta-a-ex-decrease} that 
\begin{equation}\label{ex-deltaa-2} 
|\partial_\alpha \sigma(\Delta^\I\alpha(t))| \le \left(\max_{|\alpha|\le \Cr{m-exdeltaa}^{1/2}} \partial_\alpha^2\sigma(\alpha)\right)e^{-\frac{\Cr{l-ex-deltaa} t}{2}} \sqrt{\sum_{j=1}^3 (\Delta^\J\alpha(0))^2} 
\end{equation}
for any $i \in \{1,2,3\}$. 
We thus obtain \eqref{delta-aa-ex-decrease} by letting $\Cr{c-ex-deri-deltaa} := \max_{|\alpha|\le \Cr{m-exdeltaa}^{1/2}} \partial_\alpha^2\sigma(\alpha)$. 
Applying \eqref{ex-deltaa-1} and \eqref{ex-deltaa-2} to \eqref{eq-alpha}, we obtain \eqref{delta-at-ex-decrease}. 
\end{proof}

\begin{remark}\label{rmk:bddness1}
The boundedness results follows from Lemma \ref{lem:ori-decrease} was extended to the exponential decay estimates in Corollary \ref{cor:exdecrease} due to the uniformly boundedness of $L^\J$ from below. 
The assumptions (A2)--(A3) and $E(0) \le \sigma(0) \Cr{m-exdeltaa}$ in Corollary \ref{cor:exdecrease}, which ensure the uniformly boundedness of $L^\J$ from below, are not required to obtain only the boundedness of $|\pa \sigma(\Delta^\J \alpha)|$ and $|\pt \alpha^\J|$. 
\end{remark}

Due to the exponential decay of the misorientations, as following lemma, we can see the exponential stability of the angle condition of the Steiner triod at the junction point, in other word, $\Theta^\JJ - \Theta^\J$ exponentially converges to $2\pi/3$ at the junction point if the initial misorientations is sufficiently small. 

\begin{lem}\label{lem:tri-ine-t}
Assume (A1)--(A3). 
Let $\Cr{min-length}$ be the constant in Lemma \ref{lem:min-length}.  
Then, there exists $\Cl[e]{e-ex-inner}, \Cl[c]{c-ex-inner}>0$ such that if $E(0) \le \sigma(0)\Cr{min-length}$ and 
\[ \sum_{j=1}^3 (\Delta^\J\alpha(0))^2 \le \Cr{e-ex-inner}, \] 
then the unit tangent vectors 
$\tau_t^\I$ and $\tau_t^\J$ of the geometric flow satisfy
\begin{equation}\label{bdd-sin}
\left| \langle \tau_t^\I, \tau_t^\J \rangle +\frac{1}{2} \right| \leq \Cr{c-ex-inner} e^{-\frac{\Cr{l-ex-deltaa}}{2} t} \sqrt{\sum_{j=1}^3 (\Delta^\J\alpha(0))^2} 
\quad \text{for vary} \; \; i, j \in \{1,2,3\}, \; \; t \in [0,T)
\end{equation}
at the junction point $\vec{a}$,
where $\Cr{l-ex-deltaa}$ is the constant defined in Corollary \ref{cor:exdecrease} replaced $\Cr{m-exdeltaa}$ by $\Cr{e-ex-inner}$. 
\end{lem}

\begin{proof}
Let 
\[ \tilde\varepsilon := \sqrt{\sum_{j=1}^3 (\Delta^\J\alpha(0))^2} \]
in this proof for simplicity. 
By \eqref{exdecrease-as} and \eqref{delta-a-ex-decrease}, we have
\begin{equation}\label{ex-inner-alpha}
|\sigma (0) - \sigma (\Delta^\J\alpha (t))| 
\leq \max _{|\alpha|\leq \Cr{e-ex-inner}^{1/2} } |\partial _\alpha \sigma (\alpha)| |\Delta^\J\alpha (t)|
\leq C e^{-\frac{\Cr{l-ex-deltaa}}{2} t} \tilde\varepsilon
\end{equation}
for any $j$ and some $C>0$. 
We thus obtain by \eqref{inner-tau}, \eqref{ex-inner-alpha} and $\tilde\varepsilon \le \Cr{e-ex-inner}^{1/2}$
\begin{align*}
& - \langle \tau^\J_t, \tau^{(j+1)}_t \rangle \\
= & \,
\frac{ (\sigma(\Delta^\J\alpha(t)))^2 + (\sigma(\Delta^{(j+1)}\alpha(t)))^2
-(\sigma(\Delta^{(j-1)}\alpha(t)))^2}
{2 \sigma(\Delta^\J\alpha(t)) \sigma(\Delta^{(j+1)}\alpha(t))}\\
\geq & \,
\frac{2(\sigma (0) - C e^{-\frac{\Cr{l-ex-deltaa}}{2} t}\tilde\varepsilon )^2 
- (\sigma (0) + C e^{-\frac{\Cr{l-ex-deltaa}}{2} t} \tilde\varepsilon)^2}
{ 2(\sigma (0) + C e^{-\frac{\Cr{l-ex-deltaa}}{2} t} \Cr{e-ex-inner}^{1/2})^2 }\\
= & \,
\frac{2\{ (\sigma (0) + C e^{-\frac{\Cr{l-ex-deltaa}}{2} t}\tilde\varepsilon) -2C e^{-\frac{\Cr{l-ex-deltaa}}{2} t}\tilde\varepsilon \}^2 
- (\sigma (0) + C e^{-\frac{\Cr{l-ex-deltaa}}{2} t} \tilde\varepsilon)^2}
{ 2(\sigma (0) + C e^{-\frac{\Cr{l-ex-deltaa}}{2} t}\Cr{e-ex-inner}^{1/2} )^2 } \\
= & \,
\frac{ (\sigma (0) + C e^{-\frac{\Cr{l-ex-deltaa}}{2} t}\tilde\varepsilon)^2 
- 8 (\sigma (0) + C e^{-\frac{\Cr{l-ex-deltaa}}{2} t}\tilde\varepsilon)C e^{-\frac{\Cr{l-ex-deltaa}}{2} t}\tilde\varepsilon + 8 ( C e^{-\frac{\Cr{l-ex-deltaa}}{2} t}\tilde\varepsilon)^2 }
{ 2(\sigma (0) + C e^{-\frac{\Cr{l-ex-deltaa}}{2} t} \Cr{e-ex-inner}^{1/2})^2 }\\
\ge & \, \frac12 
- \frac{4 C e^{-\frac{\Cr{l-ex-deltaa}}{2} t}\tilde\varepsilon}{\sigma (0) +C e^{-\frac{\Cr{l-ex-deltaa}}{2} t}\Cr{e-ex-inner}}
+\frac{4 (C e^{-\frac{\Cr{l-ex-deltaa}}{2} t})^2\tilde\varepsilon}{(\sigma (0) +C e^{-\frac{\Cr{l-ex-deltaa}}{2} t}\Cr{e-ex-inner}^{1/2})^2}\\
\geq & \, \frac12 - C' e^{-\frac{\Cr{l-ex-deltaa}}{2} t}\tilde\varepsilon,
\end{align*}
where $C'>0$ depends only on $\sigma (0)$ and $C$
(we may choose $\Cr{e-ex-inner}$ small so that $C'>0$ if necessary). 
Therefore we have $\langle \tau^\J_t, \tau^{(j+1)}_t \rangle 
\leq -\frac12 + C' e^{-\frac{\Cr{l-ex-deltaa}}{2} t}\tilde\varepsilon$. 
Similarly we can obtain $\langle \tau^\J_t, \tau^{(j+1)}_t \rangle \ge -\frac12 - C'' e^{-\frac{\Cr{l-ex-deltaa}}{2} t}\tilde\varepsilon$ for some constant $C''>0$. 
\end{proof}

Note that Lemma \ref{lem:tri-ine-t} also implies that $\Theta^\JJ_0 - \Theta^\J_0 \in (0, \pi)$ at the junction point $\vec{a}(0)$ whenever \eqref{bc-angle} at $t=0$ is satisfied and the initial misorientations is sufficiently small. 
Although we already obtain a sufficient condition to ensure the preservation of $\Theta^\JJ - \Theta^\J \in (0,\pi)$ at the junction point as in Lemma \ref{lem:tri-ine-t}, it will be applied also to the $L^2$ or higher order estimate of the curvatures that $|\langle \tau_t^\I, \tau_t^\J\rangle|$ is uniformly less than $1$ at the junction point. 
We thus note it as the following corollary to make it easier to cite. 

\begin{cor}\label{cor:bdd-sin2}
Assume (A1)--(A3). 
Let $\Cr{min-length}$ be the constant in Lemma \ref{lem:min-length} and $\Cl[m]{m-tri-ine-t} \in (1/2, 1)$ be arbitrary.  
Then, there exists $\Cl[e]{e-ex-inner2}$ such that if $E(0) \le \sigma(0)\Cr{min-length}$ and 
\begin{equation}\label{as-bdd-sin} 
\sum_{j=1}^3 (\Delta^\J\alpha(0))^2 \le \Cr{e-ex-inner2}, 
\end{equation} 
then the unit tangent vectors $\tau_t^\I$ and $\tau_t^\J$ of the geometric flow satisfy
\begin{equation}\label{bdd-sin2}
\left| \langle \tau_t^\I, \tau_t^\J \rangle \right| \leq \Cr{m-tri-ine-t}
\quad \text{for vary} \; \; i,j,k \in \{1,2,3\}, \; \; t \in [0,T)
\end{equation}
at the junction point $\vec{a}$.
\end{cor} 


\section{Exponential $L^2$-decay of the curvatures} \label{sec:L2-estimate}

In this section, we prove the exponential $L^2$-decay of the curvatures assuming the closeness of the family of the initial datum to the unique Steiner triod, namely, assuming smallness of the misorientations and the $L^2$-norm of the curvatures. 
The exponential decay will be continued to obtain higher order estimate of the curvatures which corresponds to a smoothing effect and it also ensures the local exponential stability of the Steiner triod in $C^\infty$-topology.

We assume that a smooth geometric flow governed by \eqref{eq-curve}--\eqref{bc-boundary} exists until a time $T>0$ as in Section \ref{sec:orientation}. 
We note that the restriction \eqref{rest-para} is not assumed also in this section.
The basic idea for the $L^2$-estimate based on the energy method as in \cite{GKS}. 
In order to obtain energy type inequality, we apply the geometric identities in Lemma \ref{lem:pros-kappa} and Lemma \ref{lem:l-v-junction}, and also additional identities as follows.

\begin{lem}\label{lem:pros-kappa2}
Any smooth geometric flow governed by \eqref{eq-curve}--\eqref{bc-boundary} fulfills the following identities. 
\begin{align}
&\partial_t \tau_t^\J = (\sigma(\Delta^\J \alpha(t))\partial_s \kappa^\J_t + \kappa_t^\J \lambda^\J_t) \nu_t^\J, \label{tau-deri-t}\\
&\partial_t \kappa_t^\J = \sigma(\Delta^\J \alpha(t))\partial_s^2 \kappa^\J_t + \sigma(\Delta^\J \alpha(t)) (\kappa_t^\J)^3 + \lambda^\J \partial_s \kappa_t^\J \label{eq-kappa}
\end{align}
for any $(x,t) \in [0,1] \times [0,T)$ and $j \in \{1,2,3\}$. 
\end{lem}

The proof is standard and thus we refer to \cite[Chapter 1]{CZ} for the details of the proof. 
An additional boundary identity at the junction point will be needed to control the boundary terms in the energy type inequality. 

\begin{lem}\label{prop:dec-f}
For any smooth geometric flow governed by \eqref{eq-curve}--\eqref{bc-boundary}, there exists a smooth function $f^\J: [0,T) \to \mathbb{R}$ for $j \in \{1,2,3\}$ such that 
\begin{equation} \label{bc-diffe-kappa}
\begin{aligned}
&(\sigma(\Delta^\J \alpha(t))\partial_s \kappa^\J_t + \kappa_t^\J \lambda^\J_t) - (\sigma(\Delta^\JJ \alpha(t))\partial_s \kappa^\JJ_t + \kappa_t^\JJ \lambda^\JJ_t) \\
=&\; f^\J(\sigma(\Delta^{(1)}\alpha(t)), \sigma(\Delta^{(2)}\alpha(t)),\sigma(\Delta^{(3)}\alpha(t))) 
\end{aligned}
\end{equation}
at the junction point $\vec{a}(t)$. 
Furthermore, under the assumptions (A1)--(A3), there exist $\Cl[e]{e-ex-f}, \Cl[c]{c-ex-f}>0$ such that if $E(0) \le \sigma(0)\Cr{min-length}$, where $\Cr{min-length}$ is the constant in Lemma \ref{lem:min-length}, and 
\begin{equation}\label{dec-f-as} 
\sum_{j=1}^3 \left(\Delta^\J\alpha_0\right)^2 \le \Cr{e-ex-f}, 
\end{equation}
then 
\begin{equation}\label{ex-f-decrease}
|f^\J(\sigma(\Delta^{(1)}\alpha(t)), \sigma(\Delta^{(2)}\alpha(t)),\sigma(\Delta^{(3)}\alpha(t)))| \le \Cr{c-ex-f} e^{-\frac{\Cr{l-ex-deltaa}}{2}t } \sqrt{\sum_{j=1}^3 \left(\Delta^\J\alpha_0\right)^2} 
\end{equation}
for any $j \in \{1,2,3\}$ and $t \in [0,T)$, where $\Cr{l-ex-deltaa}$ is the constant defined in Corollary \ref{cor:exdecrease} replaced $\Cr{m-exdeltaa}$ by $\Cr{e-ex-f}$. 
\end{lem}

\begin{proof}
Let $\sigma^\J = \sigma(\Delta^\J \alpha(t))$ for simplicity. 
Due to \eqref{tau-deri-t} at the junction point, we have for any $j \in \{1,2,3\}$ 
\[ \begin{aligned}
\partial_t \langle\tau_t^\J, \tau_t^\JJ \rangle =&\; (\sigma^\J \partial_s \kappa^\J_t + \kappa_t^\J \lambda^\J_t) \langle \nu_t^\J, \tau_t^\JJ \rangle +  (\sigma^\JJ \partial_s \kappa^\JJ_t + \kappa_t^\JJ \lambda^\JJ_t) \langle \tau_t^\J, \nu_t^\JJ \rangle \\
=&\; \left((\sigma^\J \partial_s \kappa^\J_t + \kappa_t^\J \lambda^\J_t) - (\sigma^\JJ \partial_s \kappa^\JJ_t + \kappa_t^\JJ \lambda^\JJ_t)\right) \langle \nu_t^\J, \tau_t^\JJ \rangle. 
\end{aligned}\]
We can see that $\Theta^\JJ - \Theta^\J \in (0, \pi)$ at the function point $\vec{a}$ for any $t>0$ from Lemma \ref{lem:tri-ine-t} and, additionally, \eqref{bdd-sin2} also implies
\begin{equation}\label{bdd-cos} 
\langle \nu_t^\J, \tau_t^\JJ \rangle = \sqrt{1-\langle \tau_t^\J, \tau_t^\JJ\rangle^2} \ge \sqrt{1-\Cr{m-tri-ine-t}^2} 
\end{equation}
for a fixed constant $\Cr{m-tri-ine-t} \in (1/2, 1)$ if $\Cr{e-ex-f}$ is sufficiently small. 
Letting 
\begin{equation}\label{def-f}
f^\J(\sigma^{(1)}, \sigma^{(2)},\sigma^{(3)}) := \frac{1}{\sqrt{1-\langle \tau_t^\J, \tau_t^\JJ\rangle^2}} \partial_t \left(\dfrac{(\sigma^{(j-1)})^2 - (\sigma^\J)^2 - (\sigma^{(j+1)})^2}{2 \sigma^\J \sigma^{(j+1)}}\right), 
\end{equation}
we have \eqref{bc-diffe-kappa} due to \eqref{inner-tau}. 
The time derivative term can be formulated as 
\begin{equation}\label{control-f-bdry1} 
\partial_t \left(\dfrac{(\sigma^{(j-1)})^2 - (\sigma^\J)^2 - (\sigma^{(j+1)})^2}{2 \sigma^\J \sigma^{(j+1)}}\right) = \sum_{k=1}^3 \frac{\hat{P}^\J_k(\sigma^{(1)}, \sigma^{(2)},\sigma^{(3)})}{\left(2 \sigma^\J \sigma^{(j+1)}\right)^2}\partial_\alpha \sigma(\Delta^\K\alpha(t)) \partial_t \Delta^\K\alpha(t), 
\end{equation}
where $\hat{P}^\J_k$ are polynomials in $\sigma^{(1)}, \sigma^{(2)}$ and $\sigma^{(3)}$ with finite degree. 
Since $\sigma^\K$ and $\partial_\alpha \sigma(\Delta^\K\alpha(t))$ are bounded and $\sigma^\K \ge \sigma(0)$, we thus have the boundedness of the absolute value of the coefficient of $\partial_t \Delta^\K\alpha(t)$ in \eqref{control-f-bdry1}, which yields \eqref{ex-f-decrease} due to \eqref{delta-at-ex-decrease} and \eqref{bdd-cos}. 
\end{proof}

Another key control to boundary terms in the energy type inequality is estimate of the tangent velocity by the curvatures at the junction point, which enable us to derive an energy type inequality consists of only geometric values independent of change of parametrizations.
The estimate of the tangent velocity can be obtained from Lemma \ref{lem:l-v-junction} and Corollary \ref{cor:bdd-sin2} immediately.

\begin{lem}\label{prop:con-tanv}
Assume (A1)--(A3). 
Let $\Cr{m-tri-ine-t} \in (1/2, 1)$ be an arbitrary constant. 
Let $\Cr{min-length}$ be the constant in Lemma \ref{lem:min-length} and $\Cr{e-ex-inner2}$ be the constant in Corollary \ref{cor:bdd-sin2}. 
If $E(0) \le \sigma(0)\Cr{min-length}$ and \eqref{as-bdd-sin} holds, then 
\begin{equation}\label{bdd-l-V} 
|\lambda_t^\J| \le \frac{3}{1- \Cr{m-tri-ine-t}^3} \sum_{i=1}^3 \sigma(\Delta^\I \alpha(t))|\kappa_t^\I| 
\end{equation}
at the junction point for any $j \in \{1,2,3\}$ and $t \in (0, T)$. 
\end{lem}

\begin{remark}\label{rmk:bddness2}
The uniformly boundedness \eqref{bdd-sin2} was applied to Lemma \ref{prop:dec-f} and Lemma \ref{prop:con-tanv}. 
According to Remark \ref{rmk:bddness1}, we can obtain the boundedness of $|f^\J|$ and the estimate \eqref{bdd-l-V} whenever \eqref{bdd-sin2} holds without assuming (A2)--(A3) and $E(0) \le \sigma(0)\Cr{min-length}$. 
\end{remark}

We now derive an energy type equality from the geometric identities obtained above and clarify additional necessary estimates and discussions to obtain the $L^2$-decay of the curvatures. 
The formulation and the calculation will be long, and thus we simplify the notions $\sigma(\Delta^\J\alpha(t))$ and $\partial_\alpha \sigma(\Delta^\J\alpha(t))$ to $\sigma^\J$ and $\partial_\alpha \sigma^\J$, respectively, for any $j \in \{1,2,3\}$. 

\begin{lem}\label{lem:ene-ine-kappa}
Any smooth geometric flow fulfills the following energy type identity. 
\begin{equation}\label{ene-ine-kappa}
\begin{aligned}
&\dfrac{d}{dt} \sum_{j=1}^3 \int_{\Gamma^\J} (\sigma^\J)^2 (\kappa^\J_t)^2 \; ds \\
&= \sum_{j=1}^3 \int_{\Gamma^\J_t} (\sigma^\J)^3 \{ -2(\partial_s \kappa^\J_t)^2 + (\kappa^\J_t)^4\} + 2 \sigma^\J (\partial_\alpha \sigma^\J) (\partial_t \Delta^\J\alpha) (\kappa^\J_t)^2  \; ds \\
&\; \; + \sum_{j=1}^3 \left\{\frac{2}{3}\left((\sigma^\J)^2 \kappa_t^\J\right)\Big\lfloor_{\text{at} \; \vec{a}} f^{(j-1)} - \frac{2}{3}\left((\sigma^\J)^2 \kappa_t^\J\right)\Big\lfloor_{\text{at} \; \vec{a}} f^\J + \left((\sigma^\J)^2 (\kappa_t^\J)^2 \lambda_t^\J \right)\Big\lfloor_{\text{at} \; \vec{a}}\right\}
\end{aligned}
\end{equation}
\end{lem}

\begin{proof}
We have by \eqref{jacobi-deri-t}, \eqref{bc-kappa}, \eqref{eq-kappa} and integration by parts
\begin{equation} \label{ene-ine-kappa1}
\begin{aligned}
&\; \dfrac{d}{dt} \int_{\Gamma^\J_t} (\sigma(\Delta^\J\alpha))^2 (\kappa^\J_t)^2 \; ds \\
=&\; \int_{\Gamma^\J_t}  2 (\sigma^\J)^2 \kappa^\J_t \partial_t \kappa^\J_t  + 2\sigma^\J (\partial_\alpha \sigma^\J) (\partial_t \Delta^\J\alpha) (\kappa^\J_t)^2 \\
&\; + (\sigma^\J)^2 (\kappa^\J_t)^2 (-\sigma^\J(\kappa_t^\J)^2 + \partial_s \lambda_t^\J) \; ds \\
=&\; \int_{\Gamma^\J} (\sigma^\J)^3 \{2 \kappa^\J_t \partial_s^2 \kappa_t^\J + (\kappa^\J_t)^4 \} + 2 \sigma^\J (\partial_\alpha \sigma^\J) (\partial_t \Delta^\J\alpha) (\kappa^\J_t)^2 \\
&\; + (\sigma^\J)^2\partial_s ((\kappa^\J_t)^2 \lambda_t^\J) \; ds \\
=&\; \int_{\Gamma^\J} (\sigma^\J)^3 \{ -2(\partial_s \kappa^\J_t)^2 + (\kappa^\J_t)^4\} + 2 \sigma^\J (\partial_\alpha \sigma^\J) (\partial_t \Delta^\J\alpha) (\kappa^\J_t)^2 \; ds \\
&\; + (\sigma^\J)^2 (2 \kappa_t^\J \partial_s V^\J_t + (\kappa^\J_t)^2 \lambda_t^\J ) \lfloor_{\text{at} \; \vec{a}} \\
=&\; \int_{\Gamma^\J} (\sigma^\J)^3 \{ -2(\partial_s \kappa^\J_t)^2 + (\kappa^\J_t)^4\} + 2 \sigma^\J (\partial_\alpha \sigma^\J) (\partial_t \Delta^\J\alpha) (\kappa^\J_t)^2  \; ds \\
&\; + 2 \left((\sigma^\J)^2 \kappa_t^\J (\partial_s V_t^\J + \kappa_t^\J \lambda_t^\J)\right)\Big\lfloor_{\text{at} \; \vec{a}} - \left((\sigma^\J)^2 (\kappa_t^\J)^2 \lambda_t^\J \right)\Big\lfloor_{\text{at} \; \vec{a}}. 
\end{aligned} 
\end{equation}
The equalities \eqref{bc-sum-kappa} and \eqref{bc-diffe-kappa} imply 
\begin{equation}\label{ene-ine-kappa2} 
\begin{aligned}
&\;\sum_{j=1}^3 \left((\sigma^\J)^2 \kappa_t^\J (\partial_s V_t^\J + \kappa_t^\J \lambda_t^\J)\right)\Big\lfloor_{\text{at} \; \vec{a}} \\
=&\; \frac{1}{3}\sum_{j=1}^3 \left\{\left((\sigma^\J)^2 \kappa_t^\J\right)\Big\lfloor_{\text{at} \; \vec{a}} f^\J(\sigma^{(1)}, \sigma^{(2)},\sigma^{(3)}) - \left((\sigma^\J)^2 \kappa_t^\J\right)\Big\lfloor_{\text{at} \; \vec{a}} f^{(j-1)}(\sigma^{(1)}, \sigma^{(2)},\sigma^{(3)})\right\}. 
\end{aligned} 
\end{equation}
We here note that an identity for $A^\J, B^\J$ and $F^\J$ with $\sum_{j=}^3 A_j = 0$ and $B^\J = B^\JJ + F^\J$ for $j \in \{1,2,3\}$ as 
\begin{equation}\label{simple-sum}
\begin{aligned}
\sum_{j=1}^3 A^\J B^\J =&\; \frac{1}{3} (A^{(1)}B^{(1)} + A^{(2)}(B^{(1)} -F^{(1)}) + A^{(3)}(B^{(1)} + F^{(3)})) \\
&\; + \frac{1}{3} (A^{(1)} (B^{(2)} + F^{(1)}) + A^{(2)} B^{(2)}  + A^{(3)} (B^{(2)} - F^{(2)})) \\
&\; + \frac{1}{3} (A^{(1)} (B^{(3)} - F^{(3)}) + A^{(2)} (B^{(3)} + F^{(2)}) + A^{(3)} B^{(3)})) \\
=&\; \frac{1}{3} \sum_{j=1}^3 A^\J (F^\J - F^{(j-1)}) 
\end{aligned}
\end{equation}
has been applied. 
We thus have \eqref{ene-ine-kappa} by taking the summation of \eqref{ene-ine-kappa1} with respect to $j \in \{1,2,3\}$. 
\end{proof}

Our aim is to prove that the right hand side of \eqref{ene-ine-kappa} is bounded by some negative constant times the weighted $L^2$-norm of the curvature. 
Therefore, the important term is the negative term in the integration in \eqref{ene-ine-kappa} since a Poincar\'e type inequality can be applied to the term. 
Hereafter, we will introduce the Poincar\'e type inequality under the boundary conditions \eqref{bc-kappa} and \eqref{bc-sum-kappa} by applying the Reyleigh quotient. 

\subsection{Reyleigh quotient} \label{subsec:Reyleigh}

We consider the Reyleigh quotient in the following settings. 

\begin{definition}
Let $\tilde{L}^\J, \tilde{\sigma}^\J > 0$ be fixed constants for $j \in \{1,2,3\}$ and define the Hilbert spaces $H, V$ with inner products $\langle \cdot, \cdot \rangle_H$ and $\langle \cdot, \cdot \rangle_V$ as 
\begin{align*}
H :=&\; L^2(0,\tilde{L}^{(1)}) \times L^2(0,\tilde{L}^{(2)}) \times L^2(0,\tilde{L}^{(3)}), \\
V :=&\; \{\phi := (\phi^{(1)}, \phi^{(2)}, \phi^{(3)}) \in H^1(0,\tilde{L}^{(1)}) \times H^1(0,\tilde{L}^{(2)}) \times H^1(0,\tilde{L}^{(3)}) : \\
&\quad \quad \quad \phi^\J(0) = 0, \quad (\tilde{\sigma}^{(1)})^2 \phi^{(1)}(\tilde{L}^{(1)}) + (\tilde{\sigma}^{(2)})^2 \phi^{(2)}(\tilde{L}^{(2)}) + (\tilde{\sigma}^{(3)})^2 \phi^{(3)}(\tilde{L}^{(3)})= 0 \}, \\
\langle\phi, \psi \rangle_H :=&\; \sum_{j=1}^3 (\tilde{\sigma}^\J)^2 \int_0^{\tilde{L}^\J} \phi^\J(s) \psi^\J(s) \; ds, \\
\langle\phi, \psi \rangle_V :=&\; \sum_{j=1}^3 (\tilde{\sigma}^\J)^2 \left(\int_0^{\tilde{L}^\J} \phi^\J(s) \psi^\J(s) \; ds + \int_0^{\tilde{L}^\J} \partial_s \phi^\J(s) \partial_s \psi^\J(s) \; ds\right). 
\end{align*}
The norm of $H$ and $V$ are denoted by $\| \cdot \|_H$ and  $\| \cdot \|_V$, respectively. 
\end{definition}

We should introduce the following functionals to consider the Rayleigh quotient related to the negative term in \eqref{ene-ine-kappa}.

\begin{definition}\label{def:IJ}
We define a bilinear form $J: V \times V \to \mathbb{R}$ and a functional $I: V\setminus \{0\} \to \mathbb{R}$ by 
\begin{align*}
J(\phi, \psi) := &\; \sum_{j=1}^3 (\tilde{\sigma}^\J)^3\int_{0}^{\tilde{L}^\J} \partial_s \phi^\J \partial_s \psi^\J \; ds, \\
I(\phi) :=&\; \frac{J(\phi,\phi)}{\langle \phi, \phi\rangle_H}. 
\end{align*}
\end{definition}

The following linear operator $\mathcal{L}$ is naturally leaded from the Euler-Lagrange equation of the Rayleigh quotient and we consider the weak form of $\mathcal{L}$. 

\begin{definition}
For boundary conditions 
\begin{align}
&\phi^\J(0) = 0 \quad \text{for} \; \; j \in \{1,2,3\}, \label{Ray-bc-1}\\
&\tilde{\sigma}^{(1)} \partial_s \phi^{(1)} (\tilde{L}^{(1)}) = \tilde{\sigma}^{(2)} \partial_s \phi^{(2)} (\tilde{L}^{(2)}) = \tilde{\sigma}^{(3)} \partial_s \phi^{(3)} (\tilde{L}^{(3)}), \label{Ray-bc-2}\\
&(\tilde{\sigma}^{(1)})^2 \phi^{(1)}(\tilde{L}^{(1)}) + (\tilde{\sigma}^{(2)})^2 \phi^{(2)}(\tilde{L}^{(2)}) + (\tilde{\sigma}^{(3)})^2 \phi^{(3)}(\tilde{L}^{(3)})= 0, \label{Ray-bc-3}
\end{align}
we define a linear operator $\mathcal{L}$ in $H$ and its domain $\mathcal{D}(\mathcal{L})$ by 
\[ \begin{aligned}
\mathcal{L} \phi :=&\; (\tilde{\sigma}^{(1)}\partial_s^2 \phi^{(1)}, \tilde{\sigma}^{(2)}\partial_s^2 \phi^{(2)}, \tilde{\sigma}^{(3)}\partial_s^2 \phi^{(3)}), \\
\mathcal{D}(\mathcal{L}) :=&\; \{\phi = (\phi^{(1)}, \phi^{(2)}, \phi^{(3)}) : \\
&\; \quad \phi^\J \in H^2(0, \tilde{L}^\J) \; \; (j \in \{1,2,3\}) \; \; \text{and} \; \; \phi \; \; \text{satisfies \eqref{Ray-bc-1}--\eqref{Ray-bc-3}}\}. 
\end{aligned} \]
\end{definition}

\begin{lem}\label{lem:nonzero-eigen}
All eigenvalues of $\mathcal{L}$ are nonzero real values. 
\end{lem}

\begin{proof}
Since $\mathcal{L}$ is self-adjoint with respect to the inner product $\langle \cdot, \cdot \rangle_H$ and its weight is real, $\mathcal{L}$ has only real eigenvalues. 
We thus suppose by contradiction that $\mathcal{L}$ has a zero eigenvalue. 
Then, its eigenfunction $\phi_0 \in \mathcal{D}(\mathcal{L}) \setminus \{0\}$ satisfies $\phi_0^\J \in C^\infty([0, \tilde{L}^\J])$ and is of the form 
\[ \phi_0^\J(s) = a^\J s + b^\J \quad \text{for} \; \; s \in [0, \tilde{L}^\J], \]
where $a^\J, b^\J \in \mathbb{R}$, since $\partial_s^2 \phi_0^\J = 0$ in $[0, \tilde{L}^\J]$ for $j \in \{1,2,3\}$. 
Applying the boundary conditions \eqref{Ray-bc-1}--\eqref{Ray-bc-3}, we have 
\begin{align} 
&b^\J = 0 \; \; \text{for} \; \;  j \in \{1,2,3\}, \quad \tilde{\sigma}^{(1)} a^{(1)} = \tilde{\sigma}^{(2)} a^{(2)} = \tilde{\sigma}^{(3)} a^{(3)}, \label{nonzero-eigen1}\\
&\sum_{j=1}^3 (\tilde{\sigma}^\J)^2 (a^\J \tilde L ^\J + b^\J )= 0. \label{nonzero-eigen2}
\end{align}
The equalities in \eqref{nonzero-eigen1} yield 
\[ {\rm sgn} a^\J = {\rm sgn} a^\I \quad \text{for} \; \; i,j \in \{1,2,3\}, \]
where ${\rm sgn}$ is the signum function, since $\tilde{\sigma}^\J$ is positive. 
It however contradicts \eqref{nonzero-eigen2} since $\tilde{L}^\J$ is positive. 
\end{proof}

\begin{lem}
The functional $I(\phi)$ has a minimizer $\overline{\phi} = (\overline{\phi}^{(1)}, \overline{\phi}^{(2)}, \overline{\phi}^{(3)}) \in V \setminus \{0\}$. 
Moreover, $\overline{\phi} \in \mathcal{D}(\mathcal{L})$, $\mathcal{L} \overline{\phi} = \zeta \overline{\phi}$, $\overline{\phi}^\J \in C^\infty([0, \tilde{L}^\J])$ for $j \in \{1,2,3\}$ and $\zeta < 0$, where $\zeta = - I(\overline{\phi})$. 
\end{lem}

\begin{proof}
It is obvious that $I(\phi) \ge 0$ for any $\phi \in V\setminus \{0\}$. 
Therefore, there exists a sequence $\{\phi_n\}_{n \in \mathbb{N}} \subset V \setminus \{0\}$ such that 
\[ \lim_{n \to \infty} I(\phi_n) = \inf_{\phi \in V \setminus \{0\}} I(\phi). \]
We now put 
\[ \zeta := - \inf_{\phi \in V\setminus \{0\}} I(\phi). \]
Without loss of generality we may assume that $\|\phi_n\|_H = 1$ for $n \in \mathbb{N}$. 
Then 
\[ \|\phi_n\|_V \le \left(\min_{j \in \{1,2,3\}} \tilde{\sigma}^\J\right)^{-1} I(\phi) + 1 \]
and hence $\{\phi_n\}_{n \in \mathbb{N}}$ is bounded sequence in $V$. 
By similar discussions to those in \cite[Section 8.12]{GT}, we can show the existence of $\overline{\phi} \in V\setminus \{0\}$ such that $\phi_n \to \overline{\phi}$ in $V$ and $I(\overline{\phi}) = - \zeta$. 

Since $\overline{\phi}$ is a minimizer of $I$, for any $\psi \in V$ it holds that 
\[ \frac{d I(\overline{\phi} + t \psi)}{d t} \Big\lfloor_{t=0} = 0. \]
Therefore, we obtain 
\begin{equation}\label{Ray-eigen1} 
J(\overline{\phi}, \psi) + \zeta\langle\overline{\phi}, \psi \rangle_H = 0 \quad \text{for} \; \; \psi \in V. 
\end{equation}
Since $\overline{\phi}^\J \in C^\infty([0, \tilde{L}^\J])$ by a standard argument, it remains only to prove $\overline{\phi} \in \mathcal{D}(\mathcal{L})$ and $\zeta < 0$. 
Due to $\overline{\phi} \in V$, the minimizer $\overline{\phi}$ satisfies the boundary conditions \eqref{Ray-bc-1} and \eqref{Ray-bc-3}. 
By the integration by parts for \eqref{Ray-eigen1}, we see that 
\[ \sum_{j=1}^3 (\tilde{\sigma}^\J)^3 \partial_s \overline{\phi}^\J(\tilde{L}^\J) \psi^\J(\tilde{L}^\J) = 0 \quad \text{for} \; \; \psi \in V, \]
which yields \eqref{Ray-bc-2} for $\overline{\phi}$ since $\psi$ satisfies only the boundary condition \eqref{Ray-bc-3}. 
It thus follows from Lemma \ref{lem:nonzero-eigen} that $\zeta \neq 0$. 
On the other hand, $\zeta \le 0$ is obvious due to the definition of $I$. 
Therefore, we have $\zeta < 0$. 
\end{proof}

\begin{lem}
The minimum $- \zeta$ of $I$ in $V \setminus\{0\}$ is continuous with respect to $\tilde{L}^\J, \tilde{\sigma}^\J > 0$ for $j \in \{1,2,3\}$. 
\end{lem}

\begin{proof}
Fix $\tilde{L}^\J_0, \tilde{\sigma}^\J_0 > 0$ and let $\tilde{L}^\J, \tilde{\sigma}^\J > 0$ be arbitrary constants to take the limit $\tilde{L}^\J \to \tilde{L}^\J_0$ and $\tilde{\sigma}^\J \to \tilde{\sigma}_0^\J$ for $j \in \{1,2,3\}$. 
The functional $I_0$ is defined as $I$ in Definition \ref{def:IJ} for $\tilde{L}^\J_0, \tilde{\sigma}^\J_0$ and let $- \zeta_0$ be its minimum in $V$ replaced $\tilde{L}^\J, \tilde{\sigma}^\J$ by $\tilde{L}^\J_0, \tilde{\sigma}^\J_0$. 
We will thus prove $\zeta \to \zeta_0$ as $\tilde{L}^\J \to \tilde{L}^\J_0$ and $\tilde{\sigma}^\J \to \tilde{\sigma}_0^\J$ for $j \in \{1,2,3\}$. 

Let $\phi \in V\setminus\{0\}$ be the minimizer of $I$ and thus $I(\phi) = -\zeta$. 
We define $\phi_0 = (\phi^{(1)}_0, \phi^{(2)}_0, \phi^{(3)}_0)$ as 
\[ \phi_0^\J(\tilde{s}) := \frac{\tilde{\sigma}^\J}{\tilde{\sigma}^\J_0} \phi \left(\frac{\tilde{L}^\J}{\tilde{L}^\J_0} \tilde{s}\right) \quad \text{for} \; \; \tilde{s} \in [0, \tilde{L}^\J_0] \]
then the boundary conditions \eqref{Ray-bc-1} and \eqref{Ray-bc-3} replaced $\tilde{L}^\J, \tilde{\sigma}^\J$ by $\tilde{L}^\J_0, \tilde{\sigma}^\J_0$ are satisfied, and hence $\phi_0$ is in $V$ replaced $\tilde{L}^\J, \tilde{\sigma}^\J$ by $\tilde{L}^\J_0, \tilde{\sigma}^\J_0$. 
We also have by a simple calculation 
\[ \begin{aligned}
&(\tilde{\sigma}^\J_0)^3\int_{0}^{\tilde{L}^\J_0} (\partial_{\tilde{s}} \phi^\J_0)^2 \; d\tilde{s} = \frac{\tilde{\sigma}^\J_0 \tilde{L}^\J}{\tilde{\sigma}^\J \tilde{L}^\J_0} \cdot (\tilde{\sigma}^\J)^3\int_{0}^{\tilde{L}^\J} (\partial_s \phi^\J)^2 \; ds, \\
& (\tilde{\sigma}^\J_0)^2 \int_0^{\tilde{L}^\J_0} (\phi^\J_0)^2 \; d\tilde{s} = \frac{\tilde{L}^\J_0}{\tilde{L}^\J} \cdot (\tilde{\sigma}^\J)^2 \int_0^{\tilde{L}^\J} (\phi^\J)^2 \; ds, 
\end{aligned} \]
which implies by the definition of $\zeta_0$ 
\[ -\zeta_0 \le I_0(\phi_0) \le \left(\max_{j \in \{1,2,3\}} \frac{\tilde{\sigma}^\J_0 \tilde{L}^\J}{\tilde{\sigma}^\J \tilde{L}^\J_0} \right)\left(\max_{j \in \{1,2,3\}} \frac{\tilde{L}^\J}{\tilde{L}^\J_0} \right) I(\phi) = - \left(\max_{j \in \{1,2,3\}} \frac{\tilde{\sigma}^\J_0 \tilde{L}^\J}{\tilde{\sigma}^\J \tilde{L}^\J_0} \right)\left(\max_{j \in \{1,2,3\}} \frac{\tilde{L}^\J}{\tilde{L}^\J_0} \right) \zeta. \]
We can similarly obtain 
\[ -\zeta \le -\left(\max_{j \in \{1,2,3\}} \frac{\tilde{\sigma}^\J \tilde{L}^\J_0}{\tilde{\sigma}^\J_0 \tilde{L}^\J} \right)\left(\max_{j \in \{1,2,3\}} \frac{\tilde{L}^\J_0}{\tilde{L}^\J} \right) \zeta_0. \]
We thus see $\zeta \to \zeta_0$ as $\tilde{L}^\J \to \tilde{L}^\J_0$ and $\tilde{\sigma}^\J \to \tilde{\sigma}_0^\J$ for $j \in \{1,2,3\}$. 
\end{proof}

We are now in a position to apply the Rayleigh quotient to the geometric flow. 

\begin{prop}\label{prop:Poincare-type}
Assume (A1)--(A3). 
Let a family of $\{\Gamma^\J_t\}_{j\in \{1,2,3\}}$ and $\vec{\alpha}$ be a smooth geometric flow governed by \eqref{eq-curve}--\eqref{bc-boundary} in the time interval $[0,T)$. 
Let also $\Cr{min-length} > 0$ be the constant in Lemma \ref{lem:min-length}. 
Assume $E(0) \le \sigma(0) \Cr{min-length}$. 
Then, there exists a constant $\Cl[c]{c-Ray} > 0$ depending only on $\Cr{min-length}$ and $\sum_{j=1}^3(\sigma(\Delta^\J\alpha(0)))^2$ such that 
\begin{equation}\label{Poincare-type} 
\Cr{c-Ray} \sum_{j=1}^3 \int_{\Gamma^\J} (\sigma(\Delta^\J \alpha(t)))^2 (\kappa_t^\J)^2 \; ds \le \sum_{j=1}^3 \int_{\Gamma^\J} (\sigma(\Delta^\J \alpha(t)))^3 (\partial_s \kappa_t^\J)^2 \; ds \quad \text{for} \; \; t \in [0,T). 
\end{equation}
\end{prop}

\begin{proof}
Due to \eqref{ene-decrease} and \eqref{mono-delta-a}, we have by the assumption 
\[ L_{\rm min} \le L^\J(t) \le \Cr{min-length}, \quad \sigma(0) \le \sigma(\Delta^\J\alpha(t)) \le \sqrt{\sum_{j=1}^3(\sigma(\Delta^\J\alpha(0)))^2} \]
for any $j \in \{1,2,3\}$ and $t \in [0, T)$, where $L_{\rm min}$ is the constant in Lemma \ref{lem:min-length}. 
We thus have \eqref{Poincare-type} by letting 
\[ \Cr{c-Ray} := \min \left\{ -\zeta = \inf_{\phi \in V \setminus \{0\}} I(\phi) : L_{\rm min} \le \tilde{L}^\J \le \Cr{min-length}, \; \;  \sigma(0) \le \tilde{\sigma}^\J \le \sqrt{\sum_{j=1}^3(\sigma(\Delta^\J\alpha(0)))^2} \right\} > 0 \]
since $-\zeta$ is continuous. 
\end{proof}

\subsection{Proof of the exponential $L^2$-decay of the curvatures} \label{subsec:decay-l2}

In this section, we prove the exponential $L^2$-decay of the curvature for the geometric flow applying \eqref{ene-ine-kappa}. 
In order to control the first junction term, we recall inequalities in \cite[Lemma 6.2]{GKS}. 
We refer to \cite{GKS} for the details of the proof. 

\begin{lem}\label{lem:ine-trace}
For a $C^3$-curve $\Gamma^\J_t$, there exist $\Cl[c]{c-trace1}$ and $\Cl[c]{c-trace2}$ depending only on $L^\J(t)$ and $1/L^\J(t)$, respectively, such that
\[ \begin{aligned}
&\int_{\Gamma^\J_t} (\kappa^\J_t)^4 \; ds \le 2\left(L^\J(t) \|\partial_s \kappa^\J_t\|_{L^2}^2 + \frac{1}{L^\J(t)}\|\kappa^\J_t\|^2_{L^2}\right)\|\kappa^\J_t\|^2_{L^2}, \\
&\left|\kappa^\J_t \lfloor_{s=0}\right|^3 \le \Cr{c-trace1}\|\kappa^\J_t\|_{L^2}\|\partial_s \kappa^\J_t\|^2_{L^2} + \Cr{c-trace2}\|\kappa^\J_t\|^3_{L^2}. 
\end{aligned} \]
\end{lem}

We can prove the exponential $L^2$-decay of the curvature for the geometric flow assuming the smallness of $\Delta^\J\alpha(0)$ and the weighted $L^2$-norm of $\kappa^\J_0$ due to the discussion so far. 

\begin{prop}\label{prop:decay-l2}
Assume (A1)--(A3). 
Let a family of $\{\Gamma^\J_t\}_{j\in \{1,2,3\}}$ and $\vec{\alpha}$ be a smooth geometric flow governed by \eqref{eq-curve}--\eqref{bc-boundary} in the time interval $[0,T)$. 
Let also $\Cr{min-length}$ be the constant in Lemma \ref{lem:min-length}. 
Then, there exists $\Cl[e]{e-small-initial1} > 0$ small and $\Cl[c]{c-decay-l2} > 0$ such that if 
\begin{equation}\label{small-initiala} 
E(0) \le \sigma(0)\Cr{min-length}, \quad \sum_{j=1}^3 \left(\Delta^\J\alpha_0\right)^2 \le \Cr{e-small-initial1}, \quad \sum_{j=1}^3 \int_{\Gamma_0^\J} \left(\sigma(\Delta^\J\alpha_0)\right)^2 \left(\kappa^\J_0 \right)^2 ds \le \Cr{e-small-initial1}, 
\end{equation}
then 
\begin{equation}\label{ex-decay-l2} 
\begin{aligned}
&\; \sum_{j=1}^3 \int_{\Gamma^\J_t} \left(\sigma(\Delta^\J\alpha(t))\right)^2 \left(\kappa^\J_t \right)^2 ds \\
\le &\; e^{-\Cr{c-decay-l2} t} \left\{\sum_{j=1}^3 \int_{\Gamma^\J_0} \left(\sigma(\Delta^\J \alpha_0)\right)^2 \left(\kappa^\J_0 \right)^2 ds + \frac{8 \Cr{c-ex-f}^{\frac{3}{2}}}{\Cr{l-ex-deltaa}}\left(\sum_{j=1}^3 \left(\Delta^\J\alpha_0\right)^2\right)^{\frac{3}{4}} \right\} \quad \text{for} \; \; t \in [0,T), 
\end{aligned}
\end{equation}
where $\Cr{l-ex-deltaa}$ and $\Cr{c-ex-f}$ are the constants defined in Lemma \ref{prop:dec-f}.
\end{prop}

\begin{proof}
We simplify the notion $\sigma(\Delta^\J\alpha(t))$ to $\sigma^\J$ for $j \in \{1,2,3\}$ in this proof. 
Note that we have by Lemma \ref{lem:dec-E}, Lemma \ref{lem:min-length}, \eqref{mono-delta-a} and the definition of $E$ 
\begin{equation}\label{ex-decay-l21} 
L_{\rm min} \le L^\J(t) \le \Cr{min-length}, \quad \sigma(0) \le \sigma^\J \le \sigma\left(\Cr{e-small-initial1}^{1/2}\right) \quad \text{for} \; \; t \in [0,T), \; \; j \in \{1,2,3\} 
\end{equation}
due to the first assumption in \eqref{small-initiala}. 
We will choose $\Cr{e-small-initial1}$ in \eqref{small-initiala} small so that Corollary \ref{cor:exdecrease}, Lemma \ref{prop:dec-f} and Lemma \ref{prop:con-tanv} with a fixed constant  $\Cr{m-tri-ine-t} \in (1/2, 1)$ can be applied.
It then follows from Young's inequality, the estimate of $\sigma(\Delta^\J\alpha(t))$ in \eqref{ex-decay-l21} and \eqref{bdd-l-V} that 
\begin{equation}\label{ex-decay-l22}
\begin{aligned}
\sum_{j=1}^3 (\sigma^\J)^2 (\kappa_t^\J)^2 |\lambda^\J_t| \le&\; \frac{3}{1- \Cr{m-tri-ine-t}^3} \left(\sum_{j=1}^3 (\sigma^\J)^2 (\kappa_t^\J)^2\right) \left(\sum_{i=1}^3 \sigma^\I |\kappa^\I_t| \right) \\
\le&\; M_1 \sum_{j=1}^3 |\kappa_t^\J|^3
\end{aligned}
\end{equation}
at the junction point for some constant $M_1>0$. 
We note also that there exists a constant $M_2 > 0$ such that 
\begin{equation}\label{ex-decay-l23}
\sum_{j=1}^3 \left|\frac{2}{3}\left((\sigma^\J)^2 \kappa_t^\J \lfloor_{\text{at} \; \vec{a}}\right) f^{(j-1)} - \frac{2}{3}\left((\sigma^\J)^2 \kappa_t^\J\lfloor_{\text{at} \; \vec{a}} \right) f^\J \right| \le 
M_2 \sum_{j=1}^3 |\kappa_t^\J\lfloor_{\text{at} \; \vec{a}}|^3 + \sum_{j=1}^3 |f^\J|^{3/2}
\end{equation} 
due to Young's inequality, where $f^\J$ is the function appeared in Lemma \ref{prop:dec-f}. 

We first prove that if $\Cr{e-small-initial1}$ is small, then \eqref{ex-decay-l2} holds in any time interval $[0, \tau]$ such that 
\begin{equation}\label{ex-decay-l24} 
\sum_{j=1}^3 \int_{\Gamma_t^\J} \left(\sigma(\Delta^\J\alpha(t))\right)^2 \left(\kappa^\J_t \right)^2 ds \le \Cr{e-small-initial1} + \frac{8\Cr{c-ex-f}^{\frac{3}{2}}}{\Cr{l-ex-deltaa}} \Cr{e-small-initial1}^{\frac{3}{4}} \quad \text{for} \quad t \in [0,\tau] 
\end{equation}
holds. 
We then obtain by Corollary \ref{cor:exdecrease}, \eqref{ex-f-decrease}, Lemma \ref{lem:ene-ine-kappa}, Lemma \ref{lem:ine-trace}, \eqref{ex-decay-l21}, \eqref{ex-decay-l22}, \eqref{ex-decay-l23} and \eqref{ex-decay-l24}
\begin{equation}\label{ex-decay-l25} \begin{aligned}
&\; \dfrac{d}{dt} \sum_{j=1}^3 \int_{\Gamma_t^\J} \left(\sigma^\J \right)^2 \left(\kappa_t^\J\right)^2 ds \\
\le&\; (M(\Cr{e-small-initial1}) -2) \sum_{j=1}^3 \int_{\Gamma_t^\J} \left(\sigma^\J \right)^3 \left(\partial_s \kappa_t^\J\right)^2 ds + M(\Cr{e-small-initial1}) \sum_{j=1}^3 \int_{\Gamma_t^\J} \left(\sigma^\J \right)^2 \left(\kappa_t^\J\right)^2 ds  \\
&\; + 3\Cr{c-ex-f}^{\frac{3}{2}} e^{-\frac{3\Cr{l-ex-deltaa}}{4}t } \left(\sum_{j=1}^3 \left(\Delta^\J\alpha_0\right)^2\right)^{\frac{3}{4}}
\end{aligned} \end{equation}
in the time interval $[0,\tau]$, where $M \in C([0,\infty))$ is a positive function such that $M(\Cr{e-small-initial1}) \to 0$ as $\Cr{e-small-initial1} \to 0$. Note that we used $\inf _{\alpha \in \R} \sigma (\alpha) >0$ in \eqref{ex-decay-l25}.
Therefore, due to Proposition \ref{prop:Poincare-type}, we can choose $\Cr{e-small-initial1}$ small to obtain 
\begin{equation} \label{ex-decay-l26} \begin{aligned} 
& (M(\Cr{e-small-initial1}) -2 \mu) \sum_{j=1}^3 \int_{\Gamma_t^\J} \left(\sigma^\J \right)^3 \left(\partial_s \kappa_t^\J\right)^2 ds + M(\Cr{e-small-initial1}) \sum_{j=1}^3 \int_{\Gamma_t^\J} \left(\sigma^\J \right)^2 \left(\kappa_t^\J\right)^2 ds  \\
& \le - \Cr{c-Ray} \sum_{j=1}^3 \int_{\Gamma_t^\J} \left(\sigma^\J \right)^2 \left(\kappa_t^\J\right)^2 ds. 
\end{aligned} \end{equation}
We fix such $\Cr{e-small-initial1}$ and let 
\[ \Cr{c-decay-l2} := \min \left\{\Cr{c-Ray}, \frac{3\Cr{l-ex-deltaa}}{8} \right\}, \quad a:= \frac{8 \Cr{c-ex-f}^{\frac{3}{2}}}{\Cr{l-ex-deltaa}}\left(\sum_{j=1}^3 \left(\Delta^\J \alpha_0\right)^2\right)^{\frac{3}{4}}. \]
We then obtain by \eqref{ex-decay-l25} and \eqref{ex-decay-l26}
\[ \begin{aligned}
&\; \dfrac{d}{dt}\left(e^{\Cr{c-decay-l2} t}\sum_{j=1}^3 \int_{\Gamma_t^\J} \left(\sigma^\J \right)^2 \left(\kappa_t^\J\right)^2 ds + a e^{-\frac{3\Cr{l-ex-deltaa}}{8}t} \right)\\
\le &\; (\Cr{c-decay-l2} - \Cr{c-Ray}) e^{\Cr{c-decay-l2} t}\sum_{j=1}^3 \int_{\Gamma_t^\J} \left(\sigma^\J \right)^2 \left(\kappa_t^\J\right)^2 ds + \left\{ 3\Cr{c-ex-f}^{\frac{3}{2}}\left(\sum_{j=1}^3 \left(\Delta^\J\alpha_0\right)^2\right)^{\frac{3}{4}} - \frac{3\Cr{l-ex-deltaa}}{8} a \right\}e^{-\frac{3\Cr{l-ex-deltaa}}{8}t} \le 0, 
\end{aligned} \]
which yields \eqref{ex-decay-l2} in the time interval $[0,\tau]$. 

We finally prove that the inequality \eqref{ex-decay-l24} holds whenever the flow exists under the assumption \eqref{small-initiala}. 
Due to the assumption and the continuity of $\|\kappa^\J\|_{L^2}$, the inequality \eqref{ex-decay-l24} holds a short time interval $[0,\tau]$. 
We then have \eqref{ex-decay-l2} in the time interval $[0,\tau]$, which yields \eqref{ex-decay-l24} without equality at time $t=\tau$ by virtue of the assumption \eqref{small-initiala}. 
Therefore, applying the continuity of $\|\kappa^\J\|_{L^2}$ again, we see that the inequality \eqref{ex-decay-l24} holds in a time interval $[0, \tau + \varepsilon]$ for some $\varepsilon>0$ small. 
The argument can be applied whenever the flow exists, and thus the proof is completed. 
\end{proof}

\section{Higher order decay of the curvatures} \label{sec:H2-estimate}

In this section, we derive higher order decay of the curvatures and the basic idea is based on the higher order energy method as in \cite{MR3495423}. 
Therefore, we will need higher order geometric identities as in Section \ref{sec:L2-estimate} and let a smooth geometric flow governed by \eqref{eq-curve}--\eqref{bc-boundary} exists until a time $T>0$ also in this section. 
Since statements in lemmas and calculations will be long, we simplify the notions $\sigma(\Delta^\J\alpha(t))$ and $\partial_\alpha^k \sigma(\Delta^\J\alpha(t))$ to $\sigma^\J$ and $\partial_\alpha^k \sigma^\J$, respectively, for any $j \in \{1,2,3\}$ and $k \in \mathbb{N}$.
We further introduce some non standard notions for the computations in the sequel. 
The notions are extension of those in \cite{MR3495423} to be able to apply the system including the orientation parameters $\vec{\alpha}$. 

\begin{definition}
We denote with $P_h(\ps^i \kappa^\J_t)$ a polynomial in $\kappa^\J_t, \cdots, \ps^i \kappa^\J_t$ such that every monomial it contains is of the form 
\[ C \prod_{l=0}^i (\ps^l \kappa_t^\J)^{A_l} \quad \text{with} \; \; \sum_{l=0}^i (l+1)A_l = h, \]
where $C$ is a constant coefficient. 
We will write $P_h'(\ps^i \kappa^\J_t)$ if every monomial further satisfies $\sum_{l=0}^i A_l \ge 2$. 
We also denote with $P_h(\ps^i \vec{\kappa}_t)$ a polynomial in $\kappa^\J_t, \cdots, \ps^i \kappa^\J_t$ with any $j \in \{1,2,3\}$, such that every monomial it contains is of the form 
\[ C \prod_{j=1}^3 \prod_{l=0}^i (\ps^l \kappa_t^\J)^{A_l^\J} \quad \text{with} \; \; \sum_{j=1}^3\sum_{l=0}^i (l+1)A_l^\J = h. \]
We will call $h$ the geometric order of $P_h$ for both polynomials. 
When we will write $P_h(|\ps^i \kappa_t^\J|)$ (resp.\ $P_{\le h}(|\ps^i \kappa_t^\J|)$) we will mean a finite sum of terms like 
\begin{equation}\label{def-polys-ab} 
C \prod_{l=0}^i |\ps^l \kappa_t^\J|^{A_l} \quad \text{with} \; \; \sum_{l=0}^i (l+1)A_l = h \; \; \left(\text{resp.} \; \sum_{l=0}^i (l+1)A_l \le h, \; \; \sum_{l=0}^i A_l > 0\right), 
\end{equation} 
where $C$ is a constant coefficient and the exponents $A_l$ is non negative real values. 
Similarly, let $P_h(\|\ps^i \kappa_t^\J \|)$ and $P_{\le h}(\| \ps^i \kappa_t^\J\|)$ be a finite sum of terms like \eqref{def-polys-ab} replaced $|\ps^l \kappa_t^\J|$ by $\|\ps^l \kappa_t^\J\|_{L^\infty}$. 
These notions will be used for $P_h(\ps^i \vec{\kappa}_t)$ as $P_h(|\ps^i \vec{\kappa}_t|)$, $P_{\le h}(|\ps^i \vec{\kappa}_t|)$, $P_h(\|\ps^i \vec{\kappa}_t\|)$ and $P_{\le h}(\|\ps^i \vec{\kappa}_t\|)$. 

We next denote with $Q(\vec{\sigma})$ a polynomial in $\pa^n \sigma^\J$ with $n \in \mathbb{N} \cup \{0\}, j \in \{1,2,3\}$, and $Q(\vec{\sigma}, \pt \Delta \vec{\alpha})$ a polynomial in $\pa^n \sigma^\J$ with $n \in \mathbb{N} \cup \{0\}, j \in \{1,2,3\}$ and also in $\pt \Delta^\J \alpha$ with $j \in \{1,2,3\}$. 
For $i \in \mathbb{N}$, $\hat{Q}_h (\vec{\sigma}, \pt^i \Delta \vec{\alpha})$ is denoted by a polynomial such that every monomial it contains is of the form 
\[ C(\vec{\sigma}) \prod_{j=1}^3 \prod_{l=1}^i (\pt^l \Delta^\J \alpha)^{A_l^\J} \quad \text{with} \; \; \sum_{j=1}^3 \sum_{l=1}^i l A_l^\J = h, \] 
where $C(\vec{\sigma})$ is a monomial in $\pa^n \sigma^\J$ with $n \in \mathbb{N} \cup \{0\}$ and $j \in \{1,2,3\}$. 
Let also $R (\vec{\Theta})$ be a finite sum of fractions such that every fraction is of the form 
\[ \frac{\hat{P}(c^{(1)}, c^{(2)}, c^{(3)}, s^{(1)}, s^{(2)}, s^{(3)})}{(1 - c^{(1)} c^{(2)} c^{(3)})^n}, \]
where $n$ is an natural number and $\hat{P}$ is a polynomial in $c^\J = \cos (\Theta^\JJ - \Theta^\J)\lfloor_{\text{at} \; \vec{a}}$ and $s^\J = \sin (\Theta^\JJ - \Theta^\J)\lfloor_{\text{at} \; \vec{a}}$. 
When we will write multiplications of above notions we will mean a finite sum of the multiplication (e.g.: One example of $Q(\vec{\sigma}) P_2(\ps \vec{\kappa}_t)$ is $\sigma^{(1)} (\kappa^{(1)}_t)^2 + \sigma^{(2)} \ps \kappa_t^{(2)}$). 
\end{definition}

\begin{remark}\label{rmk:pro-polys}
By means of the definitions, we may easily see that $P_h(\pt^i \kappa^\J_t) \le P_h(|\pt^i \kappa^\J_t|) \le P_h(\|\pt^i \kappa^\J_t\|)$. 
A similar inequality holds also for $P_h(\pt^i \vec{\kappa}_t)$. 
The calculus rules as 
\begin{align*} 
&P_{h_1}(\ps^{i_1} \kappa^\J_t) \cdot P_{h_2}(\ps^{i_2} \kappa^\J_t) = P_{h_1 + h_2} (\px^{\max \{i_1, i_2\}} \kappa_t^\J), \\
&\hat{Q}_{h_1}(\vec{\sigma}, \pt^{i_1} \Delta \vec{\alpha}) \cdot \hat{Q}_{h_2}(\vec{\sigma}, \pt^{i_2} \Delta \vec{\alpha}) = \hat{Q}_{h_1 + h_2} (\vec{\sigma}, \pt^{\max \{i_1, i_2\}} \vec{\alpha}),
\end{align*}
and 
\begin{equation}\label{deri-Q} 
\pt \hat{Q}_h(\vec{\sigma}, \pt^i \Delta \vec{\alpha}) = \hat{Q}_{h+1}(\vec{\sigma}, \pt^{i+1} \Delta \vec{\alpha}) 
\end{equation}
will be useful to calculate in the sequel. 
Note that the polynomials $Q$ is bounded due to Lemma \ref{lem:ori-decrease} and Corollary \ref{cor:exdecrease} (according to Remark \ref{rmk:bddness1}, (A1) is a sufficient assumption to obtain the boundedness of $Q$), while the estimate of $\hat{Q}_h(\pt^i \Delta \vec{\alpha})$ is now not obvious for $i \ge 2$ and will be considered. 
The boundedness of $R$ follows from \eqref{bdd-sin2}. 
Notice also that, \eqref{eq-lambda-V} can be re-writen as 
\begin{equation}\label{simple-l-V}
\lambda^\J_t \lfloor_{\text{at} \; \vec{a}} = Q(\vec{\sigma}) R(\vec{\Theta}) P_1(\vec{\kappa}_t) \lfloor_{\text{at} \; \vec{a}} 
\end{equation}
and the above identity, \eqref{eq-theta1} and \eqref{eq-kappa} implies 
\begin{equation}\label{simple-theta}
\pt R(\vec{\Theta}) = Q(\vec{\sigma}) R(\vec{\Theta}) P_2(\ps \vec{\kappa}_t)\lfloor_{\text{at} \; \vec{a}}, \quad \pt \lambda^\J_t\lfloor_{\text{at} \; \vec{a}} = \{\hat{Q}_1(\vec{\sigma}, \pt \Delta \vec{\alpha})R(\vec{\Theta}) P_1(\vec{\kappa}_t)  + Q(\vec{\sigma}) R(\vec{\Theta}) P_3(\ps^2 \vec{\kappa}_t)\}\lfloor_{\text{at} \; \vec{a}}. 
\end{equation}
\end{remark}

We first list higher order geometric identities as follows. 

\begin{lem}\label{lem:pros-kappa-hi}
Any smooth geometric flow governed by \eqref{eq-curve}--\eqref{bc-boundary} fulfills the following identities. 
\begin{equation} \label{simple-kappa}
\pt \ps^n \kappa_t^\J = \sigma^\J \left(\ps^{n+2} \kappa_t^\J  + P_{n+3}'(\ps^n \kappa_t^\J)\right) + \lambda^\J_t \ps^{n+1} \kappa_t^\J
\end{equation}
for any $(x,t) \in [0,1] \times (0,T)$, $n \in \mathbb{N}\cup \{0\}$ and $j \in \{1,2,3\}$. 
Furthermore, 
\begin{equation}\label{bc-ks2}
\partial_s^{2n} \kappa_t^\J \Big\lfloor_{\text{at} \; P^\J} = 0
\end{equation}
for any $n \in \mathbb{N} \cup\{0\}$ and $j \in \{1,2,3\}$. 
\end{lem}

\begin{proof}
Note that we can obtain by \eqref{jacobi-deri-t} and $\pt \px = \px \pt$
\begin{equation}\label{change-ts} 
\partial_t \partial_s = \pt \frac{\px}{|\px \xi^\J|} = \partial_s \partial_t + (\sigma^\J (\kappa^\J_t)^2 - \partial_s \lambda_t^\J)\partial_s 
\end{equation}
on each $\Gamma^\J_t$. 
The identity \eqref{simple-kappa} can be proved inductively. 
When $n=0$, the identity \eqref{simple-kappa} coincides with \eqref{eq-kappa}. 
Assuming \eqref{simple-kappa} holds for some $n \in \mathbb{N} \cup \{0\}$, due to \eqref{change-ts}, we obtain 
\begin{align*}
\partial_t \partial_s^{n+1} \kappa_t^\J =&\; \partial_s \partial_t \ps^n \kappa^\J_t + (\sigma^\J (\kappa^\J_t)^2 - \partial_s \lambda_t^\J) \partial_s^{n+1} \kappa_t^\J \\
=&\; \partial_s \left(\sigma^\J (\ps^{n+2} \kappa_t^\J  + P_{n+3}'(\ps^n \kappa_t^\J)) + \lambda^\J_t \ps^{n+1} \kappa_t^\J\right) + (\sigma^\J (\kappa^\J_t)^2 - \partial_s \lambda_t^\J) \partial_s^{n+1} \kappa_t^\J \\
=&\; \sigma^\J \left(\ps^{n+3} \kappa_t^\J  + P_{n+4}'(\ps^{n+1} \kappa_t^\J)\right) + \lambda^\J_t \ps^{n+2} \kappa_t^\J
\end{align*} 
which is \eqref{simple-kappa} replaced $n$ by $n+1$. 

The boundary condition \eqref{bc-ks2} also can be proved inductively. 
When $n=0$, the condition \eqref{bc-ks2} is proved by \eqref{bc-kappa}. 
Assume \eqref{bc-ks2} holds for any $0 \le n \le m$. 
We then obtain applying \eqref{simple-kappa} 
\[ \ps^{2(m+1)} \kappa^\J_t = \frac{1}{\sigma^\J} \left(\pt \ps^{2m} \kappa_t^\J - \lambda^\J_t \ps^{2m+1} \kappa_t^\J\right) +  P_{2m+3}'(\ps^{2m} \kappa_t^\J) = P_{2m+3}'(\ps^{2m} \kappa_t^\J) \]
since $\pt \ps^{2m} \kappa_t^\J = 0$ can be obtained by taking time derivative of \eqref{bc-ks2} with $n=m$ and $\lambda^\J_t = 0$ is already proved in \eqref{bc-kappa}. 
We now note that, for each monomial $C\prod (\ps^l \kappa_t^\J)^{A_l}$ in $P_{2m+3}'$, the exponent $A_l$ is nonzero at least for one even $l$ since $\sum(l+1) A_l = 2m+3$ is odd, which implies $P_{2m+3}'(\ps^{2m} \kappa_t^\J) = 0$ due to \eqref{bc-ks2} with $0 \le n \le m$. 
We thus obtain the conclusion. 
\end{proof}

We next introduce some identities to extend \eqref{bc-sum-kappa} and \eqref{bc-diffe-kappa} to higher order boundary conditions at the junction point. 
Since we will take higher order time derivatives of the original boundary conditions, we will need some formulas of time derivatives of $\sigma^\J, f^\J, L^\J$ and so on. 

We here note a formula to make later calculations simple as follows. 

\begin{lem} 
For any smooth function $f(t)$ and non-negative integers $A_0, \cdots, A_i \in \mathbb{N} \cup \{0\}$ with $\sum_{l=0}^i (l+1)A_l = h$, 
a smooth geometric flow governed by \eqref{eq-curve}--\eqref{bc-boundary} fulfills 
\begin{equation}\label{deri-ene-poly}
\begin{aligned}
&\; \pt \int_{\Gamma_t^\J} f(t) \prod_{l=0}^i (\ps^l \kappa_t^\J)^{A_l} \; ds \\
=&\; \int_{\Gamma_t^\J} \pt f(t) \prod_{l=0}^i (\ps^l \kappa_t^\J)^{A_l} + \sigma^\J f(t) P_{h+2}(\ps^{i+2} \kappa_t^\J) \; ds + f(t) \lambda_t^\J \prod_{l=0}^i (\ps^l \kappa_t^\J)^{A_l} \Big\lfloor_{\text{at} \; \vec{a}}
\end{aligned}
\end{equation}
\end{lem} 

\begin{proof}
By means of \eqref{jacobi-deri-t} and \eqref{simple-kappa}, we have by a simple calculation 
\begin{align*}
&\; \pt \int_{\Gamma_t^\J} f(t) \prod_{l=0}^i (\ps^l \kappa_t^\J) ^{A_l} \; ds \\
=&\; \int_{\Gamma_t^\J} \pt f(t) \prod_{l=0}^i (\ps^l \kappa_t^\J) ^{A_l} + f(t) \sum_{n=0}^i\left(\prod_{j \neq n} (\ps^l \kappa_t^\J)^{A_l}\right) A_n (\ps^n \kappa_t^\J)^{A_n - 1} \pt \ps^n \kappa_t^\J \\
&\; \qquad + f(t) \prod_{l=0}^i (\ps^l \kappa_t^\J)^{A_l} (-\sigma^\J (\kappa_t^\J)^2 + \ps \lambda_t^\J) \; ds \\
=&\; \int_{\Gamma_t^\J} \pt f(t) \prod_{l=0}^i (\ps^l \kappa_t^\J)^{A_l} + \sigma^\J f(t) \sum_{n=0}^i P_{h-(n+1)}(\ps^i \kappa) (\ps^{n+2} \kappa_t^\J + P_{n+3}(\ps^n \kappa_t^\J)) \\
&\; \qquad + f(t) \sum_{n=0}^i \left(\prod_{j \neq n} (\ps^l \kappa_t^\J)^{A_l}\right) A_n (\ps^n \kappa_t^\J)^{A_n - 1} \lambda_t^\J \ps^{n+1} \kappa_t^\J \\
&\; \qquad + f(t) \prod_{l=0}^i (\ps^l \kappa_t^\J)^{A_l} \ps \lambda_t^\J + f(t) P_{l+2}(\ps^i \kappa_t^\J) \; ds \\
=&\; \int_{\Gamma_t^\J} \pt f(t) \prod_{l=0}^i (\ps^l \kappa_t^\J)^{A_l} + \sigma^\J f(t) P_{l+2} (\ps^{i+2} \kappa) + f(t) \ps \left(\prod_{l=0}^i (\ps^l \kappa_t^\J)^{A_l} \lambda_t^\J\right) \; ds. 
\end{align*}
The boundary condition $\lambda_t^\J\lfloor_{\text{at} \; P^\J} = 0$ as in \eqref{bc-kappa} can be applied to the last term to obtain \eqref{deri-ene-poly}. 
\end{proof}

We next introduce a formula of time derivatives of $L^\J$. 

\begin{lem}\label{lem:hi-deri-L}
For any smooth geometric flow governed by \eqref{eq-curve}--\eqref{bc-boundary}, the derivative of the length $\pt^n L^\J$, for any $n \in \mathbb{N}$ and $j \in \{1,2,3\}$, is a finite sum of 
\begin{equation}\label{hi-deri-L}
Q(\vec{\sigma}, \pt \Delta \vec{\alpha})R(\vec{\Theta})\prod_{k=1}^3 \prod_{l=1}^{n} \left(\int_{\Gamma_t^\K} P_{2l} (\ps^{2l-2} \kappa_t^\K) \; ds \right)^{A_l^\K} \cdot \prod_{l' = 1}^{n} \left(P_{2l'-1}(\ps^{2l'-2}\vec{\kappa}_t)\lfloor_{\text{at} \; \vec{a}} \right)^{B_{l'}} \cdot \prod_{k=1}^3 (L^\K)^{C^\K}, 
\end{equation}
where $A_l^\K, B_{l'}$ and $C^\K$ are non-negative integers satisfying 
\[ \sum_{k=1}^3 \sum_{l=1}^{n} A_l^\K + \sum_{l'=1}^{n} B_{l'} \ge 1, \quad \sum_{k=1}^3 \sum_{l=1}^{n} 2l A_l^\K + \sum_{l'=1}^{n} (2l' - 1) B_{l'} \le 2n. \]
\end{lem}

\begin{remark}
Applying the boundedness of $L^\J$ which follows from Lemma \ref{lem:dec-E} to Lemma \ref{lem:hi-deri-L}, we obtain 
\begin{equation}\label{bdd-deri-L}
|\pt^n L^\J| \le P_{\le 2n}(\|\ps^{2n-2} \vec{\kappa}_t\|) 
\end{equation}
for any $j \in \{1,2,3\}$, and thus $\sum 2lA_l^\K + \sum (2l'-1)B_{l'}$ corresponds to the geometric order. 
\end{remark}

\begin{proof}
We prove inductively. 
We obtain by a simple calculation as in the proof of Lemma \ref{lem:dec-E} 
\begin{equation}\label{hi-deri-L1}
\pt L^\J = - \int_{\Gamma_t^\J} \sigma^\J (\kappa_t^\J)^2 \; ds + \lambda_t^\J \lfloor_{\text{at} \; \vec{a}}   
\end{equation}
and thus the claim holds for $n=1$ due to \eqref{simple-l-V}. 
We now assume Lemma \ref{lem:hi-deri-L} holds for some $n \in \mathbb{N}$. 
We can see by \eqref{eq-alpha}, \eqref{eq-theta1} and \eqref{simple-l-V}
\begin{align*} 
\pt \left( Q(\vec{\sigma}, \pt \Delta \vec{\alpha})R(\vec{\Theta}) \right)=&\; Q(\vec{\sigma}, \pt \Delta \vec{\alpha})R(\vec{\Theta}) \left(1+ \sum_{k=1}^3 (\pt^2 \Delta^\K \alpha + \pt \Theta^\K \lfloor_{\text{at} \; \vec{a}} ) \right) \\
=&\; Q(\vec{\sigma}, \pt \Delta \vec{\alpha})R(\vec{\Theta}) \left(1+ \sum_{k=1}^3 ((L^\K)^2 + \pt L^\K + P_2(\ps \vec{\kappa}_t)\lfloor_{\text{at} \; \vec{a}}) \right) \\
=&\; Q(\vec{\sigma}, \pt \Delta \vec{\alpha})R(\vec{\Theta}) \left(1 + \sum_{k=1}^3 \left((L^\K)^2 + \int_{\Gamma_t^\K} (\kappa_t^\K)^2 \; ds \right) + P_2(\ps \vec{\kappa}_t)\right),  
\end{align*}
which implies that, if we take time derivative of $Q(\vec{\sigma}, \pt \Delta \vec{\alpha})R(\vec{\Theta})$ in \eqref{hi-deri-L}, the entire product is a finite sum of \eqref{hi-deri-L} replaced $n$ by $n+1$. 
Notice also that, since $\sum 2l A_l^\K + \sum (2l' - 1) B_{l'}$ corresponds to the geometric order of the entire product \eqref{hi-deri-L}, the above identity implies that $\sum 2l A_l^\K + \sum (2l' - 1) B_{l'}$ increase at most $2$ if we take time derivative of $Q(\vec{\sigma}, \pt \Delta \vec{\alpha})R(\vec{\Theta})$ in \eqref{hi-deri-L}. 
Similarly, we obtain by \eqref{deri-ene-poly}
\begin{align*}
\pt \int_{\Gamma_t^\K} P_{2l} (\ps^{2l-2} \kappa_t^\K) \; ds = \int_{\Gamma_t^\K} \sigma^\K P_{2l+2} (\ps^{2l} \kappa_t^\K) \; ds  + Q(\vec{\sigma}, \pt \Delta \vec{\alpha})R(\vec{\Theta}) P_{2l+1} (\ps^{2l-2} \vec{\kappa}_t)\lfloor_{\text{at} \; \vec{a}}
\end{align*}
and by \eqref{simple-kappa}
\begin{align*}
\pt P_{2l'-1}(\ps^{2l'-2}\vec{\kappa}_t)\lfloor_{\text{at} \; \vec{a}} = Q(\vec{\sigma}, \pt \Delta \vec{\alpha})R(\vec{\Theta}) P_{2l'+1}(\ps^{2l'}\vec{\kappa}_t)\lfloor_{\text{at} \; \vec{a}}. 
\end{align*}
We thus, applying \eqref{hi-deri-L1}, Lemma \ref{lem:hi-deri-L} holds replaced $n$ by $n+1$. 
We here note that the increasing rate of the geometric order of $P_{2l}, P_{2l'-1}$ and the highest differential order of $\ps^{2l-2} \kappa_t^\K, \ps^{2l'-2} \vec{\kappa}_t$ can be seen from the last two identities. 
Notice also the increasing rate related to the geometric order of each part is at most $2$ per time derivative, and thus we have $\sum_{l=1}^{n} 2l A_l^\K + \sum_{l'=1}^{n} (2l' - 1) B_{l'} \le 2n$. 
\end{proof}

We next consider the higher order derivatives of $\alpha^\J$. 

\begin{lemma}\label{lem:hi-deri-a}
For any smooth geometric flow governed by \eqref{eq-curve}--\eqref{bc-boundary}, the derivative of the misorientation $\pt^n \Delta^\J \alpha$, for any $n \in \mathbb{N}$ and $j \in \{1,2,3\}$, is a finite sum of 
\begin{equation}\label{hi-deri-a}
\prod_{j=1}^3 \left(\prod_{l=0}^{n-1} (\pt^l L^\J)^{A_l^\J} \right) Q(\vec{\sigma}) \quad \text{with} \; \; \sum_{j=1}^3 \sum_{l=0}^{n-1} (l+1) A_l^\J = n, 
\end{equation}
where $\hat{P}(\sigma^\J)$ is a polynomial in $\pa^m \sigma^\K$ with $m \in \mathbb{N} \cup \{0\}$ and $k \in \{1,2,3\}$ such that every monomial it contains is non-constant. 
\end{lemma}

\begin{proof}
For $n=1$, the claim holds by means of \eqref{eq-alpha}.
The differential equation \eqref{eq-alpha} implies also 
\begin{align*}
\pt \left(\prod_{l=0}^{n-1} (\pt^l L^\J)^{A_l^\J} \right) Q(\vec{\sigma}) =&\; \sum_{m=0}^{n-1} \left(\prod_{l \neq m} (\pt^l L^\J)^{A_l^\J}\right) \cdot A_m (\pt^m L^\J)^{A_m^\J - 1} \pt^{m+1} L^\J \cdot Q(\vec{\sigma}) \\
&\; + \left(\prod_{l=0}^{n-1} (\pt^l L^\J)^{A_l^\J} \right) \cdot Q(\vec{\sigma}) \sum_{k=1}^3 \pt \Delta^\J \alpha \\
=&\; \sum_{m=0}^{n-1} \left(\prod_{l \neq m} (\pt^l L^\J)^{A_l^\J}\right) \cdot A_m (\pt^m L^\J)^{A_m^\J - 1} \pt^{m+1} L^\J \cdot Q(\vec{\sigma}) \\
&\; + \left(\prod_{l=0}^{n-1} (\pt^l L^\J)^{A_l^\J} \right) \cdot Q(\vec{\sigma}) \sum_{k=1}^3 L^\K, 
\end{align*}
which show that the form of each term is preserved in the sense of \eqref{hi-deri-a} and the increase rate of $\sum (l+1) A_l$ is just $1$ per time derivative. 
We thus see that Lemma \ref{lem:hi-deri-a} replaced $n$ by $n+1$ holds if it holds for some $n \in \mathbb{N}$. 
\end{proof}

We now derive higher order boundary conditions from \eqref{bc-sum-kappa} and \eqref{bc-diffe-kappa}. 

\begin{lem}\label{lem:bc-ks2}
For any smooth geometric flow governed by \eqref{eq-curve}--\eqref{bc-boundary}, $n \in \mathbb{N}$ and $j \in \{1,2,3\}$, there exist $I_{2n}^\J$ and $I_{2n+1}^\J$ such that 
\begin{align}
&\sum_{j=1}^3 \left((\sigma^\J)^{n+2} \ps^{2n} \kappa_t^\J\lfloor_{\text{at} \; \vec{a}} + I_{2n}^\J\right) = 0, \label{bc-sum-ks2}\\
&\begin{aligned} 
&(\sigma^\J)^{n+1} \ps^{2n+1} \kappa_t^\J\lfloor_{\text{at} \; \vec{a}} + I_{2n+1}^\J \\
&= (\sigma^\JJ)^{n+1} \ps^{2n+1} \kappa_t^\JJ \lfloor_{\text{at} \; \vec{a}} + I_{2n+1}^\JJ + \pt^n f^\J \quad \text{for} \; \; j \in \{1,2,3\}, 
\end{aligned} \label{bc-diffe-ks2}
\end{align}
where $f^\J$ is the function defined by \eqref{def-f}, and $I_{2n}^\J$ and $I_{2n+1}^\J$ are represented by 
\begin{align}
&\begin{aligned}
I_{2n}^\J =&\; n(\sigma^\J)^{n+1} \lambda_t^\J \ps^{2n-1} \kappa_t^\J\lfloor_{\text{at} \; \vec{a}} + Q(\vec{\sigma}) R(\vec{\Theta}) P_{2n+1}(\ps^{2n-2} \vec{\kappa}_t)\lfloor_{\text{at} \; \vec{a}} \\
&\; + \sum_{m=1}^{n} \hat{Q}_m(\vec{\sigma}, \pt^m \Delta \vec{\alpha}) R(\vec{\Theta}) P_{2n+1-2m}(\ps^{2n-2m} \vec{\kappa}_t)\lfloor_{\text{at} \; \vec{a}}, 
\end{aligned} \label{def-I2n}\\
&I_{2n+1}^\J =Q(\vec{\sigma}) R(\vec{\Theta}) P_{2n+2}(\ps^{2n} \vec{\kappa}_t)\lfloor_{\text{at} \; \vec{a}} + \sum_{m=1}^n \hat{Q}_m(\vec{\sigma}, \pt^m \Delta \vec{\alpha}) R(\vec{\Theta}) P_{2n+2-2m}(\ps^{2n+1-2m} \vec{\kappa}_t)\lfloor_{\text{at} \; \vec{a}}. \label{def-I2n1}
\end{align}
Furthermore, $\pt^n f^\J$ is a finite sum of 
\[ \frac{Q_{n+1}(\vec{\sigma}, \pt^{n+1} \Delta \vec{\alpha})}{(1-\langle \tau_t^\J, \tau^\JJ_t \rangle)^{(2m_1+1)/2}(\sigma^\J \sigma^\JJ)^{m_2}} \quad \text{with} \; \; m_1, m_2 \in \mathbb{N}. \]
\end{lem}

\begin{proof}
We first prove the identity \eqref{bc-sum-ks2}. 
Taking the time derivative of \eqref{bc-sum-kappa} and applying \eqref{simple-kappa}, we have 
\begin{align*} 
0 =&\; \sum_{j=1}^3 \left(2 \sigma^\J \pa \sigma^\J (\pt \Delta^\J \alpha) \kappa^\J_t + (\sigma^\J)^3 (\ps^2 \kappa_t^\J + P_3(\kappa^\J_t)) + (\sigma^\J)^2 \lambda_t^\J \ps \kappa_t^\J \right)\Big\lfloor_{\text{at} \; \vec{a}}, 
\end{align*}
which implies that \eqref{bc-sum-ks2} with $n=1$ holds. 
Assume \eqref{bc-sum-ks2} holds for some $n\in\mathbb{N}$. 
We then take the time derivative of \eqref{bc-sum-ks2} and apply \eqref{simple-l-V}, \eqref{simple-theta} and \eqref{simple-kappa} to obtain 
\begin{align*}
0 = &\; \sum_{j=1}^3 \pt((\sigma^\J)^{n+2} \ps^{2n} \kappa_t^\J)\lfloor_{\text{at} \; \vec{a}} + \pt (n(\sigma^\J)^{n+1} \lambda_t^\J \ps^{2n-1} \kappa_t^\J)\lfloor_{\text{at} \; \vec{a}}  \\
&\; \qquad + \pt(Q(\vec{\sigma}) P_{2n+1}(\ps^{2n-2} \kappa_t^\J))\lfloor_{\text{at} \; \vec{a}} + \sum_{m=1}^{n} \pt(\hat{Q}_m(\vec{\sigma}, \pt^m \Delta \vec{\alpha}) R(\vec{\Theta}) P_{2n+1-2m}(\ps^{2n-2m} \vec{\kappa}_t))\lfloor_{\text{at} \; \vec{a}} \\
= &\; \sum_{j=1}^3 \Big\{\hat{Q}_1(\vec{\sigma}, \pt \Delta \vec{\alpha}) \ps^{2n} \kappa_t^\J + (\sigma^\J)^{n+3} \ps^{2n+2} \kappa_t^\J + (\sigma^\J)^{n+2} \lambda_t^\J \ps^{2n+1} \kappa_t^\J + Q(\vec{\sigma}) P_{2n+3}(\ps^{2n} \vec{\kappa}_t) \\
&\; \qquad + \hat{Q}_1 (\vec{\sigma}, \pt \Delta \vec{\alpha}) R(\vec{\Theta}) P_{2n+1}(\ps^{2n-1} \vec{\kappa}_t) + Q(\vec{\sigma}) R(\vec{\Theta}) P_{2n+3}(\ps^{2n} \vec{\kappa}_t) + n (\sigma^\J)^{n+2} \lambda_t^\J \ps^{2n+1} \kappa_t^\J \\
&\; \qquad + \hat{Q}_1 (\vec{\sigma}, \pt \Delta \vec{\alpha}) R(\vec{\Theta}) P_{2n+1}(\ps^{2n-2} \vec{\kappa}_t) + Q(\vec{\sigma}) R(\vec{\Theta}) P_{2n+3}(\ps^{2n} \vec{\kappa}_t) \\
&\; \qquad + \sum_{m=1}^n \Big(\hat{Q}_{m+1}(\vec{\sigma}, \pt^{m+1} \Delta \vec{\alpha}) R(\vec{\Theta}) P_{2n+1-2m}(\ps^{2n-2m} \vec{\kappa}_t) \\
&\; \qquad \qquad + \hat{Q}_m(\vec{\sigma}, \pt^m \Delta \vec{\alpha}) R(\vec{\Theta}) P_{2n+3-2m}(\ps^{2n+2-2m} \vec{\kappa}_t)\Big)\Big\}\Big\lfloor_{\text{at} \; \vec{a}} \\
=&\; \sum_{j=1}^3 \Big\{(\sigma^\J)^{n+3} \ps^{2n+2} \kappa_t^\J + (n+1) (\sigma^\J)^{n+2} \lambda_t^\J \ps^{2n+1} \kappa_t^\J + Q(\vec{\sigma}) R(\vec{\Theta}) P_{2n+3}(\ps^{2n} \vec{\kappa}_t) \\
&\; \qquad + \sum_{m=1}^{n+1} \hat{Q}_m(\vec{\sigma}, \pt^m \Delta \vec{\alpha}) R(\vec{\Theta}) P_{2n+3-2m}(\ps^{2n+2-2m} \vec{\kappa}_t)\Big\}\Big\lfloor_{\text{at} \; \vec{a}}, 
\end{align*}
which coincides with \eqref{bc-sum-ks2} replaced $n$ by $n+1$. 
We thus obtain \eqref{bc-sum-ks2} for any $n \in \mathbb{N}$. 

The identity \eqref{bc-diffe-ks2} follows from the time derivatives of \eqref{bc-diffe-kappa} applying a similar calculation to obtain \eqref{bc-sum-ks2} and the form of $\pt^n f^\J$ can be obtained easily since $f^\J$ is of the form 
\[ f^\J = \frac{\hat{Q}_1(\vec{\sigma}, \pt \Delta \vec{\alpha})}{(1-\langle \tau_t^\J, \tau^\JJ_t \rangle)^{1/2}(\sigma^\J \sigma^\JJ)^2}, \]
the formula $\langle \tau_t^\J, \tau^\JJ_t \rangle$ is given by \eqref{inner-tau} and the calculus rule \eqref{deri-Q} holds. 
\end{proof}

\begin{remark}\label{rmk:bdd-lf}
Applying \eqref{bdd-deri-L} to Lemma \ref{lem:hi-deri-a}, due to the boundedness of $L^\J$, we may see that each term of $|\pt^n \Delta \alpha^\J|$, which is the form of absolute value of \eqref{hi-deri-a}, is bounded by some constant (if the term is a product of only lengths $L^\J$ and $Q(\vec{\sigma})$) or 
\begin{align*} 
&\prod_{j=1}^3 \prod_{l=1}^{n-1} \left(P_{\le 2l}(\|\ps^{2l-2} \vec{\kappa}_t\|) \right)^{A_l^\J} \le  P_{\le \sum_j \sum_l 2l A_l^\J}(\| \ps^{2n-4} \vec{\kappa}_t\|)
\end{align*} 
if $n \ge 2$. 
Since 
\[ \sum_{j=1}^3 \sum_{l=1}^{n-1} 2l A_l^\J \le 2 \sum_{j=1}^3 \sum_{l=0}^{n-1} (l+1) A_l^\J \le 2n, \] 
we have 
\begin{equation}\label{bdd-deri-alpha} |\pt^n \Delta \alpha^\J| \le 
\begin{cases}
M & \text{if} \; \; n=1, \\
M + P_{\le 2n}(\|\ps^{2n-4} \vec{\kappa}_t\|) & \text{if} \; \; n \ge 2 
\end{cases}\end{equation}
for some constant $M$, and thus we also obtain 
\begin{equation}\label{bdd-deri-I} 
|\hat{Q}_{l_1}(\vec{\sigma}, \pt^{i_1} \Delta \vec{\alpha}) P_{l_2}(\ps^{i_2} \vec{\kappa}_t)| \le  P_{\le 2l_1 + l_2} (\|\ps^{\max\{2i_1-4, i_2\}} \vec{\kappa}_t\|), 
\end{equation} 
which will be apply to estimate terms related to $I_l^\J$ in Lemma \ref{lem:bc-ks2}. 
Further applying \eqref{bdd-sin2}, we can obtain 
\begin{equation} \label{bdd-deri-f}
|\pt^n f^\J| \le \begin{cases}
M & \text{if} \; \; n=1, \\
M + P_{\le 2n+2} (\| \ps^{2n-2} \vec{\kappa}_t\|) & \text{if} \; \; n \ge 2 
\end{cases} 
\end{equation}
for some constant $M$ and any $j \in \{1,2,3\}$. 
\end{remark}

In section \ref{subsec:decay-l2}, we applied Lemma \ref{lem:ine-trace} to control the terms on the right hand side of \eqref{ene-ine-kappa} except the negative term. 
In this section, we will apply the following Gagliardo-Nirenberg interpolation inequalities instead of Lemma \ref{lem:ine-trace} to the higher order estimates. 
We refer to \cite{Ad,Au} for the details of the interpolation inequalities. 

\begin{prop}\label{prop:interpolation}
Let $\Gamma$ be a smooth plane curve with finite length $L$. 
If $u$ is a smooth function defined on $\Gamma$, $m \ge 1$, $p \in [2,\infty]$ and $n \in \{0,1, \cdots, m-1\}$, then there exist constants $C_1$ and $C_2$ are independent of $\Gamma$ such that 
\begin{equation}\label{interpolation}
\|\partial_s^n u\|_{L^p} \le C_1 \|\partial_s^m u\|_{L^2}^\rho \|u\|_{L^2}^{1- \rho} + \frac{C_2}{L^{m \rho}} \|u\|_{L^2}, 
\end{equation}
where 
\[ \rho = 
\begin{cases}
\frac{n+1/2-1/p}{m} & \text{if} \quad p \in [2, \infty), \\
\frac{n+1/2}{m} & \text{if} \quad p=\infty. 
\end{cases} \]
\end{prop}

We further derive estimates to apply to the polynomials $P_l$ in the higher order estimates of the curvatures. 

\begin{lem}
Let $\Gamma$ be a smooth plane curve with finite length $L$. 
Let also $n \in \mathbb{N}\cup\{0\}$. 
Then, the following estimates hold.

(i) If $\kappa$ is a smooth function defined on $\Gamma$ and non-negative integers $A_l$ for $l=0,1, \cdots, n$ satisfy
\[ \sum_{l=0}^n A_l \ge 2, \quad \sum_{l=0}^n (l+1) A_l \le 2n+4. \]
Then, for any $\varepsilon > 0$, there exist constants $C_3>0$ and $q_1 > 0$ such that 
\begin{equation}\label{ine-inter1}
\int_{\Gamma} \prod_{l=0}^n |\ps^l \kappa|^{A_l}  \le \varepsilon \int_{\Gamma} |\ps^{n+1} \kappa|^2 + |\kappa|^2 \; ds + C_3 \left(\int_{\Gamma} |\kappa|^2 \; ds \right)^{q_1}. 
\end{equation}

(ii) If $\kappa$ is a smooth function defined on $\Gamma$ and non-negative real values $A_l$ for $l=0,1, \cdots, n$ satisfy 
\[ \sum_{l=0}^n A_l > 0, \quad \sum_{l=0}^n (l+1) A_l \le 2n+3. \]
Then, for any $\varepsilon > 0$, there exist constants $C_4>0$ and $q_2>0$ such that 
\begin{equation}\label{ine-inter2}
\prod_{l=0}^n \|\ps^l \kappa \|_{L^\infty}^{A_l} \le \varepsilon \int_{\Gamma} |\ps^{n+1} \kappa|^2 + |\kappa|^2 \; ds + C_4 \left(\int_{\Gamma} |\kappa|^2 \; ds \right)^{q_2}.  
\end{equation}

(iii) If $\kappa^\J$ are smooth function defined on $\Gamma$ for $j = 1,2,3$, $n \ge 1$ and non-negative real values $A_l^\K$ for $l=0,1, \cdots, 2n$ and $k=1,2,3$ satisfy 
\[ \sum_{k=1}^3 \sum_{l=0}^{2n-2} (l+1) A_l^\K \le 2n. \]
Then, for any $\varepsilon > 0$ and $t>0$, there exist constants $C_5, q_3>0$ and $q_4 > -1$ such that 
\begin{equation}\label{ine-inter3}
\begin{aligned}
&\; t^{2n-1} \left\| |\ps^{2n} \kappa^\J| \cdot \prod_{k=1}^3 \prod_{l=0}^{2n-2} |\ps^l \kappa^\K|^{A_l^\K} \right\|_{L^\infty} \\
\le&\; \sum_{k=1}^3 \varepsilon t^{2n} \int_{\Gamma} |\ps^{2n+1} \kappa^\K|^2 + |\kappa^\K|^2 \; ds + C_5 t^{q_4} \left(\int_{\Gamma} |\kappa^\K|^2 \; ds \right)^{q_3}  
\end{aligned}
\end{equation}
for any $j =1,2,3$. 
\end{lem}

\begin{proof}
We first discuss the case (i). 
By the H\"{o}lder inequality, we have 
\[ \int_{\Gamma} \prod_{l=0}^n |\ps^l \kappa|^{A_l} \; ds \le \prod_{l=0}^n \left( \int_{\Gamma} |\ps^l \kappa|^{A_l B_l} \right)^{1/B_l} = \prod_{l=0}^n \|\ps^l \kappa\|_{L^{A_l B_l}}^{A_l}, \]
where the exponents $B_l$ satisfy $\sum 1/B_l = 1$ and $A_l B_l \ge 2$ for every $l \in \{0, \cdots, n\}$ such that $A_l \neq 0$. 
Indeed, letting 
\[ B_l := \begin{cases}
1 & \text{if} \; \; \#\{l: A_l \neq 0\} = 1, \; \; A_l \neq 0, \\
m & \text{if} \; \; \#\{l: A_l \neq 0\} = m, \; \; A_l \neq 0, 
\end{cases} \] 
where $\#\{l: A_l \neq 0\}$ is the number of non-zero exponents $A_l$, we can easily see that the exponents $B_l$ satisfy the conditions by applying $\sum_{l=0}^n A_l \ge 2$. 
Therefore, due to \eqref{interpolation}, we have 
\begin{equation}\label{ine-inter11}
\begin{aligned} 
\int_{\Gamma} \prod_{l=0}^n |\ps^l \kappa|^{A_l} \; ds \le&\; C \prod_{l=0}^n \left( \|\ps^{n+1} \kappa\|_{L^2} + \|\kappa\|_{L^2} \right)^{\rho_l A_l} \|\kappa\|_{L^2}^{(1-\rho_l) A_l} \\
\le&\; C \left( \|\ps^{n+1} \kappa\|_{L^2} + \|\kappa\|_{L^2} \right)^{\sum_{l=0}^n \rho_l A_l}  \|\kappa\|_{L^2}^{\sum_{l=0}^n (1-\rho_l) A_l}, 
\end{aligned}
\end{equation}
where $\rho_l = \frac{l + 1/2 - 1/(A_l B_l)}{n+1}$ and $C$ is a some constant. 
Notice that $\sum_{l=0}^n A_l > 2$ or $\sum_{l=0}^n (l+1) A_l < 2n+4$ holds under the assumptions $2 \le \sum_{l=0}^n A_l$ and $\sum_{l=0}^n (l+1) A_l \le 2n+4$. 
We then have 
\[ \sum_{l=0}^n \rho_l A_l = \sum_{l=0}^n \frac{(l+1/2) A_l - 1/B_l}{n+1} = \frac{\sum_{l=0}^n (l+1) A_l - \sum_{l=0}^n A_l/2 - 1}{n+1} < 2. \]
Therefore, \eqref{ine-inter11} can be continued to estimate by applying Young's inequality 
\[ \int_{\Gamma} \prod_{l=0}^n |\ps^l \kappa|^{A_l} \; ds \le \varepsilon \int_{\Gamma} |\ps^{n+1} \kappa|^2 + |\kappa|^2 \; ds + C_3 \left(\int_{\Gamma} |\kappa|^2 \; ds\right)^{\frac{2\sum_{l=0}^n (1-\rho_l) A_l}{2-\sum_{l=0}^n \rho_l A_l}} \]
and also the positivity of the exponent $(2\sum (1-\rho_l) A_l)/(2-\sum \rho_l A_l)$ can be seen. 
Letting $q_1 = (2\sum (1-\rho_l) A_l)/(2-\sum \rho_l A_l)$, we have \eqref{ine-inter1}. 

For the case (ii), we have by applying \eqref{interpolation} 
\begin{align*} 
\left\| \prod_{l=0}^n |\ps^l \kappa|^{A_l}  \right\|_{L^\infty} =&\; \prod_{l=0}^n \|\ps^l \kappa\|_{L^\infty}^{A_l} \le C \prod_{l=0}^n \left( \|\ps^{n+1} \kappa\|_{L^2} + \|\kappa\|_{L^2} \right)^{\rho_l A_l} \|\kappa\|_{L^2}^{(1-\rho_l) A_l} \\
\le&\; C \left( \|\ps^{n+1} \kappa\|_{L^2} + \|\kappa\|_{L^2} \right)^{\sum_{l=0}^n \rho_l A_l}  \|\kappa\|_{L^2}^{\sum_{l=0}^n (1-\rho_l) A_l}, 
\end{align*}
where $\rho_l = \frac{l + 1/2}{n+1}$ and $C$ is a some constant. 
Since  $\sum A_l \ge \sum (l+1)A_l/(n+1)$, we have by $\sum (l+1) A_l \le 2n+3$
\begin{align*} 
\sum_{l=0}^n \rho_l A_l =&\; \frac{\sum_{l=0}^n (l+1) A_l - \frac{1}{2} \sum_{l=0}^n A_l}{n+1} \le \frac{\sum_{l=0}^n (l+1)A_l - \frac{1}{2} \sum_{l=0}^n(l+1)A_l/(n+1)}{n+1} \\
=&\; \frac{(2n+1) \sum_{l=0}^n (l+1)A_l}{2(n+1)^2} \le \frac{(2n+1) (2n+3)}{2(n+1)^2} \le 2 - \frac{1}{2(n+1)^2} < 2. 
\end{align*}
Therefore, we can obtain \eqref{ine-inter2} by Young's inequality as in the case (i). 

The estimate in the case (iii) also can be proved a similar argument.
Due to the calculations in the case (ii), we have 
\begin{align*} 
&\; t^{2n-1} \left\| |\ps^{2n} \kappa^\J| \cdot \prod_{k=1}^3 \prod_{l=0}^{2n-2} |\ps^l \kappa^\K|^{A_l^\K} \right\|_{L^\infty} \\
\le&\; C (t^n\|\ps^{2n+1} \kappa^\J \|_{L^2} + t^n \|\kappa^\J \|_{L^2})^{\rho} \left(\prod_{k=1}^3 \left( t^n \|\ps^{2n+1} \kappa^\K \|_{L^2} + t^n \|\kappa^\K \|_{L^2} \right)^{\sum_{l=0}^{2n-2} \rho_l^\K A_l^\K}\right) \cdot \\
&\; \quad  t^{2n-1-n(\rho + \sum_{k=1}^3 \sum_{l=0}^{2n-2} \rho_l^\K A_l^\K)} \|\kappa^\J\|_{L^2}^{1-\rho} \cdot \prod_{k=1}^3 \|\kappa^\K\|_{L^2}^{\sum_{l=0}^{2n-2} (1-\rho_l^\K) A_l^\K}, 
\end{align*} 
where $\rho = \frac{2n+1/2}{2n+1}$, $\rho_l^\K = \frac{l + 1/2}{2n+1}$ and $C$ is a some constant. 
Since the geometric order of the left hand side is not larger than $4n+1$, which is smaller than twice of the higher order of the derivative in right hand side plus $3$, we can see 
\[ \rho + \sum_{k=1}^3 \sum_{l=0}^{2n-2} \rho_l^\K A_l^\K < 2 \]
as in the case (ii). 
Therefore, applying Young's inequality, we have, for some $C_5, q_3 > 0$, 
\begin{align*} 
&\; t^{2n-1} \left\| |\ps^{2n} \kappa^\J| \cdot \prod_{k=1}^3 \prod_{l=0}^{2n-2} |\ps^l \kappa^\K|^{A_l^\K} \right\|_{L^\infty} \le \sum_{k=1}^3 \varepsilon t^{2n} \int_{\Gamma} |\ps^{2n+1} \kappa^\K|^2 + |\kappa^\K|^2 \; ds + C_5 t^{q_4} \left(\int_{\Gamma} |\kappa^\K|^2 \; ds \right)^{q_3}
\end{align*}
with 
\[ q_4 = \frac{2 (2n-1-n(\rho + \sum_{k=1}^3 \sum_{l=0}^{2n-2}\rho_l^\K A_l^\K))}{2- (\rho + \sum_{k=1}^3 \sum_{l=0}^{2n-2} \rho_l^\K A_l^\K)}. \]
Since $q_4>-1$ is equivalent to 
\begin{align*} 
0 <&\;  2 \left(2n-1-n\left(\rho + \sum_{k=1}^3 \sum_{l=0}^{2n-2} \rho_l^\K A_l^\K\right)\right) + 2- \left(\rho + \sum_{k=1}^3 \sum_{l=0}^{2n-2} \rho_l^\K A_l^\K\right) \\
=&\; 4n - (2n+1)\left(\rho + \sum_{k=1}^3 \sum_{l=0}^{2n-2} \rho_l^\K A_l^\K\right), 
\end{align*}
we prove the positivity. 
Due to $\sum_l A_l^\K \ge \sum_l (l+1)A_l/ (2n-1)$ and $\sum_k \sum_l (l+1)A_l^\K \le 2n$, we have 
\begin{align*}
&\; 4n - (2n+1)\left(\rho + \sum_{k=1}^3 \sum_{l=0}^{2n-2} \rho_l^\K A_l^\K\right) \\
=&\; 2n - \frac{1}{2} - \sum_{k=1}^3 \sum_{l=1}^{2n-2} (l+1) A_l^\K + \frac{1}{2} \sum_{k=1}^3 \sum_{l=1}^{2n-2} A_l^\K \\
\ge &\; 2n - \frac{1}{2} - \sum_{k=1}^3 \sum_{l=1}^{2n-2} (l+1) A_l^\K + \frac{1}{2} \sum_{k=1}^3 \sum_{l=1}^{2n-2} \frac{(l+1)A_l^\K}{2n-1} \\
=&\; 2n - \frac{1}{2} - \frac{4n-3}{4n-2} \sum_{k=1}^3 \sum_{l=1}^{2n-2} (l+1) A_l^\K \ge 2n - \frac{1}{2} - \frac{4n-3}{4n-2} \cdot 2n = \frac{1}{4n-2} > 0. 
\end{align*}
We thus obtain the conclusion. 
\end{proof}

We now prove the higher order estimate of the curvatures for any $k \in \mathbb{N}$ applying the exponential $L^2$-decay of the curvatures. 
Note that the weights $t^m$ in the following proposition enable us to obtain not only the decay estimate but also the smoothing effect, namely, the following estimate is independent of the derivatives of the initial curvatures. 

\begin{prop}\label{prop:decay-h2}
Assume (A1)--(A3). 
Let a family of $\{\Gamma^\J_t\}_{j\in \{1,2,3\}}$ and $\vec{\alpha}$ be a smooth geometric flow governed by \eqref{eq-curve}--\eqref{bc-boundary} in the time interval $[0,T)$. 
Let also $\Cr{min-length}$ be the constant in Lemma \ref{lem:min-length}. 
Then, for any $n \in \mathbb{N}$, there exist $\Cl[e]{e-small-initial2}, \Cl[c]{c-decay-h21}, \Cl[c]{c-decay-h22}, \Cl[c]{c-decay-h23}, \Cl[c]{c-decay-h24} > 0$ and $\Cl[c]{c-decay-h2}> -1$ such that if 
\begin{equation}\label{small-initial2} 
\begin{aligned}
& E(0) \le \sigma(0)\Cr{min-length}, \quad \sum_{j=1}^3 \left(\Delta^\J\alpha_0\right)^2 \le \Cr{e-small-initial2}, \quad \sum_{j=1}^3 \int_{\Gamma_0^\J} \left(\sigma(\Delta^\J\alpha_0)\right)^2 \left(\kappa^\J_0 \right)^2 ds \le \Cr{e-small-initial2}, 
\end{aligned}
\end{equation}
then 
\begin{equation}\label{ex-decay-h2} 
\begin{aligned}
&\; \sum_{j=1}^3 \int_{\Gamma^\J_t} \sum_{m=0}^{2n} \frac{t^m}{m!} (\sigma^\J)^{m+2} (\ps^m \kappa_t^\J)^2 \; ds \\
\le &\; \Cr{c-decay-h21} \int_0^t \left(\sum_{m=0}^{2n} \tilde{t}^m +\tilde{t}^{\Cr{c-decay-h2}}\right) e^{-\Cr{c-decay-h22}\tilde{t}} \; d\tilde{t} + \Cr{c-decay-h23} \sum_{m'=1}^n t^{m'} e^{-\Cr{c-decay-h24}t} +  2 \sum_{j=1}^3 \int_{\Gamma^\J_0} (\sigma^\J_0)^2 (\kappa_0^\J)^2 \; ds 
\end{aligned}
\end{equation}
for any $t \in [0,T)$.
\end{prop}

\begin{proof}
Note that the length $L^\J$ is uniformly positive and bounded for $j \in \{1,2,3\}$ due to Lemma \ref{lem:dec-E}, Lemma \ref{lem:min-length} and the first assumption in \eqref{small-initial2}. 
The uniformly boundedness of $|\pa^m \sigma^\J|$ for $m \in \mathbb{N} \cup \{0\}$ and $j \in \{1,2,3\}$ also follows from \eqref{mono-delta-a} and \eqref{small-initial2}. 
We will choose $\Cr{e-small-initial2}$ in \eqref{small-initial2} small so that Corollary \ref{cor:exdecrease}, Corollary \ref{cor:bdd-sin2} and Proposition \ref{prop:decay-l2} can be applied. 
In this setting, we can regard the functions $R$ and $Q$ as a bounded coefficient, as we remarked in Remark \ref{rmk:pro-polys}, and this boundedness will be used through this proof without citing. 
Note also that we can apply the estimates \eqref{bdd-deri-I} and \eqref{bdd-deri-f} in this setting. 
We first derive a higher order energy type identity. 
Due to \eqref{jacobi-deri-t}, \eqref{bc-kappa}, \eqref{simple-kappa} and \eqref{bc-ks2}, for any $j\in\{1,2,3\}$, we have by a similar calculation to \eqref{ene-ine-kappa1}
\begin{equation}\label{decay-h21}
\begin{aligned}
&\; \dfrac{d}{dt} \int_{\Gamma^\J_t} \sum_{m=0}^{2n} \frac{t^m}{m!} (\sigma^\J)^{m+2} (\ps^m \kappa_t^\J)^2 \; ds \\
=&\; -2 \int_{\Gamma_t^\J} \sum_{m=0}^{2n} \frac{t^m}{m!} (\sigma^\J)^{m+3} (\ps^{m+1} \kappa_t^\J)^2 \; ds + \int_{\Gamma^\J_t} \sum_{m=1}^{2n} \frac{t^{m-1}}{(m-1)!} (\sigma^\J)^{m+2} (\ps^m \kappa_t^\J)^2 \; ds \\
&\; + \int_{\Gamma_t^\J} \sum_{m=0}^{2n} t^m \Big\{ Q(\sigma^\J, \pt \Delta^\J \alpha) P'_{2m+2}(\ps^m \kappa_t^\J) + Q(\sigma^\J, \pt \Delta^\J \alpha) P'_{2m+4}(\ps^m \kappa_t^\J)\Big\} \; ds \\
&\; + \sum_{m=0}^{2n} \frac{t^m}{m!} \Big\{2 (\sigma^\J)^{m+3} \ps^{m+1} \kappa_t^\J \ps^m \kappa_t^\J + (\sigma^\J)^{m+2} \lambda_t^\J (\ps^m \kappa_t^\J)^2 \Big\} \Big\lfloor_{\text{at} \; \vec{a}}, 
\end{aligned}
\end{equation}
where $Q(\sigma^\J, \pt \Delta^\J \alpha)$ is a polynomial (for each term) in $\sigma^\J, \pa \sigma^\J$ and $\pt \Delta^\J \alpha$, and thus $Q$ is uniformly bounded. 
Due to \eqref{ine-inter1}, the third integral on the right hand side of \eqref{decay-h21} can be divided by a higher order term and the $L^2$-norm of the curvature as, applying $\sigma(0) \le \sigma^\J$, 
\begin{align*} 
&\; \int_{\Gamma_t^\J} \sum_{m=0}^{2n} t^m \Big\{ Q(\sigma^\J, \pt \Delta^\J \alpha) P'_{2m+2}(\ps^m \kappa_t^\J) + Q(\sigma^\J, \pt \Delta^\J \alpha) P'_{2m+4}(\ps^m \kappa_t^\J)\Big\} \; ds \\
\le &\; \frac{1}{2} \int_{\Gamma_t^\J} \sum_{m=0}^{2n} \frac{t^m}{m!} (\sigma^\J)^{m+3} (\ps^{m+1} \kappa_t^\J)^2 \; ds + M \sum_{m=0}^{2m} t^m \|\kappa_t^\J\|_{L^2}^{q_m'}
\end{align*}
for some positive constants $M, q_m'$ with $m \in \{0,1, \cdots, 2n\}$. 
Note that although we may obtain the terms $\|\kappa_t^\J\|_{L^2}^{q_m'}$ with some different exponents $q_m'$, we can here replace the different exponents by a same exponent $q_m'$, due to the exponential $L^2$-decay of the curvature as in Proposition \ref{prop:decay-l2}, by choosing $q_m'$ small and $M$ large depending on the weighted $L^2$-norm of the initial curvatures and the initial misorientations if necessary, for example, for any $0 < q'_{m,1} < q'_{m,2}$, we have 
\[ \|\kappa_t^\J\|_{L^2}^{q'_{m,1}} + \|\kappa_t^\J\|_{L^2}^{q'_{m,2}} = (1+ \|\kappa_t^\J\|_{L^2}^{q'_{m,2}-q'_{m,1}}) \|\kappa_t^\J\|_{L^2}^{q'_{m,1}} \le M \|\kappa_t^\J\|_{L^2}^{q'_{m,1}}, \]
where $M$ is a constant depending on the weighted $L^2$-norm of the initial curvatures and the initial misorientations.
Since, on the right hand side of \eqref{decay-h21}, the second integral can be absorbed into the first integral, we thus see that
\begin{equation}\label{decay-h22}
\begin{aligned} 
&\; \dfrac{d}{dt} \int_{\Gamma^\J_t} \sum_{m=0}^{2n} \frac{t^m}{m!} (\sigma^\J)^{m+2} (\ps^m \kappa_t^\J)^2 \; ds \\
\le &\; -\frac{1}{2} \int_{\Gamma_t^\J} \sum_{m=0}^{2n} \frac{t^m}{m!} (\sigma^\J)^{m+3} (\ps^{m+1} \kappa_t^\J)^2 \; ds + M \sum_{m=0}^{2m} t^m \|\kappa_t^\J\|_{L^2}^{q_m'}\\
&\; + \sum_{m=0}^{2n} \frac{t^m}{m!} \Big\{2 (\sigma^\J)^{m+3} \ps^{m+1} \kappa_t^\J \ps^m \kappa_t^\J + (\sigma^\J)^{m+2} \lambda_t^\J (\ps^m \kappa_t^\J)^2 \Big\} \Big\lfloor_{\text{at} \; \vec{a}}
\end{aligned}
\end{equation}
by applying the above inequality. 
Therefore, we continue to estimate the boundary terms. 
We now let 
\[ J_m^\J := \frac{t^m}{m!} \Big\{2 (\sigma^\J)^{m+3} \ps^{m+1} \kappa_t^\J \ps^m \kappa_t^\J + (\sigma^\J)^{m+2} \lambda_t^\J (\ps^m \kappa_t^\J)^2 \Big\} \Big\lfloor_{\text{at} \; \vec{a}} \]
for $m \in \{0,1, \cdots, 2n\}$ and $j \in \{1,2,3\}$. 

Hereafter we omit the trace operator $\lfloor_{\text{at} \; \vec{a}}$ since we discuss only the boundary terms at the junction point $\vec{a}$ for a while. 
The boundary term $J_0^\J$ can be estimated as in \eqref{ene-ine-kappa1}, \eqref{ene-ine-kappa2}. 
We note that the division as in \eqref{ex-decay-l23}, which is necessary to gain the exponent of $\|\kappa^\J\|_{L^2}$ in the proof of Proposition \ref{prop:decay-l2}, is not needed if we apply \eqref{bdd-deri-f}. 
We thus obtain, due to \eqref{bc-sum-kappa}, \eqref{bc-diffe-kappa}, \eqref{simple-sum}, \eqref{simple-l-V} and \eqref{bdd-deri-f}, 
\begin{equation}\label{decay-h23}
\begin{aligned}
\sum_{j=1}^3 J_0^\J \le&\; \sum_{j=1}^3 |Q(\vec{\sigma})| P_1(\|\vec{\kappa}_t\|) |f^\J| + |Q(\vec{\sigma})| \cdot  \big|\lambda_t^\J \big| P_2(\|\kappa_t^\J\|) \\
\le&\; |Q(\vec{\sigma})R(\vec{\Theta})| (P_1(\| \vec{\kappa}_t \|) + P_3(\| \vec{\kappa}_t\|))
\end{aligned}
\end{equation}
This kind of strategy can be applied to $J_{2m'}^\J$ for $m' \in \mathbb{N}$. 
Indeed, since \eqref{simple-l-V}, \eqref{def-I2n} and \eqref{def-I2n1} show
\begin{align*}
&(\sigma^\J)^{2m'+2} \lambda_t^\J (\ps^{2m'} \kappa_t^\J)^2 = Q(\vec{\sigma}) R(\vec{\Theta}) P_{4m'+3}(\ps^{2m'} \vec{\kappa}_t), \\
&(\sigma^\J)^{m'+2} \ps^{2m'} \kappa_t^\J I_{2m'+1}^\J = Q(\vec{\sigma}) R(\vec{\Theta}) P_{4m'+3}(\ps^{2m'} \vec{\kappa}_t) + \sum_{l=1}^{m'} \hat{Q}_{l}(\vec{\sigma}, \pt^l \Delta \vec{\alpha}) R(\vec{\Theta}) P_{4m'+3-2l}(\ps^{2m'} \vec{\kappa}_t), \\
&\begin{aligned}
(\sigma^\J)^{m'+1} \ps^{2m'+1} \kappa_t^\J I_{2m'}^\J =&\; m' (\sigma^\J)^{2m'+2} \lambda_t^\J \ps^{2m'+1} \kappa_t^\J \ps^{2m'-1} \kappa_t^\J \\
&\; + \ps^{2m'+1} \kappa^\J_t \Big(Q(\vec{\sigma}) R(\vec{\Theta}) P_{2m'+1}(\ps^{2m'-2}\vec{\kappa}_t) \\
&\; \qquad + \sum_{l=1}^{m'} \hat{Q}_l(\vec{\sigma}, \pt^l \Delta \vec{\alpha}) R(\vec{\Theta}) P_{2m'+1 - 2l}(\ps^{2m'-2l}\vec{\kappa}_t)\Big), 
\end{aligned}\\
&\; I_{2m'}^\J I_{2m'+1}^\J = Q(\vec{\sigma}) R(\vec{\Theta}) P_{4m'+3}(\ps^{2m'} \vec{\kappa}_t) + \sum_{l=1}^{2m'} \hat{Q}_l(\vec{\sigma}, \pt^{\min\{m',l\}} \Delta \vec{\alpha}) R(\vec{\Theta}) P_{4m'+3 - 2l}(\ps^{2m'} \kappa_t^\J), 
\end{align*}
we obtain 
\begin{equation}\label{decay-h24}
\begin{aligned}
J_{2m'}^\J =&\; \frac{t^{2m'}}{(2m')!} \Big\{2((\sigma^\J)^{m'+2} \ps^{2m'} \kappa_t^\J + I_{2m'}^\J)((\sigma^\J)^{m'+1} \ps^{2m'+1} \kappa_t^\J + I_{2m'+1}^\J) - 2 I_{2m'}^\J I_{2m'+1}^\J \\
&\; - (\sigma^\J)^{2m'+2} \lambda_t^\J (\ps^{2m'} \kappa_t^\J)^2 - 2(\sigma^\J)^{m'+2} \ps^{2m'} \kappa_t^\J I_{2m'+1}^\J -2(\sigma^\J)^{m'+1} \ps^{2m'+1} \kappa_t^\J I_{2m'}\Big\}  \\
=&\; \frac{t^{2m'}}{(2m')!} 2((\sigma^\J)^{m'+2} \ps^{2m'} \kappa_t^\J + I_{2m'}^\J)((\sigma^\J)^{m'+1} \ps^{2m'+1} \kappa_t^\J + I_{2m'+1}^\J) \\
&\; - \frac{t^{2m'}}{(2m'-1)!} (\sigma^\J)^{2m'+2} \lambda_t^\J \ps^{2m'+1} \kappa_t^\J \ps^{2m'-1} \kappa_t^\J  \\
&\; + t^{2m'} \ps^{2m'+1} \kappa_t^\J \Big( Q(\vec{\sigma}) R(\vec{\Theta}) P_{2m'+1} (\ps^{2m'-2} \vec{\kappa}_t) \\
&\; \qquad + \sum_{l=1}^{m'} \hat{Q}_l(\vec{\sigma}, \pt^l \Delta \vec{\alpha}) R(\vec{\Theta}) P_{2m'+1 - 2l}(\ps^{2m'-2l}\vec{\kappa}_t)\Big) \\
&\; +t^{2m'} \left(Q(\vec{\sigma}) R(\vec{\Theta}) P_{4m'+3}(\ps^{2m'} \vec{\kappa}_t) + \sum_{l=1}^{2m'} \hat{Q}_{2l}(\vec{\sigma}, \pt^{\min\{m',l\}} \Delta \vec{\alpha}) R(\vec{\Theta}) P_{4m'+3 - 2l}(\ps^{2m'} \kappa_t^\J)\right). 
\end{aligned}
\end{equation}
For the first term, we can apply \eqref{simple-sum} with 
\[ A^\J = (\sigma^\J)^{m'+2} \ps^{2m'} \kappa_t^\J + I_{2m'}^\J, \quad B^\J=(\sigma^\J)^{m'+1} \ps^{2m'+1} \kappa_t^\J + I_{2m'+1}^\J, \] 
\eqref{bc-sum-ks2}, \eqref{bc-diffe-ks2} and \eqref{bdd-deri-f} to obtain 
\begin{align*}
&\; \sum_{j=1}^3 \frac{t^{2m'}}{(2m')!} 2((\sigma^\J)^{m'+2} \ps^{2m'} \kappa_t^\J + I_{2m'}^\J)((\sigma^\J)^{m'+1} \ps^{2m'+1} \kappa_t^\J + I_{2m'+1}^\J) \\
=&\; \sum_{j=1}^3 \frac{2t^{2m'}}{3 \cdot (2m')!} \left((\sigma^\J)^{m'+2} \ps^{2m'} \kappa_t^\J + I_{2m'}^\J \right) (\pt^{m'} f^\J - \pt^{m'} f^{(j-1)}) \le t^{2m'} P_{\le 4m'+3} (\|\ps^{2m'} \vec{\kappa}_t\|).
\end{align*}
Substituting it into the sum of \eqref{decay-h24} and applying \eqref{bdd-deri-I}, we have 
\begin{equation}\label{decay-h25}
\begin{aligned}
\sum_{j=1}^3 J_{2m'}^\J \le &\; \sum_{j=1}^3 \Big\{- \frac{t^{2m'}}{(2m'-1)!} (\sigma^\J)^{2m'+2} \lambda_t^\J \ps^{2m'+1} \kappa_t^\J \ps^{2m'-1} \kappa_t^\J  \\
&\; \quad + t^{2m'} \ps^{2m'+1} \kappa_t^\J \Big( Q(\vec{\sigma}) R(\vec{\Theta}) P_{2m'+1} (\ps^{2m'-2} \vec{\kappa}_t) \\
&\; \quad \quad + \sum_{l=1}^{m'} \hat{Q}_l(\vec{\sigma}, \pt^l \Delta \vec{\alpha}) R(\vec{\Theta}) P_{2m'+1 - 2l}(\ps^{2m'-2l}\vec{\kappa}_t)\Big) \Big\} + t^{2m'} P_{\le 4m'+3}(\|\ps^{2m'} \vec{\kappa}_t\|). 
\end{aligned}
\end{equation}
We here obtained critical terms with the coefficient $\ps^{2m'+1} \kappa_t^\J$, and thus we will change the terms to a time derivative of lower order terms applying \eqref{simple-kappa} and also \eqref{simple-l-V}, \eqref{simple-theta} as 
\begin{align*}
&\; - \frac{t^{2m'}}{(2m'-1)!} (\sigma^\J)^{2m'+2} \lambda_t^\J \ps^{2m'+1} \kappa_t^\J \ps^{2m'-1} \kappa_t^\J + t^{2m'} \ps^{2m'+1} \kappa_t^\J \Big( Q(\vec{\sigma}) R(\vec{\Theta}) P_{2m'+1} (\ps^{2m'-2} \vec{\kappa}_t)\Big) \\
=&\; - \frac{t^{2m'}}{(2m'-1)!} (\sigma^\J)^{2m'+1} \lambda_t^\J \pt\ps^{2m'-1} \kappa_t^\J \ps^{2m'-1} \kappa_t^\J \\
&\; + t^{2m'} \pt \ps^{2m'-1}\kappa_t^\J \Big( Q(\vec{\sigma}) R(\vec{\Theta}) P_{2m'+1} (\ps^{2m'-2} \vec{\kappa}_t)\Big) + t^{2m'} Q(\vec{\sigma}) R(\vec{\Theta}) P_{4m'+3}(\ps^{2m'} \vec{\kappa}_t) \\
=&\; \pt \left( -\frac{t^{2m'}}{2 \cdot (2m'-1)!} (\sigma^\J)^{2m'+1} \lambda_t^\J (\ps^{2m'-1} \kappa_t^\J)^2 + t^{2m'} \ps^{2m'-1}\kappa_t^\J \Big( Q(\vec{\sigma}) R(\vec{\Theta}) P_{2m'+1} (\ps^{2m'-2} \vec{\kappa}_t)\Big)\right) \\
&\; + t^{2m'-1} Q(\vec{\sigma}) R(\vec{\Theta}) P_{4m'+1}(\ps^{2m'-1} \vec{\kappa}_t) + t^{2m'} \hat{Q}_1(\vec{\sigma}, \pt \Delta \vec{\alpha}) R(\vec{\Theta}) P_{4m'+1}(\ps^{2m'-1} \vec{\kappa}_t) \\
&\; + t^{2m'} Q(\vec{\sigma}) R(\vec{\Theta}) P_{4m'+3}(\ps^{2m'} \vec{\kappa}_t) \\
=&\; \pt \left(t^{2m'} Q(\vec{\sigma}) R(\vec{\Theta}) P_{4m'+1}(\ps^{2m'-1} \vec{\kappa}_t) \right) + t^{2m'-1} Q(\vec{\sigma}) R(\vec{\Theta}) P_{4m'+1}(\ps^{2m'-1} \vec{\kappa}_t) \\
&\; + t^{2m'} \hat{Q}_1(\vec{\sigma}, \pt \Delta \vec{\alpha}) R(\vec{\Theta}) P_{4m'+1}(\ps^{2m'-1} \vec{\kappa}_t) + t^{2m'} Q(\vec{\sigma}) R(\vec{\Theta}) P_{4m'+3}(\ps^{2m'} \vec{\kappa}_t),   
\end{align*}
which implies, estimating the non-differential terms as in \eqref{decay-h25} and substituting it to \eqref{decay-h25}, 
\begin{align*}
\sum_{j=1}^3 J_{2m'}^\J \le &\; \pt \left(t^{2m'} Q(\vec{\sigma}) R(\vec{\Theta}) P_{4m'+1}(\ps^{2m'-1} \vec{\kappa}_t) \right) + t^{2m'-1} P_{4m'+1} (\|\ps^{2m'-1} \vec{\kappa}_t\|) + t^{2m'} P_{\le 4m'+3}(\|\ps^{2m'} \vec{\kappa}_t\|). 
\end{align*}
Since Young's inequality implies 
\[ P_{\le l} (\|\ps^i \vec{\kappa}_t\|) \le \sum_{j=1}^3 P_{\le l} (\|\ps^i \kappa_t^\J\|), \] 
the estimate \eqref{ine-inter2} can be applied to the non-differential terms as 
\begin{equation}\label{decay-h26}
\begin{aligned}
\sum_{j=1}^3 J_{2m'}^\J \le &\; \pt \left(t^{2m'} Q(\vec{\sigma}) R(\vec{\Theta}) P_{4m'+1}(\ps^{2m'-1} \vec{\kappa}_t) \right) \\
&\; + \sum_{j=1}^3 \frac{1}{8} \int_{\Gamma_t^\J} \frac{t^{2m'-1}}{(2m'-1)!} (\sigma^\J)^{2m'+2} (\ps^{2m'} \kappa_t^\J)^2 + \frac{t^{2m'}}{(2m')!} (\sigma^\J)^{2m'+3} (\ps^{2m'+1} \kappa_t^\J)^2 \; ds. \\
&\; + \tilde{M} \sum_{j=1}^3 (t^{2m'-1} \|\kappa_t^\J\|_{L^2}^{\tilde{q}_{2m'-1}'} + t^{2m'} \|\kappa_t^\J\|_{L^2}^{\tilde{q}_{2m'}'})
\end{aligned}
\end{equation}
for some positive constants $\tilde{M}, \tilde{q}_{2m'-1}'$ and $\tilde{q}_{2m'}'$. 
The strategy for critical terms in the estimate of $J_{2m'}^\J$ is not necessary to $J_{2m'-1}^\J$ since the estimate \eqref{ine-inter3} can be applied to critical terms in the following estimate of $J_{2m'-1}^\J$. 
Indeed, applying \eqref{simple-l-V}, \eqref{def-I2n} and \eqref{def-I2n1} as in \eqref{decay-h24}, we have 
\begin{align*}
J_{2m'-1}^\J =&\; 2\frac{t^{2m'-1}}{(2m'-1)!} ((\sigma^\J)^{m'+2} \ps^{2m'} \kappa_t^\J + I_{2m'}^\J)((\sigma^\J)^{m'} \ps^{2m'-1} \kappa_t^\J + I_{2m'-1}^\J) \\
&\; + t^{2m'-1} \ps^{2m'} \kappa_t^\J \left(Q(\vec{\sigma}) R(\vec{\Theta}) P_{2m'}(\ps^{2m'-2} \vec{\kappa}_t) + \sum_{l=1}^{m'-1} \hat{Q}_l(\vec{\sigma}, \pt^l \Delta \vec{\alpha}) R(\vec{\Theta}) P_{2m'-2l}(\ps^{2m'-1-2l} \vec{\kappa}_t)\right) \\
&\; + t^{2m'-1} \left(Q(\vec{\sigma}) R(\vec{\Theta}) P_{4m'+1}(\ps^{2m'-1} \vec{\kappa}_t) + \sum_{l=1}^{2m'-1} \hat{Q}_l (\vec{\sigma}, \pt^{\min \{m', l\}} \Delta \vec{\alpha}) R(\vec{\Theta}) P_{4m'+1 - 2l}(\ps^{2m'-1} \vec{\kappa}_t)\right), 
\end{align*}
which implies, applying the formulas \eqref{simple-sum}, \eqref{bc-sum-ks2} and \eqref{bc-diffe-ks2} to the first term and the estimates \eqref{bdd-deri-I} and \eqref{bdd-deri-f} to $\hat{P}_l$ and $\pt^{m'-2} f^\J$, respectively, 
\begin{align*}
\sum_{j=1}^3 J_{2m'-1}^\J =&\; \sum_{j=1}^3 \frac{2t^{2m'-1}}{3 \cdot (2m'-1)!} \left((\sigma^\J)^{m'+2} \ps^{2m'} \kappa_t^\J + I_{2m'}^\J\right) \kappa_t^\J (\pt^{m'-1} f^\J - \pt^{m'-1} f^{(j-1)}) \\
&\; + t^{2m'-1} \ps^{2m'} \kappa_t^\J \left(Q(\vec{\sigma}) R(\vec{\Theta}) P_{2m'}(\ps^{2m'-2} \vec{\kappa}_t) + \sum_{l=1}^{m'-1} \hat{Q}_l(\vec{\sigma}, \pt^l \Delta \vec{\alpha}) R(\vec{\Theta}) P_{2m'-2l}(\ps^{2m'-1-2l} \vec{\kappa}_t)\right) \\
&\; + t^{2m'-1} \left(Q(\vec{\sigma}) R(\vec{\Theta}) P_{4m'+1}(\ps^{2m'-1} \vec{\kappa}_t) + \sum_{l=1}^{2m'-1} \hat{Q}_l (\vec{\sigma}, \pt^{\min \{m', l\}} \Delta \vec{\alpha}) R(\vec{\Theta}) P_{4m'+1 - 2l}(\ps^{2m'-1} \vec{\kappa}_t)\right) \\
\le &\; \sum_{j=1}^3 t^{2m'-1} |\ps^{2m'} \kappa_t^\J| (P_{\le 2m'}(\|\ps^{2m'-2} \vec{\kappa}_t\|) + \hat{M}') + t^{2m'-1} P_{\le 4m'+1}(\|\ps^{2m'-1} \vec{\kappa}_t\|) 
\end{align*}
for some positive constant $\hat{M}'$. 
Applying \eqref{ine-inter3} to the critical terms with the coefficient $\ps^{2m'} \kappa_t^\J$ and \eqref{ine-inter2} to the remained terms, we have 
\begin{equation}\label{decay-h27}
\begin{aligned}
\sum_{j=1}^3 J_{2m'-1}^\J \le&\; \sum_{j=1}^3 \frac{1}{8} \int_{\Gamma_t^\J} \frac{t^{2m'-1}}{(2m'-1)!} (\sigma^\J)^{2m'+2} (\ps^{2m'} \kappa_t^\J)^2 + \frac{t^{2m'}}{(2m')!} (\sigma^\J)^{2m'+3}(\ps^{2m'+1} \kappa_t^\J)^2 \; ds \\
&\; + \hat{M} \sum_{j=1}^3 (t^{2m'-1} \|\kappa_t^\J\|_{L^2}^{\hat{q}_{2m'-1}} + t^{r_{2m'-1}} \|\kappa_t^\J\|_{L^2}^{\hat{q}_{2m'-1}'})
\end{aligned}
\end{equation}
for some constants $\hat{M}, \hat{q}_{2m'-1}, \hat{q}_{2m'-1}' > 0$ and $r_{2m'-1} > -1$. 

We now substitute \eqref{decay-h23}, \eqref{decay-h26} and \eqref{decay-h27} into the summation of \eqref{decay-h22} with respect to $j \in \{1,2,3\}$ and apply the exponential $L^2$-decay of the curvatures as in Proposition \ref{prop:decay-l2} to obtain
\begin{align*}
&\; \dfrac{d}{dt} \sum_{j=1}^3  \int_{\Gamma_t^\J} \sum_{m=0}^{2n} \frac{t^m}{m!} (\sigma^\J)^{m+2} (\ps^m \kappa_t^\J)^2\; ds \\
\le &\; \sum_{m'=1}^n \pt \left(t^{2m'} Q(\vec{\sigma})R(\vec{\Theta}) P_{4m'+1}(\ps^{2m'-1} \vec{\kappa}_t)\lfloor_{\text{at} \; \vec{a}}\right) + \frac{\Cr{c-decay-h21}}{2} \left(\sum_{m=0}^{2n} t^m +t^{\Cr{c-decay-h2}}\right) e^{-\Cr{c-decay-h22}t} 
\end{align*}
for some constants $\Cr{c-decay-h21}, \Cr{c-decay-h22}>0$ and $\Cr{c-decay-h2} > -1$. 
Integrating it and applying \eqref{ine-inter2} again to the polynomials $P_{4m'+1}(\ps^{2m'-1} \vec{\kappa}_t)\lfloor_{\text{at} \; \vec{a}}$, we have by means of the $L^2$-exponential decay of the curvatures as in Proposition \ref{prop:decay-l2} 
\begin{align*}
& \sum_{j=1}^3  \int_{\Gamma_t^\J} \sum_{m=0}^{2n} \frac{t^m}{m!} (\sigma^\J)^{m+2} (\ps^m \kappa_t^\J)^2 \; ds - \sum_{j=1}^3 \int_{\Gamma_0^\J} (\sigma_0^\J)^2 (\kappa_0^\J)^2 \; ds\\
&\le \sum_{m'=1}^n t^{2m'} Q(\vec{\sigma})R(\vec{\Theta}) P_{4m'+1}(\ps^{2m'-1} \vec{\kappa}_t)\lfloor_{\text{at} \; \vec{a}} + \frac{\Cr{c-decay-h21}}{2} \int_0^t \left(\sum_{m=0}^{2n} \tilde{t}^m +\tilde{t}^{\Cr{c-decay-h2}}\right) e^{-\Cr{c-decay-h22}\tilde{t}} \; d\tilde{t} \\
&\le \frac{1}{2} \sum_{j=1}^3 \int_{\Gamma_t^\J} \sum_{m'=1}^n \frac{t^{2m'}}{(2m')!} (\sigma^\J)^{2m'+2} (\ps^{2m'} \kappa_t^\J)^2 \; ds + \frac{\Cr{c-decay-h21}}{2} \int_0^t \left(\sum_{m=0}^{2n} \tilde{t}^m +\tilde{t}^{\Cr{c-decay-h2}}\right) e^{-\Cr{c-decay-h22}\tilde{t}} \; d\tilde{t} + \frac{\Cr{c-decay-h23}}{2} \sum_{m'=1}^n t^{2m'} e^{-\Cr{c-decay-h24}t}
\end{align*}
for some $\Cr{c-decay-h23}, \Cr{c-decay-h24}>0$, which implies \eqref{ex-decay-h2}.
\end{proof}

\begin{remark}\label{rmk:hi-deris}
(i) The Reyleigh quotient in Section \ref{subsec:Reyleigh} cannot be applied to the negative terms in \eqref{decay-h22} because of the difference of the boundary conditions for the curvature and those derivatives. 
Therefore, we cannot obtain exponential decay of the higher order derivatives of the curvatures directly extending the method for the exponential $L^2$-decay of the curvatures.

(ii) Since the exponential functions in \eqref{ex-decay-h2} were derived from the exponential $L^2$-decay of the curvatures, according to the sufficient assumption to obtain the boundedness of $|\pa^i \sigma^\J|$ with $i \in \mathbb{N}$, $|\pt \Delta \alpha^\J|$ and $|f^\J|$ as in Remark \ref{rmk:bddness1} and Remark \ref{rmk:bddness2} (see also Remark \ref{rmk:bdd-lf} for the boundedness of $|\pt^i \Delta \alpha|$ and $|\pt^i f^\J|$ with $i \in \mathbb{N}$), the exponential functions can be replaced by some constants only assuming (A1); $E(0) \le \sigma(0) \Cr{min-length}$ or uniformly positivity of $L^\J$; \eqref{bdd-sin2}; and uniformly boundedness of $L^2$-norm of the curvatures. 
Therefore, the assumptions (A2)--(A3) and smallness of the misorientations and the curvature as in \eqref{small-initial2} is not necessary to obtain the locally boundedness of the derivatives of the curvatures under the above assumptions. 
\end{remark}

While we already derived the higher order estimates of the curvatures, $C^\infty$ or H\"older estimate of the angle function $\Theta^\J$ defined by \eqref{def-theta} is not obvious and necessary to discuss the maximum existence time via the existence theorem as in Proposition \ref{prop:short-time-system}. 
Therefore, we give the following corollary. 

\begin{cor}\label{cor:hi-estimate-theta}
Let $\Theta^\J$ be the angle function, defined by \eqref{def-theta} and satisfying the restriction on the parametrization \eqref{rest-para}, for a smooth geometric flow governed by \eqref{eq-curve}--\eqref{bc-boundary}.
Then, under the assumptions in \ref{prop:decay-h2}, for any $\beta \in (0,1/2)$, $t_0 > 0$ and $k \in \mathbb{N} \cup \{0\}$, there exists $\Cl[c]{c:bdd-theta1}, \Cl[c]{c:bdd-theta2}, \Cl[c]{c:bdd-theta3}, \Cl[c]{c:bdd-theta4} > 0$ such that 
\begin{equation}
\sum_{j=1}^3 \|\Theta^\J(\cdot,t) - \Theta_0^\J\|_{C_x^{\beta}([0,1])} \le \Cr{c:bdd-theta1} t^{\frac{1}{4} - \frac{\beta}{2}} \quad \text{for} \; \; t \in [0,1], \quad \sum_{j=1}^3 \|\Theta^\J\|_{C_{x,t}^k([0,1] \times [t_0, \infty))} \le \Cr{c:bdd-theta2} \label{conti-holder}
\end{equation}
and 
\begin{equation}\label{conti-La}
\sum_{j=1}^3 |L^\J(t) - L^\J(0)| + |\alpha^\J(t) - \alpha^\J_0| \le \Cr{c:bdd-theta3}t^\frac{3}{4} \quad \text{for} \; \; t \in [0,1], \quad \sum_{j=1}^3 \|L^\J\|_{C_t^k([t_0, \infty))} +  \|\alpha^\J\|_{C_t^k([t_0, \infty))} \le \Cr{c:bdd-theta4}. 
\end{equation}
\end{cor}

\begin{proof}
Let $M$ and $c$ be positive constants independent of $\Theta^\J$ and will be re-chosen as necessary through this proof. 
Due to the restriction \eqref{rest-para}, we can see that 
\begin{equation}\label{dep-s} 
s = s^\J_t(x) = L^\J(t) x \quad \text{for} \; \; x \in [0,1], 
\end{equation}
where $s=s^\J_t$ is the arc-length parameter of $\Gamma^\J_t$. 
Since the right hand side of \eqref{ex-decay-h2} is uniformly bounded in $t$ and $\sigma^\J \ge \sigma(0)$ holds, we have 
\begin{equation}\label{theta-l2-decay} 
\|\ps \kappa^\J_t\|_{L^2} \le \frac{M}{t^{1/2}}, \quad  \|\ps^2 \kappa^\J_t\|_{L^2} \le \frac{M}{t} \quad \text{for} \; \; t > 0, \; \; j \in \{1,2,3\}. 
\end{equation}
Here, $\| \cdot \|_{L^2}$ means that the $L^2$-norm on $\Gamma_t ^\J$ with respect to the arc-length. 
Furthermore, by means of Proposition \ref{prop:decay-l2} and Proposition \ref{prop:interpolation}, we have 
\begin{equation}\label{theta-sup-decay}
\begin{aligned} 
&\|\ps \kappa^\J_t\|_{L^\infty} \le M(\|\ps^2 \kappa_t^\J\|_{L^2}^{\frac{3}{4}} + \|\kappa_t^\J\|_{L^2}^{\frac{3}{4}})\|\kappa_t^\J\|_{L^2}^{\frac{1}{4}} \le Mt^{-\frac{3}{4}} e^{-ct}, \\ 
&\|\kappa_t^\J\|_{L^\infty} \le M(\|\ps \kappa_t^\J\|_{L^2}^{\frac{1}{2}} + \|\kappa_t^\J\|_{L^2}^{\frac{1}{2}}) \|\kappa_t^\J\|_{L^2}^{\frac{1}{2}} \le M t^{-\frac{1}{4}} e^{-ct} 
\end{aligned}
\end{equation}
for any $t > 0$ and $j \in \{1,2,3\}$.

First we prove 
\begin{equation}\label{conti-sup}
\sum_{j=1}^3 \|\Theta^\J(\cdot, t) - \Theta_0^\J\|_{L^\infty} \le M (\min\{t,1\})^\frac{1}{4} \quad \text{for} \; \; t \ge 0. 
\end{equation}
Re-write the first identity in \eqref{eq-theta1} to obtain 
\begin{equation}\label{re-eq-theta}
\begin{aligned}
\pt \Theta^\J (x,t) =&\; \sigma^\J \ps \kappa^\J_t(s^\J_t(x),t) + x\kappa^\J_t (s^\J_t(x), t) \lambda^\J_t \lfloor_{\text{at} \; \vec{a}} \\
&\; + \sigma^\J \kappa_t^\J(s^\J_t(x),t) \left\{ \int_0^{s^\J_t(x)} (\kappa_t^\J(s,t))^2 \; ds - x \int_{\Gamma_t^\J} (\kappa_t^\J(s,t))^2 \; ds\right\}, 
\end{aligned}
\end{equation}
which implies, due to \eqref{bdd-l-V}, \eqref{theta-sup-decay} and the boundedness of $\|\kappa_t^\J\|_{L^2}$, 
\begin{align*} 
\|\Theta^\J(\cdot,t) - \Theta^\J_0\|_{L^\infty} \le&\; \int_0^t \|\pt \Theta^\J(\cdot,\tilde{t})\|_{L^\infty} \; d\tilde{t} \\
\le&\; M \int_0^t \|\ps \kappa_{\tilde{t}}^\J\|_{L^\infty} + \|\kappa_{\tilde{t}}^\J\|_{L^\infty} \cdot \sum_{k=1}^3 \|\kappa^\K_{\tilde{t}}\|_{L^\infty} + \|\kappa_{\tilde{t}}^\J\|_{L^\infty} \|\kappa_{\tilde{t}}^\J\|_{L^2}^2 \; d\tilde{t} \\
\le&\; M \int_{0}^t (\tilde{t}^{-\frac{3}{4}} + \tilde{t}^{-\frac{1}{2}} + \tilde{t}^{-\frac{1}{4}}) e^{-c\tilde{t}} \; d\tilde{t} \le M (\min\{t,1\})^\frac{1}{4}. 
\end{align*}

By means of \eqref{re-eq-theta} and $\px \Theta^\J = \kappa^\J/L^\J$, we can see that $\pt^n \px^m \Theta^\J$ can be represented by a polynomial in the derivatives of the curvatures with respect to the arc-length and the time derivatives of the misorientations and lengths for any $n, m \in \mathbb{N} \cup \{0\}$ except $n=m=0$ since we already derived the formulation of $\lambda^\J_t \lfloor_{\text{at} \; \vec{a}}$ in \eqref{eq-lambda-V}. 
Notice that the second estimate in \eqref{conti-La} follows from \eqref{eq-alpha}, \eqref{bdd-deri-L} and \eqref{bdd-deri-alpha} since the derivatives of the curvatures are bounded away from $t=0$ as in Proposition \ref{prop:decay-h2}. 
Therefore, the second estimate in \eqref{conti-holder} also can be obtained easily. 

We thus next prove the first estimate in \eqref{conti-holder}. 
If $t \le |x-y|^2$ and $t \le 1$, we have by \eqref{conti-sup}
\[ |(\Theta^\J(x,t) - \Theta^\J_0(x)) - (\Theta^\J(y,t) - \Theta^\J_0(y))| \le M t^\frac{1}{4} \le M t^{\frac{1}{4}-\frac{\beta}{2}} |x-y|^\beta. \]
It thus sufficient to consider the case $|x-y|^2 \le t$. 
Since the length $L^\J$ is uniformly positive and bounded, by means of the Sobolev inequality and \eqref{dep-s}, we have 
\begin{align*} 
|\ps^n \kappa^\J_t(s(x), t) - \ps^n \kappa^\J_t(s(y),t)| \le&\; M (\|\ps^n \kappa_t^\J\|_{L^2} + \|\ps^{n+1} \kappa^\J_t\|_{L^2}) |s(x) - s(y)|^{\frac{1}{2}} \\
\le&\; M (\|\ps^n \kappa_t^\J\|_{L^2} + \|\ps^{n+1} \kappa^\J_t\|_{L^2}) |x - y|^{\frac{1}{2}} 
\end{align*}
for any $n \in \mathbb{N} \cup \{0\}$. 
It can be applied to obtain, due to \eqref{re-eq-theta}, 
\begin{align*}
&\; |\pt\Theta^\J(x,t) - \pt\Theta^\J(y,t)| \\
\le&\; \sigma^\J|\ps \kappa_t^\J(s(x),t) - \ps \kappa_t^\J(s(y),t)| \\
&\; + |x-y| \cdot |\kappa^\J_t (s(x),t) \lambda_t^\J \lfloor_{\text{at} \; \vec{a}}| + |y \lambda_t^\J \lfloor_{\text{at} \; \vec{a}}| \cdot |\kappa_t^\J(s(x),t) - \kappa_t^\J(s(y),t)| \\
&\; + \sigma^\J|\kappa_t^\J(s(x), t) - \kappa_t^\J(s(y),t)| \cdot \left| \int_0^{s^\J_t(x)} (\kappa_t^\J(s,t))^2 \; ds - x \int_{\Gamma_t^\J} (\kappa_t^\J(s,t))^2 \; ds \right| \\
&\; + \sigma^\J |\kappa_t^\J(s(y),t)| \cdot \left| \int_{s^\J_t(y)}^{s^\J_t(x)} (\kappa_t^\J(s,t))^2 \; ds - (x-y) \int_{\Gamma_t^\J} (\kappa_t^\J(s,t))^2 \; ds\right| \\
\le &\; M \Big\{(\|\ps \kappa_t^\J\|_{L^2} + \|\ps^2 \kappa_t^\J\|_{L^2})|x-y|^{\frac{1}{2}} \\
&\; \quad + \left(\sum_{k=1}^3 \|\kappa_t^\K\|_{L^\infty}\right) \{\|\kappa_t^\J\|_{L^\infty} |x-y| + (\|\kappa_t^\J\|_{L^2} + \|\ps \kappa_t^\J\|_{L^2})|x-y|^{\frac{1}{2}}\} \\
&\; \quad + (\|\kappa_t^\J\|_{L^2} + \|\ps \kappa_t^\J\|_{L^2}) \|\kappa_t^\J\|_{L^2}^2|x-y|^{\frac{1}{2}} + \|\kappa_t^\J\|_{L^\infty}(\|\kappa_t^\J\|_{L^\infty}^2 + \|\kappa_t^\J\|_{L^2}^2)|x-y|, 
\end{align*}
which implies, applying \eqref{theta-l2-decay}, \eqref{theta-sup-decay} and the boundedness of $\|\kappa_t^\J\|_{L^2}$, 
\[ |\pt\Theta^\J(x,t) - \pt\Theta^\J(y,t)| \le M\{ (1+ t^{-\frac{1}{2}} + t^{-1})|x-y|^{\frac{1}{2}} + (t^{-\frac{1}{4}} + t^{-\frac{1}{2}} + t^{-\frac{3}{4}}) |x-y|\} \le M t^{-\frac{3+2\beta}{4}}|x-y|^\beta. \]
Here $t \le 1$ and $t \le |x-y|^2$ have been used. 
Since $\frac{3+2\beta}{4} < 1$, we thus obtain 
\[ |(\Theta^\J(x,t) - \Theta^\J_0(x)) - (\Theta^\J(y,t) - \Theta^\J_0(y))| \le \int_0^t |\pt \Theta^\J(x,\tilde{t}) - \pt \Theta^\J(y, \tilde{t})| \; d\tilde{t} \le M t^{\frac{1}{4} - \frac{\beta}{2}} |x-y|^\beta. \]
Combining \eqref{conti-sup}, we have the first estimate in \eqref{conti-holder}. 

For the continuity of $L^\J$ at $t=0$ as in \eqref{conti-La} can be obtained applying the boundedness of $\|\kappa^\J\|_{L^2}$ and \eqref{theta-sup-decay} to \eqref{hi-deri-L1} and also the continuity of $\alpha^\J$ at $t=0$ follows from the exponential decay of $\pt \alpha^\J$ as in \eqref{delta-at-ex-decrease}. 
\end{proof}

We finally proceed to the proof of Theorem \ref{thm:global-asymptotic}.

\begin{proof}[Proof of Theorem \ref{thm:global-asymptotic}]
Let 
\[ m:= \frac{\Cr{min-length}}{2}, \quad \varepsilon:= \frac{\min\{\Cr{e-small-initial1}, \Cr{e-small-initial2}\}}{2}, \]
where $\Cr{min-length}, \Cr{e-small-initial1}$ and $\Cr{e-small-initial2}$ are constants obtained in Lemma \ref{lem:min-length}, Proposition \ref{prop:decay-l2} and Proposition \ref{prop:decay-h2}, respectively. 
We first define the angle function $\Theta_0^\J$ satisfying the restriction on the parametrization \eqref{rest-para} by \eqref{def-theta}. 
Since we assume that $\Gamma^\J_0$ is of class $H^2$ and satisfies \eqref{main:small-ini}, by means of the Sobolev inequality and $\px \Theta^\J_0 = L^\J(0) \kappa_t^\J$, we then see that $\Theta_0^\J \in C^\frac{1}{2} ([0,1]) \cap H^1(0,1)$ and the family of $\vec{\Theta}_0$ and $\vec{\alpha}_0$ satisfy the compatibility condition of order $0$ for \eqref{system-theta}, namely, 
\[ \sum_{j=1}^3 \sigma(\Delta^\J \alpha_0) \cos \Theta_0^\J(1) = \sum_{j=1}^3 \sigma(\Delta^\J \alpha_0) \sin \Theta_0^\J(1) = 0. \]
Notice also that, since \eqref{bc-boundary} is satisfied at $t=0$, the parametrization $\xi_0^\J$ of $\Gamma_0^\J$ is of form 
\begin{equation}\label{main-formula-curve} 
\xi_0^\J(x) = \left(\int_0^x L^\J_0 \cos \Theta^\J_0(\tilde{x}) \; d\tilde{x}, \int_0^x L^\J_0 \sin \Theta^\J_0(\tilde{x}) \; d\tilde{x} \right) + P^\J, 
\end{equation}
where $L^\J_0$ is the length of $\Gamma_0^\J$, and we also have
\begin{align*} 
&\; \left(\int_0^1 L^\J_0 \cos \Theta^\J_0(x) \; dx, \int_0^1 L^\J_0 \sin \Theta^\J_0(x) \; dx \right) + P^\J \\
=&\; \left(\int_0^1 L^\K_0 \cos \Theta^\K_0(x) \; dx, \int_0^1 L^\K_0 \sin \Theta^\K_0(x) \; dx \right) + P^\K 
\end{align*}
for any $j,k \in \{1,2,3\}$ since $\cup_{j=1}^3 \Gamma^\J_0$ is a triod. 
Therefore, we can approximate $\vec{\Theta}_0$ and $\vec{L}_0$ by $\vec{\Theta}_{0,\varepsilon'} \in (C^\infty([0,1]))^3$ and $\vec{L}_{0, \varepsilon'}$, respectively, so that
\begin{itemize}
\item[(i)] $\Theta^\J_{0,\varepsilon'} \to \Theta_0^\J$ in $C^\beta([0,1])$ and $H^1(0,1)$, and $L^\J_{0,\varepsilon'} \to L^\J_0$ as $\varepsilon' \to 0$ for any $j \in \{1,2,3\}$. 
\item[(ii)] The family of $\vec{\Theta}_{0,\varepsilon'}$ and $\vec{\alpha}_0$ satisfy the compatibility condition of any order $k \in \mathbb{N}$ for \eqref{system-theta}. 
\item[(iii)] $\vec{\Theta}_{0, \varepsilon'}$ and $\vec{L}_{0,\varepsilon'}$ satisfy, for any $j,k \in \{1,2,3\}$, 
\begin{align*} 
&\; \left(\int_0^1 L^\J_{0,\varepsilon'} \cos \Theta^\J_{0,\varepsilon'}(x) \; dx, \int_0^1 L^\J_{0,\varepsilon'} \sin \Theta^\J_{0,\varepsilon'} (x) \; dx \right) + P^\J \\
=&\; \left(\int_0^1 L^\K_{0,\varepsilon'} \cos \Theta^\K_{0,\varepsilon'}(x) \; dx, \int_0^1 L^\K_{0,\varepsilon'} \sin \Theta^\K_{0,\varepsilon'}(x) \; dx \right) + P^\K. 
\end{align*}
\end{itemize}
Note that, since we only need to adjust the approximation near the boundary to satisfy (ii), the conditions (i), (ii) and (iii) are compatible even if the family of $\vec{\Theta}_0$ and $\vec{\alpha}_0$ satisfy the compatibility condition of only order $0$.
Constructing an approximated initial curve $\Gamma^\J_{0,\varepsilon'}$ by the formula as in \eqref{main-formula-curve} from $\Theta^\J_{0, \varepsilon'}$ and $L_{0,\varepsilon'}$, we then see that the family of $\{\Gamma_{0,\varepsilon'}^\J\}_{j \in \{1,2,3\}}$ and $\vec{\alpha}_0$ satisfy the assumptions in Theorem \ref{thm:exists-flow} with any $k \ge 3$ and obtain a smooth geometric flow governed by \eqref{eq-curve}--\eqref{bc-boundary} starting from the approximated initial datum. 
Let $\{\Gamma^\J_{t,\varepsilon'}\}_{j \in \{1,2,3\}}$ and $\{\alpha^\J_{\varepsilon'}(t)\}_{j \in \{1,2,3\}}$ be the moving triod and the set of misorientations of the smooth geometric flow. 
Since the length of $\Gamma^\J_{0,\varepsilon'}$ is $L^\J_{0,\varepsilon'}$ and the curvature $\kappa^\J_{0,\varepsilon'}$ of $\Gamma^\J_{0,\varepsilon'}$ is $\ps \Theta^\J_{0,\varepsilon'} = \px \Theta^\J_{0,\varepsilon'}/L_{0,\varepsilon'}$, we may see that 
\begin{align*} 
\sum_{j=1}^3 \sigma(\Delta^\J \alpha_0) L^\J_{0, \varepsilon'} \le \sigma(0)\Cr{min-length}, \quad \sum_{j=1}^3 \left\{\left(\Delta^\J \alpha_0\right)^2 + \int_{\Gamma^\J_{0, \varepsilon'}} \left(\sigma(\Delta^\J \alpha_0)\right)^2 (\kappa_{0, \varepsilon'}^\J)^2 \; ds\right\} \le \min\{\Cr{e-small-initial1}, \Cr{e-small-initial2}\}
\end{align*}
is satisfied for $\varepsilon'>0$ small by means of the convergence as in (i).
Therefore, all estimates in Section \ref{sec:orientation}--Section \ref{sec:H2-estimate} can be applied to the smooth geometric flow starting from the approximated initial datum to extend the maximum existence time. 
We summarize the estimates as follows. 
\begin{itemize}
\item[(a)] By means of Lemma \ref{lem:dec-E} and Lemma \ref{lem:min-length}, the length $L^\J_{\varepsilon'}$ of $\Gamma_{\varepsilon'}^\J$ is uniformly positive and bounded from above. 
\item[(b)] By means of Lemma \ref{lem:tri-ine-t}, the angle function $\Theta^\J_{\varepsilon'}$ of $\Gamma^\J_{t,\varepsilon'}$ satisfy $\Theta^\JJ_{\varepsilon'} - \Theta^\J_{\varepsilon'} \in (0, \pi)$ at the junction point $\vec{a}_{\varepsilon'}(t)$ of the moving triod $\{\Gamma_{t, \varepsilon'}^\J\}_{j \in \{1,2,3\}}$ for any $j \in \{1,2,3\}$ and $t\ge 0$. 
\item[(c)] The a priori estimates as in Corollary \ref{cor:hi-estimate-theta} holds. 
\end{itemize} 
Therefore, we can apply Proposition \ref{prop:short-time-system} to conclude that the flow exists globally in time. 

Furthermore, applying the Arzel\'a-Ascoli theorem, there exist a solution $\vec{\Theta} \in (C^\infty([0,1] \times (0, \infty)))^3$, $\vec{L} \in (C^\infty((0,\infty)))^3$, $\vec{\alpha} \in (C^\infty((0,\infty)))^3$ to \eqref{system-theta} such that $\vec{\Theta}_{\varepsilon'}, \vec{L}_{\varepsilon'}, \vec{\alpha}_{\varepsilon'}$ sub-converges to the solution in $(C^\infty_{\rm loc}([0,1] \times (0, \infty)))^3 \times (C^\infty_{\rm loc}((0,\infty)))^3 \times (C^\infty_{\rm loc}((0,\infty)))^3$. 
Then, we can construct a smooth geometric flow governed by \eqref{eq-curve}--\eqref{bc-boundary} from the solution to \eqref{system-theta} by the formulation as \eqref{main-formula-curve}, and we can also see that each curve $\Gamma_t^\J$ of the smooth geometric flow converges to the original initial curve $\Gamma^\J_0$ as $t \to 0$ in the $C^{1+\beta}$ topology, for any $\beta \in (0,1/2)$, due to the convergence property (i) and the continuities as in \eqref{conti-holder} and \eqref{conti-La}. 
On the other hand, since the exponential $L^2$-decay of the derivatives of the curvatures can be obtained by a similar estimate for \eqref{theta-sup-decay} away from $t=0$, by means of the uniqueness of the stationary solution as in Proposition \ref{prop:unique-stationary} and the exponential decay of $\Delta^\J \alpha$, as $t \to \infty$, each curve $\Gamma^\J_t$ converges to a line segment of the unique Steiner triod in $C^\infty$ topology and $\alpha^\J(t)$ converges to $(\alpha_0^{(1)} + \alpha_0^{(2)} + \alpha_0^{(3)})/3$ due to the preservation of $\sum_{j=1}^3 \alpha^\J(t)$. 
\end{proof}

\begin{remark}
(i) According to Remark \ref{rmk:hi-deris}, the boundedness of 
\[ \|\Theta^\J\|_{C_{x,t}^k([0,1] \times [t_0, T])} + \|L^\J\|_{C^k_t ([t_0,T])} + \|\alpha^\J\|_{C^k_t ([t_0,T])}, \]
for each $k \in \mathbb{N}, 0 < t_0 < T$, can be obtained whenever, for $t \le T$ and $j \in \{1,2,3\}$,
\begin{itemize} 
\item[(a)] $\|\kappa_t^\J\|_{L^2}$ is bounded, 
\item[(b)] $L^\J(t)$ is positive,
\item[(c)] $\Theta^\JJ-\Theta^\J \in (0,\pi)$ at the junction point $\vec{a}$ or \eqref{pre-complementarity} holds. 
\end{itemize}
Therefore, a smooth geometric flow can be constructed from a triod of class $H^2$ and if the maximum existence time $T'$ is finite, then either (a), (b) or (c) is lost at the time $T'$.

(ii) In the study \cite{MR2075985} for the classical curvature flow, the regularity estimates of the parametrization $\xi^\J$ are derived from the energy type estimates for not only the curvature but also a restricted tangent velocity $\lambda_t^\J = \langle \frac{\px^2 \xi^\J}{|\px \xi^\J|^2}, \tau^\J_t\rangle$. 
In our method, we does not require the energy type estimate of the tangent velocity and, additionally, some discussion to conclude that $|\px \xi^\J| \neq 0$ if $L^\J>0$, which is not obvious for the re-parametrization in \cite{MR2075985}, is not necessary in our method. 
\end{remark}

\input{references.tex}


\end{document}

%% file: introduction_final2.tex

There are many kinds of research about the curvature flow of networks
related to planar grain boundary motion. According to the theory by
Mullins and Herring \cite{doi:10.1007/978-3-642-59938-5_2,
doi:10.1063/1.1722511, doi:10.1063/1.1722742}, the curvature flow of
networks is well-known as a typical model for the evolution of grain
boundaries in polycrystalline materials as evolution law of local
interaction. In addition, the classical curvature flow of networks can
be derived from the principle of maximum dissipation for the total grain
boundaries, under the assumption that the total grain boundary energy
depends only on the surface tension of the grain boundaries and the
surface tension is constant. When the surface tension depends on the
grain boundary's normal vector in order to take into account the grain
boundary energy anisotropy, the anisotropic curvature flow of networks
can be derived similarly. These flows are well-studied as significant
geometric variational problems.

In this paper, we study the case that the grain boundary energy is
affected by the shape of the grain boundary and other grain boundary
structures. Namely, the surface tension has other state variables in our
case. A grain boundary is made from two or more single grains, which
have different lattice orientations.  A mismatch of the lattice
orientations neighboring grains is called a lattice
misorientation. Because the lattice orientation is discontinuous at the
grain boundary in general, one should include the lattice misorientation
in state variables of the grain boundary energy.  Kinderlehrer and
Liu~\cite{MR1833000} derived the governed equations, by the principle of
maximum dissipation, of the evolution of grain boundaries with the
lattice misorientations as a parameter and studied mathematical analysis
of the system. Epshteyn, Liu, and the second author \cite{MR4263432}
developed their work to assume that the lattice misorientations depend
on time. Here we briefly explain the case of one triple junction (see
figure \ref{fig:1.1}).

We consider the following grain boundary energy of the system
\begin{equation}
 \label{eq:1.1}
  E(t)
  =
  \sum_{j=1}^3
  \sigma(\Delta^{(j)}\alpha(t))L^{(j)}(t),
\end{equation}
where $\sigma=\sigma(\alpha):\mathbb{R}\rightarrow\mathbb{R}$ is a given
surface tension of a grain boundary,
$\alpha^{(j)}=\alpha^{(j)}(t):[0,\infty)\rightarrow\mathbb{R}$ is a
time-dependent orientation of the grains,
$\Delta^{(j)}\alpha:=\alpha^{(j-1)}-\alpha^{(j)}$ is a lattice
misorientation of the grain boundary $\Gamma^{(j)}_t$, and $L^{(j)}(t)$
is the length of $\Gamma_t^{(j)}$ for $j=1,2,3$. For simplicity of
notation, we write $\alpha^{(0)}=\alpha^{(3)}$.  Impose that three ends
of the grain boundaries $\Gamma^{(j)}$ are the same position
$\vec{a}(t)$, called a triple junction, and the other end of
$\Gamma^{(j)}$ is a given fixed point $P^{(j)}$.  In this work, we
assume that a grain boundary energy density $\sigma(\Delta^{(j)}\alpha)$
is independent of the normal vector of the grain boundary
$\Gamma_t^{(j)}$, namely, our system is isotropic with respect to
$\Gamma_t^{(j)}$. As a result of applying the principle of maximum
dissipation for the energy \eqref{eq:1.1}, the following model
was derived in \cite{MR4263432}:

 \begin{figure}
  \centering
  \includegraphics[height=7cm]{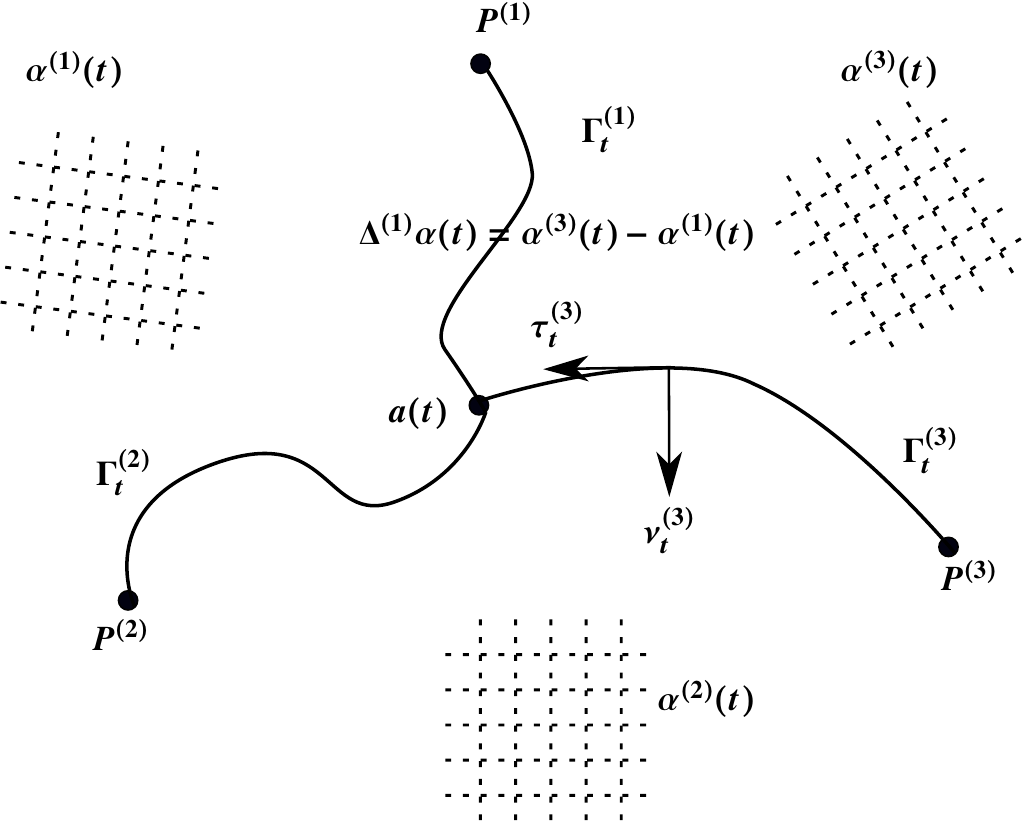}
  \caption{Grain boundaries/curves $\Gamma_t^{(j)}$ that meet at a
  triple junction $\vec{a}(t)$.  Lattice orientations are angles
  (scalars) $\alpha^{(j)}(t)$. Misorientation $\Delta^{(j)}\alpha(t)$ of
  $\Gamma^{(j)}_t$ is the difference between two lattice orientations of
  grains that share grain boundary $\Gamma^{(j)}_t$. The unit tangent
  vector of $\Gamma_t^{(j)}$ with the direction starting from
  $P^{(j)}_t$ and ending to $\vec{a}(t)$ is denoted by
  $\tau_t^{(j)}$. The unit normal vector, which is the 90 degree
  counterclockwise rotation of $\tau_t^{(j)}$, is denoted by
  $\nu_t^{(j)}$.}
  \label{fig:1.1}
 \end{figure}

\begin{equation}
 \label{eq:1.2}
  \left\{
  \begin{aligned}
   V_t^{(j)}
   &=
   \mu
   \sigma(\Delta^{(j)}\alpha)
   \kappa_t^{(j)},\quad\text{on}\ \Gamma_t^{(j)},\ t>0,\ j=1,2,3, \\
   \frac{d\alpha^{(j)}}{dt}
   &=
   -\gamma
   \Bigl(
   \partial_\alpha\sigma(\Delta^{(j+1)}\alpha)
   L^{(j+1)}(t)
   -
   \partial_\alpha\sigma(\Delta^{(j)}\alpha)
   L^{(j)}(t)
   \Bigr)
   ,\quad t>0, \ 
   j=1,2,3,
   \\
   \frac{d\vec{a}}{dt}(t)
   &=
   -\eta
   \sum_{k=1}^3
   \sigma(\Delta^{(k)}\alpha)
   \tau_t^{(k)},
   \quad t>0,\ \text{at}\ \vec{a}, \\
   \vec{a}(t)
   &=
   \vec{\xi}^{(1)}(1,t)
   =
   \vec{\xi}^{(2)}(1,t)
   =
   \vec{\xi}^{(3)}(1,t),
   \quad
   \text{and}
   \quad
   \vec{\xi}^{(j)}(0,t)=P^\J,\quad
   t>0, \ j=1,2,3,
  \end{aligned}
  \right.
\end{equation}
where $\xi^\J(\cdot,t): [0,1] \to \mathbb{R}^2$ is a parametrization of
$\Gamma_t^\J$ for $t\ge 0$, $V^\J_t$ is the normal velocity of
$\Gamma^\J_t$, $\kappa^\J_t$ is the curvature of $\Gamma^\J_t$, $\mu,
\gamma, \eta > 0$ are constants, $\tau_t^{(k)}$ is the unit tangent
vector on $\Gamma_t^{(k)}$ and $P^\J$ is a fixed point.  For the
sake of notational simplicity, we denote
$\Delta^{(4)}\alpha=\Delta^{(1)}\alpha$. From \eqref{eq:1.2}, we have
energy dissipation which took a form as presented below;
\begin{equation*}
 \frac{dE}{dt}
  =
  -
  \sum_{j=1}^3
  \left(
  \frac{1}{\mu}
  \int_{\Gamma_t^{(j)}}
  |V_t^{(j)}|^2 ds
  +
  \frac{1}{\gamma}
  \left|
   \frac{d\alpha^{(j)}}{dt}(t)
   \right|^2
  \right)
  -
  \frac{1}{\eta}
  \left|
   \frac{d\vec{a}}{dt}
   \right|^2,
\end{equation*}
where $s$ is the arc-length parameter of $\Gamma_t^\J$. Note that the
third equation of \eqref{eq:1.2} is a kind of dynamic boundary condition
for the differential equations. When we take the relaxation limit
$\eta\rightarrow\infty$, the third equation turns into a force balance
condition, known as the Herring condition, at the triple junctions.
More in-depth discussion and complete details of the derivation of the
model \eqref{eq:1.2} can be found in the recent paper by Epshteyn, Liu
and the second author \cite[Section 2]{MR4263432}.

In \cite{MR4283537,MR4263432}, they relaxed the curvature effect by taking
the limit $\mu\rightarrow\infty$ and derived an ODE system, and studied
wellposedness, long-time asymptotics, and numerical analysis of the
system. In \cite{arXiv:2106.14249}, they considered ensembles of triple
junctions and misorientations (without the curvature effect), and they
used white noise to describe interactions among the grain boundaries and
the triple junctions in a grain boundary network, including modeling of
critical/disappearance events, e.g., grain disappearance, facet/grain
boundary disappearance, facet interchange and splitting of unstable
junctions. However, the well-posedness and long-time asymptotics of the
system with the curvature effect is not still well-known.

On the other hand, the second and third authors \cite{MR4292952}
considered the well-posedness and long-time asymptotics of a curve
shortening equation related to the grain boundary motion
\eqref{eq:1.2}. The curve shortening equation was derived from the
principle of maximum dissipation of the energy
$\sigma(\Delta\alpha)|\Gamma_t|$ with the periodic boundary
condition. Since the misorientation is a state variable of the energy
and depends on time, the associated curve shortening equation has
time-dependent mobility. However, the interaction among the grain
boundaries at the triple junction, even though the limiting case
$\eta\rightarrow\infty$, is not well-known since the prior work studied
only one grain boundary.

In this paper, we consider \eqref{eq:1.2} in the case
$\eta\rightarrow\infty$, namely, curvature flow of networks with time
dependent mobility governed by the evolving lattice misorientations, and
with the Herring condition at the triple junction.  
We thus consider a motion of connected curves $\Gamma^\J_t \subset \mathbb{R}^2 (j=1,2,3)$ and anisotropy parameters $\alpha^\J(t) \in \mathbb{R}$ governed by
\begin{align}
&V^{\J}_t = \sigma(\Delta^\J \alpha(t)) \kappa^\J_t \quad \text{on} \ \Gamma_t^\J,  \label{eq-curve}\\
&\pt \alpha^\J(t) = -\gamma \{ \pa \sigma (\Delta^\JJ \alpha(t)) L^\JJ(t) - \pa \sigma (\Delta^\J \alpha(t)) L^\J(t) \}, \label{eq-alpha} \\
&\vec{a}(t)=\vec{\xi}^{(1)}(1,t)=\vec{\xi}^{(2)}(1,t)=\vec{\xi}^{(3)}(1,t), \label{bc-concurrency}\\
&\sum_{k=1}^3 \sigma(\Delta^\K \alpha(t)) \tau^\K_t = 0 \quad \text{at}\ \vec{a}, \label{bc-angle}\\
&\vec{\xi}^{(j)}(0,t)=P^\J. \label{bc-boundary}
\end{align}
Here, we let $\mu$ in \eqref{eq:1.2} be $1$ without loss of generality.
Throughout the paper, a unit normal vector $\nu^\J_t$ of $\Gamma^\J_t$ is chosen as the $90$ degree rotation of $\tau^\J_t$ counter-clockwise and $V^\J$ and $\kappa^\J_t$ are positive if the normal velocity vector and the curvature vector face the same direction of $\nu^\J_t$, respectively. 
Furthermore, the curves $\Gamma^\J_t$ are numbered counter-clockwise around the junction point. 
According to an argument as in \cite{BR}, the Herring condition \eqref{bc-angle} implies that 
\[ \frac{\sigma(\Delta^{(1)} \alpha(t))}{\sin \theta^{(2)}} = \frac{\sigma(\Delta^{(2)} \alpha(t))}{\sin \theta^{(3)}} = \frac{\sigma(\Delta^{(3)} \alpha(t))}{\sin \theta^{(1)}} \]
by letting $\theta^\J(t) \in (0,2\pi)$ be the angle between $\Gamma^\J_t$ and $\Gamma^\JJ_t$. 
Therefore, the angles at the junction changes with time evolution in our setting, which is an especially different behavior from the classical curvature flow, and they are always $120$ agree if $\sigma$ is constant. 
Note that, because of the time dependence of the misorientation parameters, the equilibriums are not unique in general (see Remark \ref{rk:equilibrium}). 
We thus intend, to study the existence and, for a convex $\sigma$ as a first step to analyze the geometric flow, long-time asymptotics of the geometric flow governed by the system \eqref{eq-curve}--\eqref{bc-boundary}. 

For the classical curvature flow of networks, in other word, when $\sigma$ is a constant, the well-poseness was studied by Bronzard-Reitich \cite{BR} for initial triod of class $C^{2+\beta}$ and the theory was developed by Mantegazza-Novaga-Tortorelli \cite{MR2075985} for initial network with four or more grains of same class. 
The regularity assumption on the initial network to obtain the uniqueness was improved by \cite{GMP} in Sobolev classes, which implies, in particular, the flow of networks of class $H^2$ is unique. 
In the all articles, they first re-formulate the geometric flow to a system of the parametrization $\xi^\J$ under a restriction of choice of variable $x$, since the tangent velocity will be changed depending on the change of variables while the normal velocity always coincides with the curvature, and apply a classical existence theory for the system. 
Mantegazza-Novaga-Tortorelli \cite{MR2075985} also proved that the flow become smooth immediately by deriving energy type estimates for the derivatives of the curvatures. 
In this argument, similar estimates for the tangent velocity is required since the regularity of the flow can be discussed from the smoothness of $\xi^\J$ and the system of $\xi^\J$ depends on not only the normal velocity but also the tangent velocity. 
We note that the smoothing effect was proved also in \cite{GMP} introducing a linearized operator around the unique solution and applying functional analysis. 
In this paper, we re-formulate from the geometric flow \eqref{eq-curve}--\eqref{bc-boundary} to a system of an angle $\Theta^\J$ of the tangent vector (see Remark \ref{rmk:construction-flow} for the relationship between $\Theta^\J$ and $\xi^\J$), the orientation $\alpha^\J$ and the length of the curves $L^\J$. 
Since the curvature coincides with the derivative of $\Theta^\J$ with respect to the arc-length, the re-formulation make it easier to discuss the relation between the regularity and also the asymptotics of $\Theta^\J$ and the curvature, and also we can skip to analyze the asymptotics of the tangent velocity when we modify the energy type estimates as in \cite{MR2075985} to discuss the convergence of the curvatures. 

First present result is the short time well-posedness for initial smooth triod under the following assumption. 
\begin{itemize}
\item[(A1)] $\gamma>0, \sigma \in C^{\infty}(\mathbb{R}), \sigma(\alpha) > 0$ for any $\alpha \in \mathbb{R}$. 
\end{itemize}
We apply a classical existence theory for systems as in Lady\v{z}enskaja-Solonnikov-Ural'ceva's book \cite{LSU} to a linearized system of $\Theta^\J$, $\alpha^\J$ and $L^\J$ almost in line with the arguments in \cite{BR, MR2075985}. 
The compatibility condition for the initial curves and the complementing condition for the boundary conditions (see Section \ref{subsec:linear} for the details of the definitions) will be required to apply the theory, and thus we obtain the well-posedness result as follows. 

\begin{thm}\label{thm:exists-flow}
Assume (A1). 
Let $k \in \mathbb{N}$ with $k \ge 3$ and $\beta \in (0,1)$. 
Assume that $\Gamma_0^\J$ is a $C^{k+\beta}$ curve with respect to the arc-length and connect the fixed point $P^\J$ and a junction point $\vec{a}(0)$ for $j \in \{1,2,3\}$. 
Assume also each angle between $\Gamma_0^\JJ$ and $\Gamma_0^\J$ at the junction point is less than $\pi$. 
Further assume that the family of $\{\Gamma_0^\J\}_{j \in \{1,2,3\}}$ and $\{\alpha_0^\J\}_{j \in \{1,2,3\}} \subset \mathbb{R}^3$ satisfies the compatibility condition of order $k$ for the geometric flow \eqref{eq-curve}--\eqref{bc-boundary}. 
Then, there exists $T>0$ such that a geometric flow governed by \eqref{eq-curve}--\eqref{bc-boundary} with initial datum $\{\Gamma_0^\J\}_{j \in \{1,2,3\}}$ and $\{\alpha_0^\J\}_{j \in \{1,2,3\}}$, which at every time the triod is of class $C^{k+\beta}$ with respect to the arc-length, uniquely exists until the time $T$. 

Furthermore, if $\Gamma_0^\J$ is smooth for $j \in \{1,2,3\}$ and the family of $\{\Gamma_0^\J\}_{j \in \{1,2,3\}}$ and $\{\alpha_0^\J\}_{j \in \{1,2,3\}}$ satisfies the compatibility condition of any order for the geometric flow \eqref{eq-curve}--\eqref{bc-boundary}, then a smooth geometric flow governed by \eqref{eq-curve}--\eqref{bc-boundary} with initial datum $\{\Gamma_0^\J\}_{j \in \{1,2,3\}}$ and $\{\alpha_0^\J\}_{j \in \{1,2,3\}}$ uniquely exists until some time $T'>0$. 
\end{thm}

Note that the compatibility condition at least require that the initial triod and initial orientation parameters satisfy the Herring condition \eqref{bc-angle} at $t=0$ and, the initial angle condition at the junction point ensures that the linearized system satisfies the complimenting condition. 
The more detailed regularity of the geometric flow will be stated at the level of parametrization as in Proposition \ref{prop:short-time-system}. 
The extension of the well-posedness result as in \cite{GMP} is a future work to obtain the uniqueness of the geometric flow with initial triod of lower regularity classes. 

We now focus on the asymptotics of the geometric flow of triods. 
For the classical curvature flow, Magni, Mantegazza and Novaga \cite{MR3495423} proved the uniformly $L^2$-boundedness of the curvatures if the length of each curves is uniformly positive, combining blow-up argument and Huisken's monotonicity formula \cite{MR1030675}. 
This boundedness implies the convergence of the flow of triods to the Steiner triod connecting three endpoints $P^\J$ in the $C^\infty$ topology by means of the energy type estimates in \cite{MR2075985}. 
One of the key properties in this stability result is that the Steiner triod is a unique minimizer of \eqref{eq:1.1} with constant $\sigma$ if the following condition (A2) holds. 
\begin{itemize}
\item[(A2)] Each interior angle of the triangle generated by the fixed points $\{P^\J\}_{j \in \{1,2,3\}}$ is less than $2\pi/3$.
\end{itemize}
In our problem, as a first step to analyze the asymptotics, we further assume the following condition (A3), which yields the convexity of $\sigma$. 
\begin{itemize}
\item[(A3)] $\sigma$ satisfies $\partial_\alpha \sigma(0) = 0$, $\partial_\alpha^2 \sigma(\alpha) > 0$ for any $\alpha \in \mathbb{R}$. 
\end{itemize}
The equilibrium of our problem is then a family of the unique Steiner triod and orientations with $0$-misorientation. 
Since the monotonicity formula can not be easily extended to our problem because of the differently development of the surface tensions $\sigma(\Delta^\J \alpha)$, while the monotonicity formula for the flow of just one curve was proved in \cite{MR4292952} by applying a time-rescaling along with the development of the surface tension, we extend the exponential $L^2$-decay estimate of the curvatures as in \cite{GKS} for the classical curvature flow of triods with Neumann boundary condition and modify the dissipation estimate of the misorientation as in \cite{MR4263432} for the system \eqref{eq:1.2} with relaxed the curvature effect to obtain the exponential decay of the misorientations in our problem under a closeness of initial datum to one of the equilibriums. 
As a result, extending the energy type estimates as in \cite{MR2075985} for our problem, we obtain the smoothing effect, the global existence theory and the asymptotic result as follows.

\begin{thm}\label{thm:global-asymptotic}
Let $\gamma, \sigma, P^\J (j=1,2,3)$ satisfy the assumptions (A1)--(A3). 
Assume that each initial curve $\Gamma^\J$ is of class $H^2$ and connect the fixed point $P^\J$ and a junction point $\vec{a}(0)$ for $j \in \{1,2,3\}$. 
Further assume the family of the initial triod and the initial orientations $\{\alpha^\J_0\}_{j \in \{1,2,3\}}$ satisfies the Herring condition \eqref{bc-angle} at $t=0$. 
Then, there exist $m > 0$ and $\varepsilon > 0$ such that a geometric flow governed by \eqref{eq-curve}--\eqref{bc-boundary} with the initial datum $\{\Gamma_0^\J\}_{j \in \{1,2,3\}}$ and $\{\alpha_0^\J\}_{j \in \{1,2,3\}}$, which at every positive time the triod is smooth, exists globally in time if 
\begin{equation}\label{main:small-ini}
\begin{aligned}
&\sum_{j=1}^3 \sigma(\Delta^\J \alpha_0) L^\J(0) \le \sigma(0)m, \quad \sum_{j=1}^3 \left\{\left(\Delta^\J \alpha_0\right)^2 + \int_{\Gamma^\J_0} \left(\sigma(\Delta^\J \alpha_0)\right)^2 (\kappa_0^\J)^2 \; ds\right\} \le \varepsilon, 
\end{aligned}
\end{equation}
where $s$ is the arc-length for each initial curve $\Gamma^\J_0$. 
Furthermore, the triod $\cup_{j=1}^3 \Gamma^\J_t$ converges exponentially fast to the unique Steiner triod connecting the three fixed endpoints $P^\J$ in the $C^\infty$ topology and $\alpha^\J(t)$ converges exponentially fast to $(\alpha_0^{(1)} + \alpha_0^{(2)} + \alpha_0^{(3)})/3$ as $t \to \infty$ for any $j \in \{1,2,3\}$. 
\end{thm}

Here, we note that the uniqueness of the flow is unknown in our method except the case that the initial triod is of class $C^{3+\beta}$, while we can obtain the H\"{o}lder continuity (with respect to the time variable) of the moving triod at the initial time in $C^{1+\beta}$ topology for arbitrary $\beta \in (0,1/2)$ (see also Corollary \ref{cor:hi-estimate-theta}).

We also refer to other works related to our problem. 
For the classical curvature flow of triods with a homogenous Neumann boundary in bounded domain, the asymptotics was studied in \cite{GKS, MR1973276}. 
The former study \cite{MR1973276} show that the linear stability of stationary triod consisting of three line segment depends on the curvature of the domain at the three end-points of the triod, and the latter study \cite{GKS} gives a rigorous proof of the local exponential stability of the stationary triod when it has linear stability. 
Some arguments in the previous studies are adopted in this paper to derive the Poincar\'e type inequality, which is a key to prove the exponential $L^2$-decay of the curvatures. 
On the other hand, for the interior behavior of networks with a lot of grains, a part of grains may vanish in finite time. 
The shrinking result was first proved in \cite{MR0840401, MR0906392} for motion of Jordan curves and self similar shrinkers are constructed and classified in \cite{MR3649351, MR3804202, MR2340176} for motion of networks for example. 
Mantegazza-Novaga-Tortorelli \cite{MR2075985} also studied type I and type II singularities in addition to the existence theory and smoothing effect. 
Because of the singularities, weak solutions of the multi-phase mean curvature flow have also been studied well. 
We here refer to \cite{MR0485012, MR3612327, MR4184584} for the existence theory of a measure theoretic solution, which is so-called the Brakke flow, and \cite{MR3556529, MR4056816, MR3847750} for the construction of distributional solutions in the framework of BV functions. 
The existence of strong solutions without the Herring condition at the initial time is also important in view of the analysis of the singularities since the condition is lost when a grain vanish, and the existence theory was studied in \cite{MR3909904, LMPS}.

The rest of the paper is organized in the following way. 
We first introduce a re-formulation of the geometric flow and prove Theorem \ref{thm:exists-flow} in Section \ref{sec:local-exists}. 
The equilibriums will be studied in Section \ref{sec:equilibrium}. 
Section \ref{sec:orientation} lists some geometric properties for the computations in the sequel and the exponential decay of the misorientations is also prove in the section. 
The exponential $L^2$-decay of the curvatures is proved in Section \ref{sec:L2-estimate} and higher order estimates are derived in Section \ref{sec:H2-estimate} to continue to prove Theorem \ref{thm:global-asymptotic}.

%% file: references.tex
\providecommand{\bysame}{\leavevmode\hbox to3em{\hrulefill}\thinspace}
\providecommand{\MR}{\relax\ifhmode\unskip\space\fi MR }
\providecommand{\MRhref}[2]{%
  \href{http://www.ams.org/mathscinet-getitem?mr=#1}{#2}
}
\providecommand{\href}[2]{#2}